\documentclass{editor}
\usepackage{amsmath}
\usepackage{amssymb, amsfonts}
\usepackage{amsthm}
\usepackage{hyperref}
\usepackage{xfp}

\usepackage{mathtools}
\usepackage[enableskew]{youngtab}

\usepackage{ytableau}
\usepackage[shortlabels]{enumitem}
\usepackage{color} 

\usepackage{tikz} % drawing for math people.
\usetikzlibrary{decorations.markings}
\usepackage[utf8]{inputenc}

\usepackage{tikz-cd}%%%%% call using tikzcd
\usetikzlibrary{calc,math}
\usepackage{ifthen}
\usepackage{xparse}

\usetikzlibrary{arrows,decorations.pathmorphing,backgrounds,positioning,fit,petri,arrows.meta,bending} % extra good stuff.

\tikzset{help lines/.style={step=#1cm,thin, color=gray},
help lines/.default=1} % draws a grid spaced #1 cm

\newcommand{\Hom}{\mathrm{Hom}}

\newcommand{\ul}{\underline} % Here I defined the short cut ``\ul'' to underline text that follows. For example, \ul word will give ``word with the first letter underlined'' \ul{word} will underline the whole word.

\NewDocumentCommand{\ball}{m m m O{white}}{
\tikzmath{
    \x1 = \fpeval{#1 + (#3 / 2)};
    \x2 = \fpeval{#1 - (#3 / 2)};
    \y1 = \fpeval{#2 + (#3 / 2)};
    \y2 = \fpeval{#2 - (#3 / 2)};
    \r1 = \fpeval{#3 * 0.72};
}
\draw[thick, fill=#4] (#1,#2) circle [radius=\r1];
}

\NewDocumentCommand{\bball}{m m m O{black} O{black} O{0}}{
\tikzmath{
    \x1 = \fpeval{#1 + (#3 / 2)};
    \x2 = \fpeval{#1 - (#3 / 2)};
    \y1 = \fpeval{#2 + (#3 / 2)};
    \y2 = \fpeval{#2 - (#3 / 2)};
    \r1 = \fpeval{#3 * 0.72};
}
\draw[thick] (#1,#2) circle [radius=\r1];
\draw[thick, #4] (\x1, \y1) arc (-20:-160:\fpeval{#3 * 0.56});
\draw[thick, #5] (\x1, \y2) arc (20:160:\fpeval{#3 * 0.56});

\draw[thick, #4, dashed, opacity=#6] (\x1, \y1) -- (\x1 + \r1, \y1 + \r1);
\draw[thick, #4, dashed, opacity=#6] (\x2, \y1) -- (\x2 - \r1, \y1 + \r1);

\draw[thick, #5, dashed, opacity=#6] (\x1, \y2) -- (\x1 + \r1, \y2 - \r1);
\draw[thick, #5, dashed, opacity=#6] (\x2, \y2) -- (\x2 - \r1, \y2 - \r1);
}

\NewDocumentCommand{\cball}{m m m O{black} O{black} O{0}}{
\tikzmath{
    \x1 = \fpeval{#1 + (#3 / 2)};
    \x2 = \fpeval{#1 - (#3 / 2)};
    \y1 = \fpeval{#2 + (#3 / 2)};
    \y2 = \fpeval{#2 - (#3 / 2)};
    \r1 = \fpeval{#3 * 0.72};
}
\draw[thick] (#1,#2) circle [radius=\fpeval{#3 * 0.72}];
\draw[thick, #4] (\x1, \y2) -- (\x2, \y1);
\draw[thick, #5] (\x1, \y1) -- (\x2, \y2);

\draw[thick, #5, dashed, opacity=#6] (\x1, \y1) -- (\x1 + \r1, \y1 + \r1);
\draw[thick, #4, dashed, opacity=#6] (\x2, \y1) -- (\x2 - \r1, \y1 + \r1);

\draw[thick, #4, dashed, opacity=#6] (\x1, \y2) -- (\x1 + \r1, \y2 - \r1);
\draw[thick, #5, dashed, opacity=#6] (\x2, \y2) -- (\x2 - \r1, \y2 - \r1);
}

\newtheorem{thm}{Theorem}[chapter]
\newtheorem{lem}[thm]{Lemma}

\newtheorem{Ex}[thm]{Example}

\theoremstyle{definition}
\newtheorem{define}[thm]{Definition}

\newtheorem{definition}{Definition}

\theoremstyle{remark}

\newtheorem{remark}[thm]{Remark}
 
\numberwithin{equation}{section}

\newcommand{\comment}[1]{}

\newcommand{\AL}{{\draw[very thick, color=blue!50!white] (1,1)--(2,0);}}
\newcommand{\AR}{{\draw[very thick, color=red!50!white] (3,1)--(2,0)
;}}
\newcommand{\BL}{{\draw[very thick, color=blue!50!white] (3,1)--(4,0);}}
\newcommand{\BR}{{\draw[very thick, color=red!50!white] (5,1)--(4,0);}}
\newcommand{\CL}{{\draw[very thick, color=blue!50!white] (2,2)--(3,1);}}

\newcommand{\CR}{{\draw[very thick, color=red!50!white] (4,2)--(3,1);}}

\begin{document}
\frontmatter

\title{Binary trees using the bookshelf and baseball constructions}

% Remove or comment out any unused author tags.
\editor{Kiyoshi Igusa}% first editor
%\address{}
%\curraddr{}
\email{igusa@brandeis.edu}

% Ruiyang Hu, Liang Jinlong, Bradley Kaplan, Zhaonan Li, Serra Pelin, Michael Richard, Yicheng Tao, Max Weinstein, Kiyoshi Igusa

%\date{\today}

% 05C05  	Trees
%16G20: Representations of quivers and partially ordered sets (2000=2010=2020)

%\begin{abstract}
%This is a largely expository paper in which we discuss various sets having a Catalan number of objects and some well-known bijections between these sets presented in a new and hopefully interesting way. We call these concepts ``bookshelf'' and ``baseball'' constructions. No knowledge of these topics is assumed. These are the final student papers for a course at Brandeis University.
%\end{abstract}

% 49 pages (several blank), 12 figures.

\maketitle

\begin{contentslist}
%  Insert contents entries here, using this template.  If a title has
%  multiple authors, use a separate \author tag for each author.
%  Leave the \page number field empty, but do not remove the tag; the
%  \page tag is required for formatting.  Numbers will be filled in
%  at the AMS.
\contitem
\title{Dyck paths and their bijections}
\author{Bradley Kaplan}
\page{1}
\contitem
\title{Relationship between binary trees and Young diagrams: a story of bookshelves}
\author{Liang Jinlong}
\author{Serra Pelin}
\page{5}
\contitem
\title{A bijection between binary trees and torsion classes}
\author{Yicheng Tao}
\page{22}
\contitem
\title{Torsion classes and binary trees}
\author{Max Weinstein}
\page{26}
\contitem
\title{Young diagrams with gaps and torsion classes}
\author{Ruiyang Hu}
\page{30}
\contitem
\title{213-avoiding permutations and binary trees}
\author{Zhaonan Li}
\author{Michael Richard}
\page{40}

\end{contentslist}

%\tableofcontents

%\input{aOutline.tex}

\setcounter{section}{1}

\chapter*{Preface}This volume includes the results of a Fall 2019 undergraduate course at Brandeis University called ``Introduction to Mathematical Research''. We studied bijections between various well-known sets having a Catalan number of objects, such as binary trees, Young diagrams, torsion classes and 213-avoiding permutations.

The $n$-th \ul{Catalan number} is $C_n=\frac1{n+1}\binom{2n}n$. There are a large number of sets having a Catalan number of objects \cite{ref1}. One of these is the set of \ul{magma} which are patterns of association such as
\[
	((\bullet \bullet)\bullet)(\bullet \bullet)
\]
There are $C_n$ magma of size $n+1$. For example, there are $C_2=2$ ways to associate three items: $(\bullet\bullet)\bullet$ and $\bullet(\bullet\bullet)$. These form what is known as the \ul{Tamari lattice} which is a partial ordering on the collection of magma where the ``covering relation'' (See Definition A1) is given by shifting one set of parentheses to the right. For example
\[
	(\bullet (\bullet\bullet))(\bullet \bullet)\text{ covers }((\bullet \bullet)\bullet)(\bullet \bullet)
\]
A \ul{maximal chain} in the Tamari lattice is defined to be a maximal sequence of covering relations for magma of a fixed size. It is a famous open problem to give the formula for the number of such chains.

Recently, Nelson \cite{N} gave a recursive formula for the number of maximal chains in the Tamari lattice. To do this he used the correspondence between certain \ul{Young diagrams} and \ul{Dyck paths} and the description of the covering relation given in \cite{BB}.

The Tamari lattice is also isomorphic to the lattice of torsion classes for the path algebra of the quiver of type $A_n$. (See the papers by Tao and Weinstein for definitions.) Recently, Barnard, Carroll and Zhu \cite{BCZ} have given a description of the covering relation for torsion classes for arbitrary finite dimensional algebras over any field. These papers grew out of our attempt to understand the relationship between these ideas.

%https://www.overleaf.com/project/5fb3f737c40b291947be87aa
In Chapter 1, Bradley Kaplan describes how certain Young diagrams correspond to Dyck paths. In Chapter 2, Liang Jinjong (Olly) and Serra Pelin use what they call the ``\ul{stacked bookshelf}'' construction to gives an explicit bijection between binary trees and the Young diagrams considered in Chapter 1.

Next, we recall the definition of {Dyck path} and discuss the well-known bijections between Dyck paths and binary trees and Nelson's bijection between Dyck paths and Young diagrams. Nelson's construction is simply to ``fill in'' the space above the Dyck path. Olly and Serra show that these two bijections together give the same bijection as their ``bookshelf'' construction.

Torsion pairs (\ul{torsion classes} and \ul{torsion-free classes}) are discussed next. Yicheng Tao and Max Weinstein show the correspondence between torsion pairs and binary trees using bookshelves without ``stacking''. Ruiyang Hu (Sam) gives a direct bijection between torsion pairs and Young diagrams by using \ul{Young diagrams with gaps} which can be pushed together to give standard Young diagrams without gaps.

In the last chapter, Zhaonan Li (Leo) and Michael Richard discuss pattern avoiding permutations. By a 213-avoiding permutation we mean a permutation $\sigma$ of $n$ letters so that there do not exist $1\le i<j<k\le n$ for which $\sigma(j)<\sigma(i)<\sigma(k)$. Leo and Mike use the \ul{baseball} construction to give an explicit bijection between torsion classes and 213-avoiding permutations.

We conclude with several diagrams of representations of the Tamari lattice in which the nodes are the various constructions discussed in the contents of these papers.

%We conclude with some comments on the relationship between these constructions and the work of Nelson and others.

This chart shows how these papers are related.
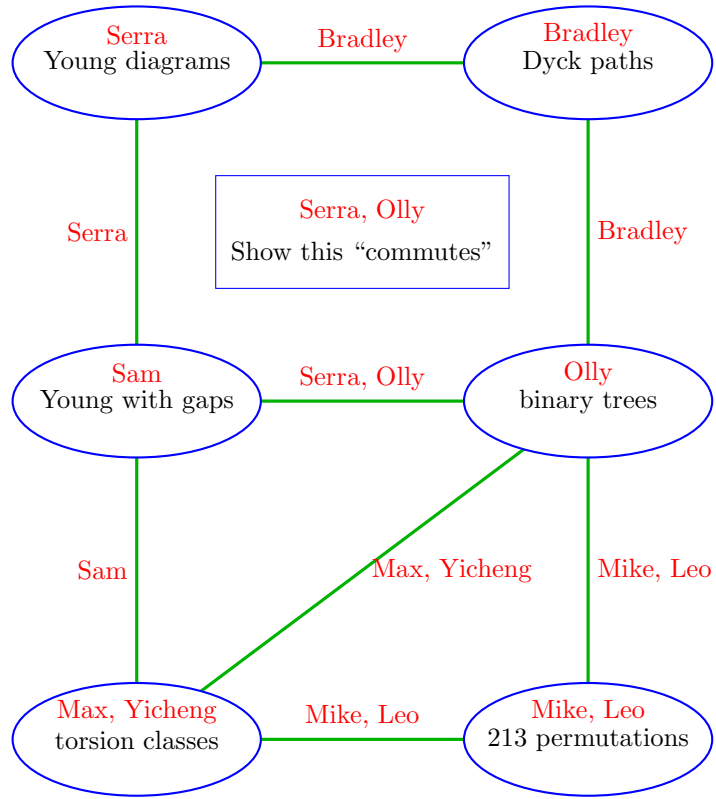
\begin{center}
\begin{tikzpicture}[scale=1.5]
%\draw[help lines=1,thick] (-1,-1) grid (4,5);
%\foreach \x in {-5,...,5}\draw (\x,0) node{\x};\foreach \y in {-5,...,5}\draw (0,\y) node{\y};
%\draw[thick,color=blue] (0,1) ellipse [x radius=2.8cm,y radius=2.1cm];
\coordinate (A1) at (0,6);
\coordinate (A2) at (4,6);
\coordinate (A0) at (2,6.2);

\coordinate (AB1) at (0,4.5);
\coordinate (AB2) at (4,4.5);

\coordinate (B1) at (0,3);
\coordinate (B2) at (4,3);
\coordinate (B0) at (2,3.2);

\coordinate (BC1) at (0,1.5);
\coordinate (BC2) at (4,1.5);

\coordinate (BC0) at (2,1.5);
\coordinate (BC0p) at (4,1.5);

\coordinate (C1) at (0,0);
\coordinate (C2) at (4,0);
\coordinate (C0) at (2,0.2);

\coordinate (A1p) at (0,6.25);
\coordinate (A2p) at (4,6.25);
\coordinate (B1p) at (0,3.25);
\coordinate (B2p) at (4,3.25);
\coordinate (C1p) at (0,0.25);
\coordinate (C2p) at (4,0.25);
% \coordinate (C3p) at (4.75,1.5);
\foreach \x/\xtext in {A1,A2,B1,B2,C1,C2}
\draw (\x) node{$\xtext$};
\draw[very thick,color=green!70!black]
(A1)--(A2)--(B2)--(B1)--cycle (B1)--(C1)--(B2)--(C2)--(C1);
\foreach \x/\xtext in {A1,A2,B1,B2,C1,C2}
\draw[fill,color=white] (\x) ellipse[x radius=11mm, y radius=5mm];
\foreach \x/\xtext in {A1,A2,B1,B2,C1,C2}
\draw[thick,color=blue] (\x) ellipse[x radius=11mm, y radius=5mm];
\draw (A1) node{Young diagrams};
\draw[color=red] (A1p) node{Serra};
\draw (A2) node{Dyck paths};
\draw[color=red] (A2p) node{Bradley};
\draw (B1) node{Young with gaps};
\draw[color=red] (B1p) node{Sam};
\draw (B2) node{binary trees};
\draw[color=red] (B2p) node{Olly};
\draw (C1) node{torsion classes};
\draw[color=red] (C1p) node{Max, Yicheng};
\draw[color=red] (BC0p) node[right]{Mike, Leo};
\draw (C2) node{213 permutations};
\draw[color=red] (C2p) node{Mike, Leo};
% \draw[color=red] (C3p) node{Mike, Leo};
\draw[color=red] (C0) node{Mike, Leo};
\draw[color=red] (BC1) node[left]{Sam};
\draw[color=red] (B0) node{Serra, Olly};
\draw[color=red] (A0) node{Bradley};
\draw[color=red] (AB2) node[right]{Bradley};
\draw[color=red] (AB1) node[left]{Serra};

\draw[color=red] (BC0) node[right]{Max, Yicheng};

\draw[color=blue] (.7,4)rectangle (3.3,5);

\draw[color=red] (2,4.5) node[above]{Serra, Olly};
\draw (2,4.5) node[below]{Show this ``commutes''};
\end{tikzpicture}
\end{center}

\aufm{Kiyoshi Igusa, 2021}

% Here are links to your contributions.
%\newpage

\mainmatter

%\setcounter{page}1

%We start with Young diagrams defined in xxxx below. These are (short description here). The \ul{stacked bookshelf} construction gives a bijection between these two sets. This is described in xxx below. Roughly, yyy.

% Begin each "chapter" file with \chapter*(<title>)
% Bradley's contribution 
%\newpage

\chapter*{Dyck paths and their bijections}

\centerline{Bradley Kaplan}

\begin{abstract}
The purpose of this chapter is to define, describe, and construct examples of Dyck paths, to give the history of Dyck paths, to give the bijection between Dyck paths and binary trees, and to give the bijection between Dyck paths and Young diagrams. Following the idea of Nelson. \cite{N}
\end{abstract}

\smallskip

%\maketitle

\subsubsection{Introduction} We will review the definition of a Dyck path, give some of the history of Dyck paths, and describe and construct examples of Dyck paths.

In the second section we will show, using the description of a binary tree and the definition of a Dyck path, that there is a bijection between binary trees and Dyck paths. In the third section we will give an explicit bijection between Young diagrams and Dyck paths, which will be used in other parts of this paper (written by others).

\subsection{Definition of Dyck path}

\begin{define}\label{1}
We recall that a \emph{Dyck path} is a series of equal length steps, going up and right, that form a staircase walk from (0, 0) to (\emph{n, n}) that will lie strictly above, or touching, the diagonal \emph{x = y}.  \cite{GroupTheoryBook}
\end{define}
Now we will discuss the history of Dyck paths.
We review: the Catalan number is $C_{n}$.

$C_{n}=(1\div(n+1)){2n \choose n}$

In 1887, D\'esir\'e Andr\'e found that the Catalan number is the number of Dyck paths of length $2n$ that can be constructed. \cite{Desire}
\bigskip

The following are some examples of Dyck paths. Note that the line going from left to right must go up 1 unit (represented here by two blocks) before it goes to the right 1 unit. In later sections we will show that binary trees and Young diagrams can be created from these Dyck paths, and Dyck paths can be created from Young diagrams and binary trees.
\bigskip\bigskip\bigskip

\newcommand\dyckpath[5]{
  \begin{scope}[local bounding box=#4]
    \fill[red!45!white]  (#1) rectangle +(#2,#2);
    \fill[white] (#1) foreach \dir in {#3}{-- ++(\dir*90:1)} |- (#1);
    \path[fill] (#1) foreach \i [count=\j] in {1,...,#5} { +(\i,0) node[anchor=north]{\j} \ifnum\i>#2 circle (1pt) \fi};
    \draw[help lines] (#1) grid +(#2,#2);
    \draw[line width=2pt] (#1) foreach \dir in {#3}{ -- ++(\dir*90:1)};
  \end{scope}
}
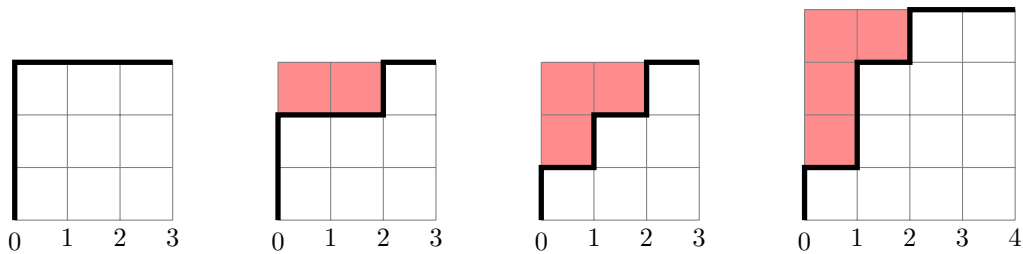
\begin{figure}
\begin{tikzpicture}[scale=0.7]
  \draw(0,-.4)node{0};
  \draw(5,-.4)node{0};
  \draw(10,-.4)node{0};
  \draw(15,-.4)node{0};
  \dyckpath{0,0}{3}{1,1,1,0,0,0}{dyck1}{3};
  \dyckpath{5,0}{3}{1,1,0,0,1,0}{dyck2}{3};
  \dyckpath{10,0}{3}{1,0,1,0,1,0}{dyck3}{3};
   \dyckpath{15,0}{4}{1,0,1,1,0,1,0,0}{dyck4}{4};

\end{tikzpicture}
\caption{Some examples of Dyck paths}
\end{figure}

\subsection{Bijection between Dyck and binary trees}\label{section on YD}

Now we will demonstrate that the above Dyck paths can be translated into binary trees and that binary trees can be translated into Dyck paths, thus giving a bijection between the two.
\bigskip

\begin{remark}\label{OP}
We recall that a \emph{binary tree} is a diagram that starts with one node at the top (called the root), where each node has two children, except for the leaves which have no children and are located at the bottom of the diagram. Also, lines are drawn between parents and children. Given a parent, one of its children is located diagonally down and to the right of the parent while the other child is located down and to the left of the parent.
\end{remark}

As shown by Olivier Bernardi, if we are given any binary tree, we can construct a Dyck path from it. We label any node that has two children (a parent) as a "left" equivalent and any leaf (a child that does not have any children) as a "down" equivalent. We start at the top of the binary tree and start reading the equivalents going strictly diagonally down and to the right. When we cannot read in this direction we read strictly to the left, except when we meet a line segment drawn strictly diagonally from the top left direction to the bottom right direction, in which case we jump to the top of that line segment. We start drawing at the point (n,n) (where the Dyck path would finish if it was normally drawn) and when we read a "left" label we draw a line segment of a single unit's length going to the left. When we read a "down" label we draw a line segment of a single unit's length going down the page. We do not read the last node's label. An example of this labeling of a binary tree and the construction of a Dyck path from it is shown in the rightmost pair of a binary tree and a Dyck path in {Figure 2}, where "left" is replaced with "L" and "down" is replaced with "D".

Similarly, if we are given any Dyck path, we can construct a binary tree from it. We start tracing the Dyck path at the point (n,n) and follow the line to (0,0). If the line in the Dyck path goes one unit to the left we, starting at the parent node of the binary tree, draw a line segment of a single unit's length going down and to the right with a node at the end of it (that can become a parent or an leaf). If the line in the Dyck path goes one unit down the page we create an leaf in the binary tree (with the last node created becoming an leaf). If the leaf cannot be created without drawing an additional line in the binary tree (the leaf would have no parent), a line is drawn between the leaf and the lowest parent node that does not yet have two children. If the line in the Dyck path goes one unit to the left after the first leaf is created, we draw a line segment of a single unit's length going diagonally down and to the left with a node at the end of it starting at the lowest node that is not an leaf that does not have two children and if the line in the Dyck path goes multiple units to the left we draw line segments of a single unit's length going diagonally down and to the right with a node at the end of it starting at the node created by the initial line segment of one unit's length going to the left in the Dyck path, one for each additional unit the line in the Dyck path goes to the left beyond the initial unit. All nodes created after the first leaf is created must become parents (they cannot become leaves). Next, we add a leftmost leaf. Finally, we need to extend all lines in the binary tree such that all leaves are aligned horizontally. We do this by tracing starting from the root. Of the root's two children, if one or more of one of the children's lowest leaves are lower than the other child's lowest leaves, we extend the segment connecting the parent to the root such that both of the children's lowest leaves are aligned horizontally. We repeat this process with every parent (with the parent becoming the "root" in this process), starting with the highest parents and ending with the lowest parents.

Thus, since we can construct a binary tree from an equivalent Dyck path and do the converse, a bijection exists between Dyck paths and binary trees. 

Below are some examples of Dyck paths and their equivalent binary trees (directly below the corresponding Dyck paths):
\pagebreak

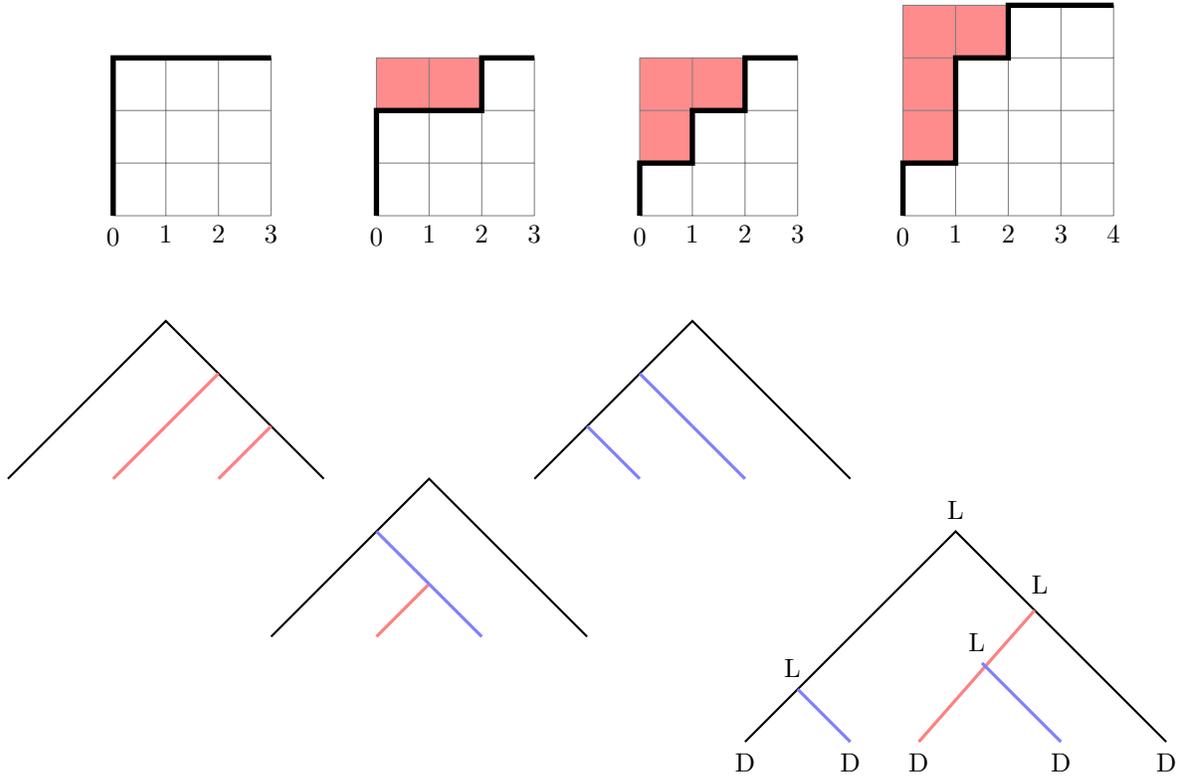
\begin{figure}
\begin{tikzpicture}[scale=0.7]
  \draw(0,-.4)node{0};
  \draw(5,-.4)node{0};
  \draw(10,-.4)node{0};
  \draw(15,-.4)node{0};
  \dyckpath{0,0}{3}{1,1,1,0,0,0}{dyck1}{3};
  \dyckpath{5,0}{3}{1,1,0,0,1,0}{dyck2}{3};
  \dyckpath{10,0}{3}{1,0,1,0,1,0}{dyck3}{3};
   \dyckpath{15,0}{4}{1,0,1,1,0,1,0,0}{dyck4}{4};

\begin{scope}[xshift=8cm,yshift=-5cm] % binary tree at middle right
\draw[thick] (0,0)--(3,3)--(6,0);
\AL
\BL
\CL
\end{scope}
\begin{scope}[xshift=3cm,yshift=-8cm] % binary tree at middle left
	\draw[thick] (0,0)--(3,3)--(6,0);
	\AR
	\BL
	\CL
\end{scope}
\begin{scope}[xshift=12cm,yshift=-10cm] % binary tree at right
	\draw(4,4.4)node{L};
	\draw(5.6,3)node{L};
	\draw(4.4,1.9)node{L};
	\draw(.9,1.4)node{L};
	\draw(2,-.4)node{D};
	\draw(0,-.4)node{D};
	\draw(3.3,-.4)node{D};
	\draw(6,-.4)node{D};
	\draw(8,-.4)node{D};
	\draw[thick] (0,0)--(4,4)--(8,0);
	\draw[very thick, color=red!50!white] (5.5,2.5)--(3.3,0);
	\draw[very thick, color=blue!50!white] (4.5,1.5)--(6,0);
	\draw[very thick, color=blue!50!white] (1,1)--(2,0);
\end{scope}
\begin{scope}[xshift=-2cm,yshift=-5cm] % binary tree at left
	\draw[thick] (0,0)--(3,3)--(6,0);
	\AR
	\BR
	\CR
\end{scope}
\end{tikzpicture}
\caption{Dyck paths and their corresponding binary trees.}
\end{figure}

\subsection{Bijection between Dyck paths and Young diagrams}\label{ss2.3: Dyck and Young}
Now we will go through the easy construction that translates Dyck paths into Young diagrams and the converse translation, thus giving a bijection between the two.

\begin{remark}\label{RP}
Recall that a \emph{Young diagram} is a picture of several connected squares where, starting from the top left, the length of each successive row or column in the Young diagram (going down and to the right) is strictly non-increasing.
\end{remark}

The conversion of Dyck paths to Young diagrams is simple. A Young diagram is constructed by filling in the empty space in the $n$ x $n$ square that the Dyck path occupies.

Similarly, given $n$ and a Young diagram $\lambda$ with $\lambda_{i}+i \leq n\sqrt{i}$ we can construct a Dyck path of size $2n$ from it. We start drawing the Dyck path from (0,0). Given the first column of the Young diagram is of length $x$, we draw a line going up of $n-x$ units and then draw a line of one unit's length to the right. Given the second column of the Young diagram is of length $y$, we draw a line going up of $x-y$ units and then draw a line of one unit's length to the right. We repeat this process until we get to the point (n,n).

Thus, since we can construct a Young diagram from an equivalent Dyck path and do the converse, a bijection exists between Dyck paths and Young diagrams. 

Below are some examples of Dyck paths and their equivalent Young diagrams (directly below the corresponding Dyck paths):
\pagebreak

\begin{figure}
\begin{tikzpicture}[scale=0.7]
  \draw(5,-.4)node{0};
  \draw(10,-.4)node{0};
  \draw(15,-.4)node{0};
  \draw(6.5,-2)node{\yng(2)};
  \draw(11.5,-2)node{\yng(2,1)};
  \draw(17,-2)node{\yng(2,1,1)};
  \dyckpath{5,0}{3}{1,1,0,0,1,0}{dyck2}{3};
  \dyckpath{10,0}{3}{1,0,1,0,1,0}{dyck3}{3};
   \dyckpath{15,0}{4}{1,0,1,1,0,1,0,0}{dyck4}{4};
\end{tikzpicture}
\caption{Dyck paths and their corresponding Young diagrams.}
\end{figure}
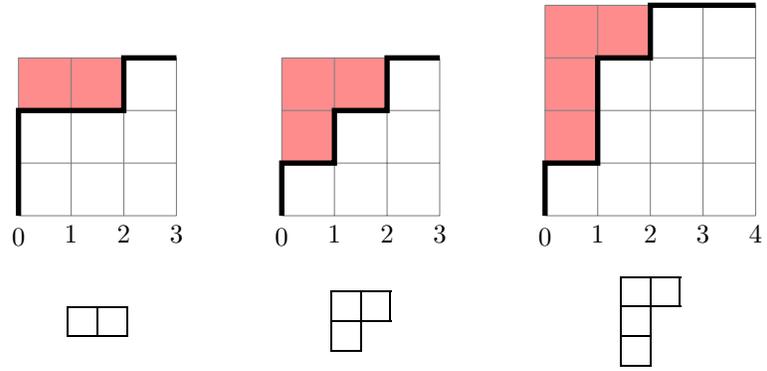

%end of section.

%\newpage

%\input{OllyAndSerra.tex} %
% Olly and Serra's contribution

\newpage

\chapter*{Relationships between Binary Trees and Young Diagrams: A Story of Bookshelves}
\label{section:OS}

\centerline{Liang Jinlong and Serra Pelin }
%\date{December 9th, 2019}
%\begin{document}
%\newtheorem{definition}{Definition}
%\newtheorem{theorem}{Theorem}

%\maketitle

\setcounter{section}2

\begin{abstract}
    In this chapter a new function between Young diagram and binary trees is defined, which we call 'the bookshelf construction'. Both the bookshelf and the inverse bookshelf mappings will be well defined, with important properties, then the bijection will be shown. Further, we will prove that the mapping between Young diagram, Dyck path and binary trees commute.
\end{abstract}
%\tableofcontents
\bigskip

\subsubsection{Introduction}
%introduce topic
In this paper, we will define a mapping from binary trees to Young diagrams, called the 'bookshelf construction' as well as prove the bijection between the 2 classes by defining the inverse of the 'bookshelf construction'. We will discuss the appropriate Young diagram transformation that allows such a mapping. Further we will prove the commutativity between binary trees, Dyck path and Young diagrams using the 'bookshelf construction' and other well known mappings.

Due to the nature of this topic we will make use of geometric examples and definition as well as algrebraic. Our paper builds on top of the recent literature by Nelson (2016) where a relationship between Young diagrams and Tamari lattices are constructed. In turn this implies a relationship between Young diagrams and binary trees, which is the main objective of our paper by using the 'bookshelf construction'.

\subsection{Young Diagram}
\begin{definition}
Young diagram: A \textit{Young diagram} is a geometric object consisting of boxes such that for $\lambda_i$ representing the number of boxes in row $i$ of the Young diagram, $\lambda_j \geq \lambda_i$ for all $i>j$ and $\lambda_{k+1}=0$ for a Young diagram with $k$ rows. Similarly for each column going from left to right.
\end{definition}
Consider the following example of a Young diagram with 5 boxes:

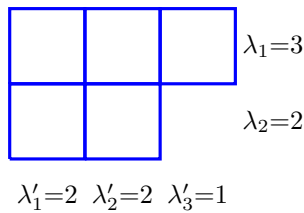
\begin{center}
\begin{tikzpicture}
\coordinate (A) at (3.5,1.5);
\coordinate (B) at (3.5,.5);
\coordinate (C) at (1.5,-.5);
\coordinate (D) at (0.5,-.5);
\coordinate (E) at (2.5,-.5);

\foreach \x/\y in{0/0, 1/0, 2/1, 0/1, 1/1}
\draw[very thick, color=blue] (\x,\y)--(\x+1,\y)--(\x+1,\y+1)--(\x,\y+1)--(\x,\y);
\draw (A) node{$\lambda_1$=3};
\draw (B) node{$\lambda_2$=2};
\draw (C) node{$\lambda_2'$=2};
\draw (D) node{$\lambda_1'$=2};
\draw (E) node{$\lambda_3'$=1};
\end{tikzpicture}
\end{center}
Notice that for the rows of the diagram, $\lambda_1=3 > \lambda_2=2$, whereas for the columns of the diagram $\lambda_1'=2 = \lambda_2'> \lambda_3'=1$, thus satisfying the properties of a Young diagram.
\subsection{Binary Tree}
In tradition, a binary tree is defined as follow:
\begin{definition}
Binary tree: A \textit{binary tree} $A$ is a planar graph which is a tree such that each node except leaves has 2 children and one parent (the root doesn't have a parent).
\end{definition}
For the purposes of this paper and the 'bookshelf' construction that will be explained in the upcoming section our definition of a binary tree is amended as follows: 
\begin{definition}
New binary tree: As opposed to the traditional definition, a \textit{new binary tree} $A'$, which is based on Definition 2, has the amendment such that every child of the binary tree lies on the same, i.e. the last level of the binary tree.\footnote{Going forward, any mention of the binary tree refers to the new binary tree as defined in Definition 3.} 
\end{definition}

Consider the example below:\\
\\
\begin{minipage}{0.5\textwidth}
\begin{center}
    \begin{tikzpicture}[scale=0.5]
    \draw[thick](1,1)--(3,3)--(4,2); 
    \draw[thick](2,2)--(3,1); 
    \draw[thick](3,1)--(4,0); 
    \draw[thick](3,1)--(2,0);
    \draw (1,1)[fill=black, thick] circle (1mm);
    \draw (2,2)[fill=black, thick] circle (1mm);
    \draw (3,3)[fill=black, thick] circle (1mm);
    \draw (4,2)[fill=black, thick] circle (1mm);
    \draw (3,1)[fill=black, thick] circle (1mm);
    \draw (2,0)[fill=black, thick] circle (1mm);
    \draw (4,0)[fill=black, thick] circle (1mm);
    \end{tikzpicture}
   
    the traditional binary tree
\end{center}
\end{minipage}
\begin{minipage}{0.5\textwidth}
\begin{center}
    \begin{tikzpicture}[scale=.5]
   \draw[thick, color=red](0,0)--(1,1);
   \draw[thick](1,1)--(3,3)--(4,2); 
   \draw[thick](2,2)--(3,1); 
   \draw[thick](3,1)--(4,0); 
   \draw[thick](3,1)--(2,0);
   \draw[thick, color=red](4,2)--(6,0);
   \draw (0,0)[fill=black, thick] circle (1mm);
   \draw (2,2)[fill=black, thick] circle (1mm);
   \draw (3,3)[fill=black, thick] circle (1mm);
   \draw (6,0)[fill=black, thick] circle (1mm);
   \draw (3,1)[fill=black, thick] circle (1mm);
   \draw (2,0)[fill=black, thick] circle (1mm);
   \draw (4,0)[fill=black, thick] circle (1mm);
   \end{tikzpicture}

   the new binary tree
\end{center}
\end{minipage}

In this example notice that the right most and the left most branches are extended (in red) such that all children of the tree are aligned.

Note that by definition any binary tree has 2 subtrees:
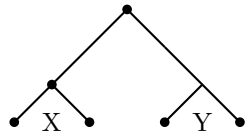
\begin{center}
    \begin{tikzpicture}[scale=.5]
   \draw[thick](0,0)--(1,1);
   \draw[thick](1,1)--(3,3)--(4,2); 
   %\draw[thick](2,2)--(3,1); 
   \draw[thick](5,1)--(4,0); 
   \draw[thick](1,1)--(2,0);
   \draw[thick](4,2)--(6,0);
   \draw (0,0)[fill=black, thick] circle (1mm);
   %\draw (2,2)[fill=black, thick] circle (1mm);
   \draw (1,1)[fill=black, thick] circle (1mm);
   \draw (6,0)[fill=black, thick] circle (1mm);
   \draw (3,3)[fill=black, thick] circle (1mm);
   %\draw (53,1)[fill=black, thick] circle (1mm);
  % \draw (1,0)[fill=black, thick] circle (1mm);
   \draw (2,0)[fill=black, thick] circle (1mm);
   \draw (4,0)[fill=black, thick] circle (1mm);
   \draw (1,0) node{X};
   \draw (5,0) node{Y};
   \end{tikzpicture}
\end{center}
where X is the left subtree and Y is the right subtree. Also, X and Y could be the empty set.
\subsection{Coordinate system for binary trees}\label{ss: coordinate system}
Define a coordinate system on a binary tree with $n$ children, or of size $n-1$, as defined above as follows:
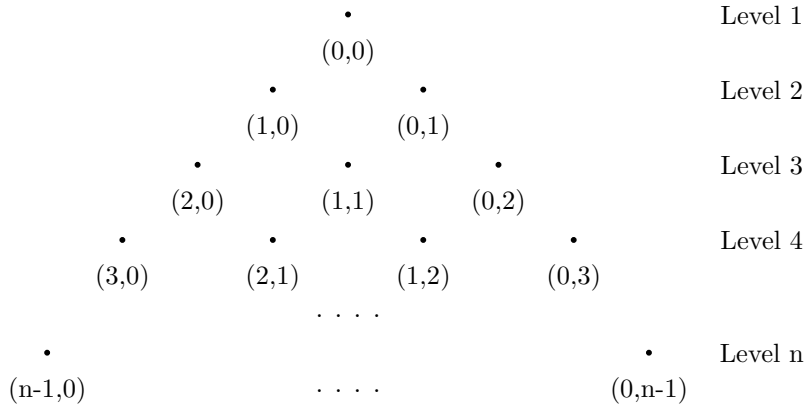
\begin{center}
    \begin{tikzpicture}[scale=1]
    \coordinate (A) at (0,0);
    \coordinate (B) at (1,-1);
    \coordinate (C) at (-1,-1);
    \coordinate (D) at (2,-2);
    \coordinate (E) at (0,-2);
    \coordinate (F) at (-2,-2);
    \coordinate (G) at (3,-3);
    \coordinate (H) at (1,-3);
    \coordinate (I) at (-1,-3);
    \coordinate (J) at (-3,-3);
    \coordinate (K) at (0,-3.5);
    \coordinate (L) at (-4,-4.5);
    \coordinate (M) at (4,-4.5);
    \coordinate (M) at (4,-4.5);
    \coordinate (N) at (0,-4.5);
    \coordinate (1) at (0,0.5);
    \coordinate (2) at (1,-0.5);
    \coordinate (3) at (-1,-0.5);
    \coordinate (4) at (2,-1.5);
    \coordinate (5) at (0,-1.5);
    \coordinate (6) at (-2,-1.5);
    \coordinate (7) at (3,-2.5);
    \coordinate (8) at (1,-2.5);
    \coordinate (9) at (-1,-2.5);
    \coordinate (10) at (-3,-2.5);
    \coordinate (11) at (-4,-4);
    \coordinate (12) at (4,-4);
    \draw (A) node{(0,0)};
    \draw (5.5,0.5) node{Level 1};
    \draw (5.5,-0.5) node{Level 2};
    \draw (5.5,-1.5) node{Level 3};
    \draw (5.5,-2.5) node{Level 4};
    \draw (5.5,-4) node{Level n};
    \draw (B) node{(0,1)};
    \draw (C) node{(1,0)};
    \draw (D) node{(0,2)};
    \draw (E) node{(1,1)};
    \draw (F) node{(2,0)};
    \draw (G) node{(0,3)};
    \draw (H) node{(1,2)};
    \draw (I) node{(2,1)};
    \draw (J) node{(3,0)};
    \draw (K) node{.    .    .    .};
    \draw (L) node{(n-1,0)}; 
    \draw (M) node{(0,n-1)}; 
    \draw (N) node{.    .    .    .};
    \draw (1) [fill=black,thick] circle (0.25mm);
    \draw (2) [fill=black,thick] circle (0.25mm);
    \draw (3) [fill=black,thick] circle (0.25mm);
    \draw (4) [fill=black,thick] circle (0.25mm);
    \draw (5) [fill=black,thick] circle (0.25mm);
    \draw (6) [fill=black,thick] circle (0.25mm);
    \draw (7) [fill=black,thick] circle (0.25mm);
    \draw (8) [fill=black,thick] circle (0.25mm);
    \draw (9) [fill=black,thick] circle (0.25mm);
    \draw (10) [fill=black,thick] circle (0.25mm);
    \draw (11) [fill=black,thick] circle (0.25mm);
    \draw (12) [fill=black,thick] circle (0.25mm);
    \end{tikzpicture}
\end{center}
The root of the tree is centered at $(0,0)$ defined as Level 1. For a point on the binary tree at coordinate $(m,n)$, a row-coordinate $m$ and a column-coordinate $n$, the level of that point is defined as $m+n+1$.
\begin{Ex}
Below is a binary tree with 4 leaves. Above each point, its respective coordinate is denoted.
\begin{center}
\begin{tikzpicture}
\draw [thick] (-3,-3) -- (0,0) -- (3,-3);
\draw [thick] (1,-1) -- (-1,-3);
\draw [thick] (2,-2) -- (1,-3);
\draw[fill] (0,0) circle [radius=0.025];
\node [above , black] at (0,0) {(0,0)};
\draw[fill] (-1,-1) circle [radius=0.025];
\node [above left, black] at (-1,-1) {(1,0)};
\draw[fill] (1,-1) circle [radius=0.025];
\node [above right, black] at (1,-1) {(0,1)};
\draw[fill] (-2,-2) circle [radius=0.025];
\node [above left, black] at (-2,-2) {(2,0)};
\draw[fill] (0,-2) circle [radius=0.025];
\node [above , black] at (-0.15,-2) {(1,1)};
\draw[fill] (2,-2) circle [radius=0.025];
\node [above right, black] at (2,-2) {(0,2)};
\draw[fill] (-3,-3) circle [radius=0.025];
\node [above left, black] at (-3,-3) {(3,0)};
\draw[fill] (-1,-3) circle [radius=0.025];
\node [above , black] at (-1.15,-3) {(2,1)};
\draw[fill] (1,-3) circle [radius=0.025];
\node [above , black] at (0.85,-3) {(1,2)};
\draw[fill] (3,-3) circle [radius=0.025];
\node [above right, black] at (3,-3) {(0,3)};
\end{tikzpicture}
\end{center}
\end{Ex}
\subsection{The 'Bookshelf' Construction}
\subsubsection{Bookshelf Mapping}
\begin{definition}
Shelf: Given a binary tree that is tilted by 45 degrees counterclockwise, a \emph{shelf} in this binary tree is any of the horizontal line except the top most one in the binary tree, which we call the \textit{ceiling}.   
\end{definition}

Notice that prior to tilting the tree, any shelf is a branch of the binary tree with negative slope, excluding the right most one which is the ceiling. 

Consider the binary tree example from before and tilt 45 degrees counter-clockwise:

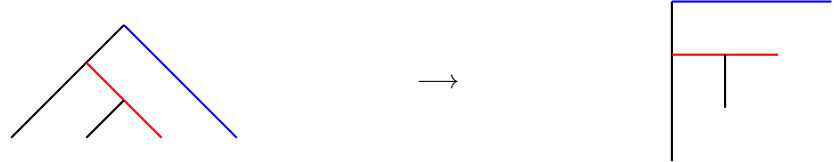
\begin{minipage}{0.5\textwidth}
\begin{center}
    \begin{tikzpicture}[scale=.5]
   \draw[thick](0,0)--(3,3); 
   \draw[thick, color=red](2,2)--(3,1); 
   \draw[thick, color=red](3,1)--(4,0); 
   \draw[thick](3,1)--(2,0);
   \draw[thick, color=blue](3,3)--(6,0);
   \end{tikzpicture}
\end{center}
\end{minipage}
$\longrightarrow$
\begin{minipage}{0.5\textwidth}
\begin{center}
    \begin{tikzpicture}[scale=.5, rotate=45]
   \draw[thick](0,0)--(3,3); 
   \draw[thick, color=red](2,2)--(3,1); 
   \draw[thick, color=red](3,1)--(4,0); 
   \draw[thick](3,1)--(2,0);
   \draw[thick, color=blue](3,3)--(6,0);
   \end{tikzpicture}
\end{center}
\end{minipage}
\\
where the shelf is marked in red and the ceiling marked in blue.\\
\begin{definition}
Length of the Shelf: Given any shelf on the binary tree from coordinate $(m,n)$ to $(m,k)$, \emph{the length of the shelf} is $k-n$.
\end{definition}

\textbf{Property 1.} Each shelf is at most the length of the previous shelf, i.e. the shelf above it.\footnote{By Definition 1, we know that Young Diagram must have $\lambda_{i}\geq\lambda_{j}$ for all $i > j$. Since the length of the shelf $l_1$ at coordinate $(i,j)$ is equivalent to the value of $\lambda_{i}$, the length of the shelf $l_2$ at coordinate $(i+1,k)$ must be smaller than or equal to $l_1$ by for $l_1=\lambda_{i}\geq\lambda_{i+1}=l_2$.}\\

\textbf{Proof.} To see why is true, notice that (1) if both shelves start at the same column such as $(n,0)$ and $(m,0)$, each shelf's length is the number of points that start with the same row value $n$ and $m$ respectively. By our coordinate construction if $m>n$ then number of points of the form $(n,k)$  is more than that of $(m,k)$ for some $k$.

(2) if the lower shelf begins at $(m,k)$ and the shelf above $(n,k')$ for $m>n$
 and $k'>k$ then the shelf at $n$ is shorter by a similar argument to (1).
 
(3) Suppose now that the lower shelf starts at a column before the shelf above it, that is the lower shelf begins at $(m,k)$ and the shelf above $(n,k')$ for $m>n$
 but $k'<k$. Then by the definition of a binary tree the point $(n,k')$ must be a node since otherwise the shelf would start before it. So there are 2 branches from this node, one constructing the shelf along row $n$ and the other going down along the column $k'$. Since $k'<k$  and $m>n$, there exists the point $(m,k')$ such that there are 2 different branches passing through it. This contradicts the definition of a binary tree. 
 
 We have shown that in all possible cases a shelf always has at most the same length as the shelf above it.\\

%Maybe show contradiction example for case (3). 

%I add the explanation of the property of the length of shelves.
%\textbf{Note on the length of the Shelf}: By Property 1, the length of each shelf is smaller than or equal to the higher shelves up to down.

Given a binary tree, the \textbf{'bookshelf' construction} puts $n$ boxes above any shelf with length $n$ recursively up until the ceiling, over all the shelves, i.e. above any branch with negative slope. This construction yields in a diagram that is a Young diagram or Young diagram with gaps.

Consider the following binary tree:

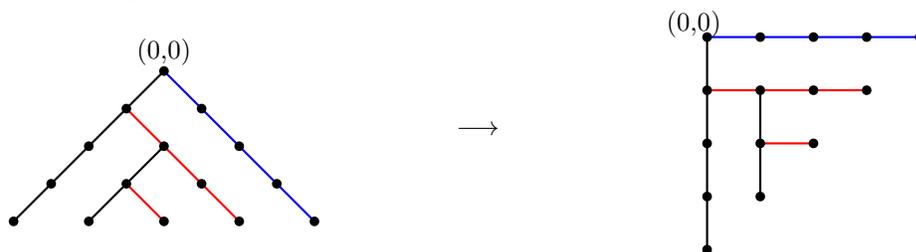
\begin{minipage}{0.5\textwidth}
\begin{center}
  \begin{tikzpicture}[scale=.5]
  \draw[thick](0,0)--(4,4);
  \draw[thick,color=blue](4,4)--(8,0);
  \draw[thick,color=red](3,3)--(6,0);
  \draw[thick](2,0)--(4,2);
  \draw[thick,color=red](3,1)--(4,0);
  \draw (4,4) [fill=black,thick] circle (1mm);
  \draw (3,3) [fill=black,thick] circle (1mm);
  \draw (2,2) [fill=black,thick] circle (1mm);
  \draw (1,1) [fill=black,thick] circle (1mm);
  \draw (0,0) [fill=black,thick] circle (1mm);
  \draw (3,1) [fill=black,thick] circle (1mm);
  \draw (2,0) [fill=black,thick] circle (1mm);
  \draw (4,2) [fill=black,thick] circle (1mm);
  \draw (5,3) [fill=black,thick] circle (1mm);
  \draw (6,2) [fill=black,thick] circle (1mm);
  \draw (7,1) [fill=black,thick] circle (1mm);
  \draw (8,0) [fill=black,thick] circle (1mm);
  \draw (4,0) [fill=black,thick] circle (1mm);
  \draw (6,0) [fill=black,thick] circle (1mm);
  \draw (5,1) [fill=black,thick] circle (1mm);
  \draw (4,4.5) node{(0,0)};
  %add the imaginary line for representing the boxes we want to put on
  \end{tikzpicture}
\end{center}
\end{minipage}
$\longrightarrow$
\begin{minipage}{0.5\textwidth}
\begin{center}
  \begin{tikzpicture}[scale=.5, rotate=45]
  \draw[thick](0,0)--(4,4);
  \draw[thick,color=blue](4,4)--(8,0);
  \draw[thick,color=red](3,3)--(6,0);
  \draw[thick](2,0)--(4,2);
  \draw[thick,color=red](3,1)--(4,0);
  \draw (4,4) [fill=black,thick] circle (1mm);
  \draw (3,3) [fill=black,thick] circle (1mm);
  \draw (2,2) [fill=black,thick] circle (1mm);
  \draw (1,1) [fill=black,thick] circle (1mm);
  \draw (0,0) [fill=black,thick] circle (1mm);
  \draw (3,1) [fill=black,thick] circle (1mm);
  \draw (2,0) [fill=black,thick] circle (1mm);
  \draw (4,2) [fill=black,thick] circle (1mm);
  \draw (5,3) [fill=black,thick] circle (1mm);
  \draw (6,2) [fill=black,thick] circle (1mm);
  \draw (7,1) [fill=black,thick] circle (1mm);
  \draw (8,0) [fill=black,thick] circle (1mm);
  \draw (4,0) [fill=black,thick] circle (1mm);
  \draw (6,0) [fill=black,thick] circle (1mm);
  \draw (5,1) [fill=black,thick] circle (1mm);
  \draw (4,4.5) node{(0,0)};
  %add the imaginary line for representing the boxes we want to put on
  \end{tikzpicture}
\end{center}
\end{minipage}

After rotating notice that by the bookshelf construction we have 2 shelves: one from the point $(1,0)$ to $(1,3)$ of length $3$, second from the point $(2,1)$ to $(2,2)$ of length $1$.

Putting in boxes (in gray) above each shelf:

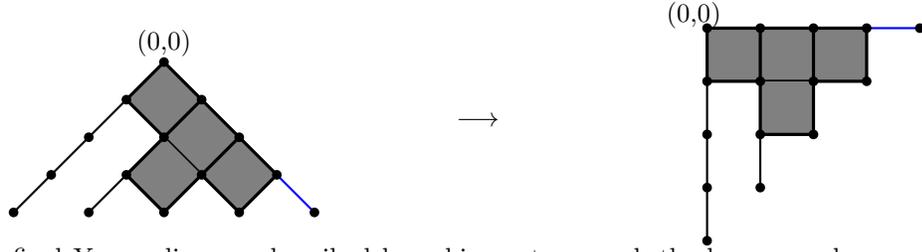
\begin{minipage}{0.5\textwidth}
\begin{center}
  \begin{tikzpicture}[scale=.5]
  \draw[thick](0,0)--(4,4);
  \draw[thick,color=blue](4,4)--(8,0);
  \draw[thick,color=red](3,3)--(6,0);
  \draw[thick](2,0)--(4,2);
  \draw[thick,color=red](3,1)--(4,0);
  \draw[very thick, fill=gray] (4,4)--(7,1)--(6,0)--(3,3)--(4,4);
  \draw[very thick, fill=gray] (4,2)--(5,3);
  \draw[very thick, fill=gray] (5,1)--(6,2);
  \draw[very thick, fill=gray] (5,1)--(4,0)--(3,1)--(4,2);
  \draw (4,4) [fill=black,thick] circle (1mm);
  \draw (3,3) [fill=black,thick] circle (1mm);
  \draw (2,2) [fill=black,thick] circle (1mm);
  \draw (1,1) [fill=black,thick] circle (1mm);
  \draw (0,0) [fill=black,thick] circle (1mm);
  \draw (3,1) [fill=black,thick] circle (1mm);
  \draw (2,0) [fill=black,thick] circle (1mm);
  \draw (4,2) [fill=black,thick] circle (1mm);
  \draw (5,3) [fill=black,thick] circle (1mm);
  \draw (6,2) [fill=black,thick] circle (1mm);
  \draw (7,1) [fill=black,thick] circle (1mm);
  \draw (8,0) [fill=black,thick] circle (1mm);
  \draw (4,0) [fill=black,thick] circle (1mm);
  \draw (6,0) [fill=black,thick] circle (1mm);
  \draw (5,1) [fill=black,thick] circle (1mm);
  \draw (4,4.5) node{(0,0)};
  \end{tikzpicture}
\end{center}
\end{minipage}
$\longrightarrow$
\begin{minipage}{0.5\textwidth}
\begin{center}
  \begin{tikzpicture}[scale=.5, rotate=45]
  \draw[thick](0,0)--(4,4);
  \draw[thick,color=blue](4,4)--(8,0);
  \draw[thick,color=red](3,3)--(6,0);
  \draw[thick](2,0)--(4,2);
  \draw[thick,color=red](3,1)--(4,0);
  \draw[very thick, fill=gray] (4,4)--(7,1)--(6,0)--(3,3)--(4,4);
  \draw[very thick, fill=gray] (4,2)--(5,3);
  \draw[very thick, fill=gray] (5,1)--(6,2);
  \draw[very thick, fill=gray] (5,1)--(4,0)--(3,1)--(4,2);
  \draw (4,4) [fill=black,thick] circle (1mm);
  \draw (3,3) [fill=black,thick] circle (1mm);
  \draw (2,2) [fill=black,thick] circle (1mm);
  \draw (1,1) [fill=black,thick] circle (1mm);
  \draw (0,0) [fill=black,thick] circle (1mm);
  \draw (3,1) [fill=black,thick] circle (1mm);
  \draw (2,0) [fill=black,thick] circle (1mm);
  \draw (4,2) [fill=black,thick] circle (1mm);
  \draw (5,3) [fill=black,thick] circle (1mm);
  \draw (6,2) [fill=black,thick] circle (1mm);
  \draw (7,1) [fill=black,thick] circle (1mm);
  \draw (8,0) [fill=black,thick] circle (1mm);
  \draw (4,0) [fill=black,thick] circle (1mm);
  \draw (6,0) [fill=black,thick] circle (1mm);
  \draw (5,1) [fill=black,thick] circle (1mm);
  \draw (4,4.5) node{(0,0)};
  \end{tikzpicture}
\end{center}
\end{minipage}

To get the final Young diagram described by a binary tree, push the boxes row-by-row from right to left such that the left-most box at each row is aligned, thus yielding a Young diagram. 

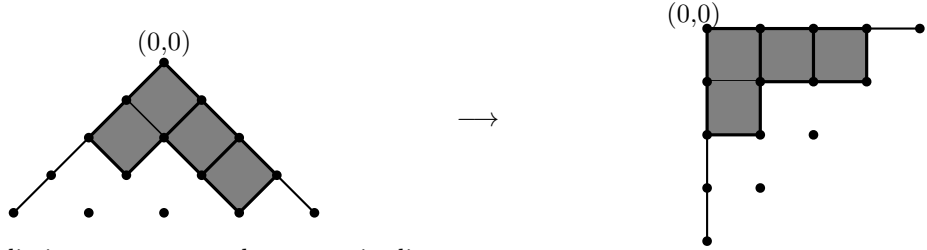
\begin{minipage}{0.5\textwidth}
\begin{center}
  \begin{tikzpicture}[scale=.5]
  \draw[thick](0,0)--(4,4);
  \draw[thick](4,4)--(8,0);
  %\draw[thick,color=red](3,3)--(6,0);
  %\draw[thick](2,0)--(4,2);
  %\draw[thick,color=red](3,1)--(4,0);
  \draw[very thick, fill=gray] (4,4)--(7,1)--(6,0)--(3,3)--(4,4);
  \draw[very thick, fill=gray] (4,2)--(5,3);
  \draw[very thick, fill=gray] (5,1)--(6,2);
  \draw[very thick, fill=gray] (4,2)--(3,1)--(2,2)--(3,3);
  \draw (4,4) [fill=black,thick] circle (1mm);
  \draw (3,3) [fill=black,thick] circle (1mm);
  \draw (2,2) [fill=black,thick] circle (1mm);
  \draw (1,1) [fill=black,thick] circle (1mm);
  \draw (0,0) [fill=black,thick] circle (1mm);
  \draw (3,1) [fill=black,thick] circle (1mm);
  \draw (2,0) [fill=black,thick] circle (1mm);
  \draw (4,2) [fill=black,thick] circle (1mm);
  \draw (5,3) [fill=black,thick] circle (1mm);
  \draw (6,2) [fill=black,thick] circle (1mm);
  \draw (7,1) [fill=black,thick] circle (1mm);
  \draw (8,0) [fill=black,thick] circle (1mm);
  \draw (4,0) [fill=black,thick] circle (1mm);
  \draw (6,0) [fill=black,thick] circle (1mm);
  \draw (5,1) [fill=black,thick] circle (1mm);
  \draw (4,4.5) node{(0,0)};
  \end{tikzpicture}
\end{center}
\end{minipage}
$\longrightarrow$
\begin{minipage}{0.5\textwidth}
\begin{center}
  \begin{tikzpicture}[scale=.5, rotate=45]
  \draw[thick](0,0)--(4,4);
  \draw[thick](4,4)--(8,0);
  %\draw[thick,color=red](3,3)--(6,0);
  %\draw[thick](2,0)--(4,2);
  %\draw[thick,color=red](3,1)--(4,0);
  \draw[very thick, fill=gray] (4,4)--(7,1)--(6,0)--(3,3)--(4,4);
  \draw[very thick, fill=gray] (4,2)--(5,3);
  \draw[very thick, fill=gray] (5,1)--(6,2);
  \draw[very thick, fill=gray] (4,2)--(3,1)--(2,2)--(3,3);
  \draw (4,4) [fill=black,thick] circle (1mm);
  \draw (3,3) [fill=black,thick] circle (1mm);
  \draw (2,2) [fill=black,thick] circle (1mm);
  \draw (1,1) [fill=black,thick] circle (1mm);
  \draw (0,0) [fill=black,thick] circle (1mm);
  \draw (3,1) [fill=black,thick] circle (1mm);
  \draw (2,0) [fill=black,thick] circle (1mm);
  \draw (4,2) [fill=black,thick] circle (1mm);
  \draw (5,3) [fill=black,thick] circle (1mm);
  \draw (6,2) [fill=black,thick] circle (1mm);
  \draw (7,1) [fill=black,thick] circle (1mm);
  \draw (8,0) [fill=black,thick] circle (1mm);
  \draw (4,0) [fill=black,thick] circle (1mm);
  \draw (6,0) [fill=black,thick] circle (1mm);
  \draw (5,1) [fill=black,thick] circle (1mm);
  \draw (4,4.5) node{(0,0)};
  \end{tikzpicture}
\end{center}
\end{minipage}
This last step eliminates any gaps that occur in diagram.\\\\

\begin{center}
    \begin{tikzpicture}
    
    \end{tikzpicture}
\end{center}
\subsubsection{Bijection between Binary Trees and Young Diagrams}

To show that there's a bijection between Binary Trees and Young diagram, we need firstly define the inverse book-shelf construction, which maps a given Young Diagram to a binary tree\footnote{The bijection is outlined in detail in the appendix of this paper.}.
\subsubsection{Inverse Bookshelf Mapping}
Define the 'Inverse Bookshelf' construction from a Young diagram to binary tree as follows:

Consider a Young diagram where $k = \# i$'s such that $\lambda'_i=\lambda'_1$ and $n = \# j$'s such that $\lambda_j=\lambda_1$.  In order to place gaps appropriately to achieve the corresponding binary tree follow the steps as follows:
\begin{enumerate}
         \item[-] if $\lambda'_1+k+1> \lambda_1+n+1$: locate first  column $i$ such that $\lambda_i' > \lambda_{i+1}' +1$ and put a column gap between columns $i$ and $i+1$ of size $\lambda_i'-\lambda_{i+1}'-1$. 
    \begin{itemize}
        \item  The gap creates 2 subdiagrams: the subdiagram to the left of the gap forms, X, and the subdiagram to the right of the gap forms, Y.
    \end{itemize}
    \item[-]  if $\lambda_1+n+1 \geq \lambda_1'+k+1$: locate first row $j$ such that $\lambda_j > \lambda_{j+1} +1$. %Shift the remaining boxes below row $j$ to the right by $\lambda_j-\lambda_{j+1}-1$. 
    \begin{itemize}
        \item The boxes below row $j$ yields a new subdiagram Z. %the shifting creates one subdiagram: The remaining boxes form the subdiagram Z. 
        %The right subtree Y, consists only of $\#m | \lambda_j=\lambda_m$ many positive sloped branches.
    \end{itemize}

\end{enumerate}
Repeat the process above for all subdiagrams, X,Y,Z with a 'stopping condition':
\begin{itemize}
    \item[-] a diagram with single column at $(i,j)$ \footnote{a box at $(i,j)$ means the box is at row i and column j}, i.e. $\lambda_{j+1}' = 0$ (or $\lambda_i = 1$)\\ \textbf{or}
    \item[-]  a diagram with a single row at $(i,j)$, i.e. $\lambda_{j}' = 1$ (or $\lambda_{i+1} = 0$). 
\end{itemize}

Conclude the stopping condition by row shifting to the right:
\begin{itemize}
    \item[-] any diagram with a single row by $\lambda_{i-1}-\lambda_{i}-1$
    \item[-] any diagram with single column by $\lambda_{i-1}-\lambda_j'+i$
\end{itemize}

Notice that a single box yields the same shift in either case.

This process yields the final corresponding Young diagram with gaps. 

To map from the Young diagram with gaps to the binary tree, put shelves below each row. Fill in the rest of the binary tree with positive sloped branches until the definition of a binary tree is satisfied.

\textbf{Note on the size of the binary tree:}
Let $k = \# i$'s such that $\lambda'_i=\lambda'_1$.
Let $n = \# j$'s such that $\lambda_j=\lambda_1$
Then the minimum size of the binary tree required to fit a Young diagram is:
$$max \{\lambda'_1+k+1, \lambda_1+n+1\} $$

Consider the following step-by step example: 

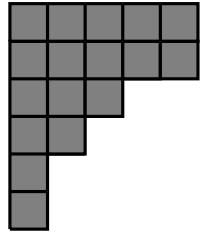
\begin{minipage}{0.5\textwidth}
    \begin{center}
    \begin{tikzpicture}[scale=.5]
    \draw[very thick, fill=gray] (0,0)--(1,0)--(1,5)--(0,5)--(0,0);
    \draw[very thick,fill=gray] (0,4)--(1,4);
    \draw[very thick,fill=gray] (0,3)--(1,3);
    \draw[very thick,fill=gray] (0,2)--(2,2);
    \draw[very thick,fill=gray] (0,1)--(1,1);
    \draw[very thick,fill=gray]
    (1,2)--(2,2)--(2,5)--(1,5)--(1,2);
    \draw[very thick,fill=gray] (1,4)--(2,4);
    \draw[very thick,fill=gray] (1,3)--(3,3);
    \draw[very thick,fill=gray]
    (2,3)--(3,3)--(3,5)--(2,5)--(2,3);
    \draw[very thick,fill=gray] (2,4)--(4,4);
    \draw[very thick,fill=gray]
    (3,4)--(4,4)--(4,5)--(3,5)--(3,4);
    \draw[very thick,fill=gray]
    (4,4)--(5,4)--(5,5)--(4,5)--(4,4);
    \draw[very thick,fill=gray](0,5)--(5,5)--(5,6)--(0,6)--(0,5);
    \draw[very thick,fill=gray](1,5)--(1,6);
    \draw[very thick,fill=gray](2,5)--(2,6);
    \draw[very thick,fill=gray](3,5)--(3,6);
    \draw[very thick,fill=gray](4,5)--(4,6);
    \end{tikzpicture}
    \end{center}
    Given a Young Diagram as above.
    \end{minipage}
    $\Longrightarrow$
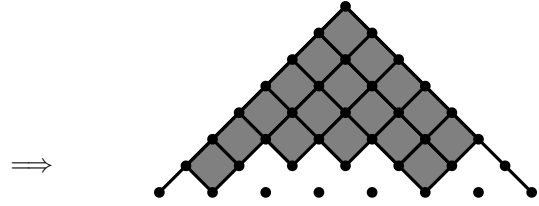
    \begin{minipage}{0.5\textwidth}
    \begin{center}
    \begin{tikzpicture}[scale=.5,rotate=-45]
    \draw[very thick, fill=gray] (0,0)--(1,0)--(1,5)--(0,5)--(0,0);
    \draw[very thick,fill=gray] (0,4)--(1,4);
    \draw[very thick,fill=gray] (0,3)--(1,3);
    \draw[very thick,fill=gray] (0,2)--(2,2);
    \draw[very thick,fill=gray] (0,1)--(1,1);
    \draw[very thick,fill=gray]
    (1,2)--(2,2)--(2,5)--(1,5)--(1,2);
    \draw[very thick,fill=gray] (1,4)--(2,4);
    \draw[very thick,fill=gray] (1,3)--(3,3);
    \draw[very thick,fill=gray]
    (2,3)--(3,3)--(3,5)--(2,5)--(2,3);
    \draw[very thick,fill=gray] (2,4)--(4,4);
    \draw[very thick,fill=gray]
    (3,4)--(4,4)--(4,5)--(3,5)--(3,4);
    \draw[very thick,fill=gray]
    (4,4)--(5,4)--(5,5)--(4,5)--(4,4);
    \draw[very thick,fill=gray](0,5)--(5,5)--(5,6)--(0,6)--(0,5);
    \draw[very thick,fill=gray](1,5)--(1,6);
    \draw[very thick,fill=gray](2,5)--(2,6);
    \draw[very thick,fill=gray](3,5)--(3,6);
    \draw[very thick,fill=gray](4,5)--(4,6);
    \foreach \Point in {(0,-1),(1,0),(2,1),(3,2),(4,3),(5,4),(6,5),(0,0),(1,1),(2,2),(3,3),(4,4),(5,5),(0,1),(1,2),(2,3),(3,4),(4,5),(0,2),(1,3),(2,4),(3,5),(0,3),(1,4),(2,5),(0,4),(1,5),(0,5),(0,6),(1,6),(2,6),(3,6),(4,6),(5,6),(6,6),(7,6)}
    \draw \Point[very thick,fill=black] circle(1mm);
    \draw [color=black, very thick] (0,6)--(7,6);
    \draw [color=black,very thick] (0,-1)--(0,6);
    \end{tikzpicture}
    \end{center}
    Tilt it $45$ degrees clockwise and set the top-left corner of the diagram to the root of the appropriate sized tree of $max\{ 6+1+1, 5+2+1\}=8$ children
    \end{minipage}
    \\
    \\
    \\
    \\
    $\Longrightarrow$
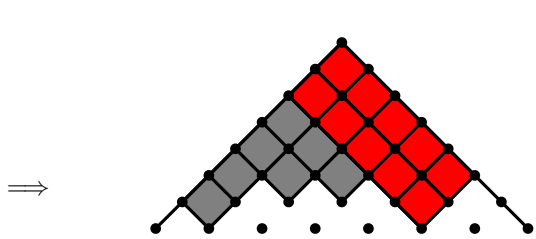
    \begin{minipage}{0.5\textwidth}
    \begin{center}
        \begin{tikzpicture}[scale=.5,rotate=-45]
        \draw[very thick,fill=gray] (0,0)--(1,0)--(1,4)--(0,4)--(0,0);
        \draw[very thick,fill=gray] (0,1)--(1,1);
        \draw[very thick,fill=gray] (0,2)--(1,2);
        \draw[very thick,fill=gray] (0,3)--(1,3);
        \draw[very thick,fill=gray]
        (1,2)--(2,2)--(2,4)--(1,4)--(1,2);
        \draw[very thick,fill=gray] (1,3)--(2,3);
        \draw[very thick,fill=gray] 
        (2,3)--(3,3)--(3,4)--(2,4)--(2,3);
         \draw[very thick,fill=red](0,5)--(5,5)--(5,6)--(0,6)--(0,5);
        \draw[very thick,fill=gray](1,5)--(1,6);
        \draw[very thick,fill=gray](2,5)--(2,6);
        \draw[very thick,fill=gray](3,5)--(3,6);
        \draw[very thick,fill=gray](4,5)--(4,6);
        \draw[very thick,fill=red] (0,4)--(5,4)--(5,5)--(0,5)--(0,4);
        \draw[very thick,fill=gray] (1,4)--(1,5);
        \draw[very thick,fill=gray] (2,4)--(2,5);
        \draw[very thick,fill=gray] (3,4)--(3,5);
        \draw[very thick,fill=gray] (4,4)--(4,5);
        \foreach \Point in {(0,-1),(1,0),(2,1),(3,2),(4,3),(5,4),(6,5),(0,0),(1,1),(2,2),(3,3),(4,4),(5,5),(0,1),(1,2),(2,3),(3,4),(4,5),(0,2),(1,3),(2,4),(3,5),(0,3),(1,4),(2,5),(0,4),(1,5),(0,5),(0,6),(1,6),(2,6),(3,6),(4,6),(5,6),(6,6),(7,6)}
        \draw \Point[very thick,fill=black] circle(1mm);
        \draw [color = black, very thick] (0,5)--(0,6)--(7,6);
        \draw [color=black, very thick] (0,4)--(5,4);
        \draw [color=black,very thick] (0,-1)--(0,5);
        \end{tikzpicture}
    \end{center}
        Then since we have $\lambda_{1}+2+1\geq\lambda_{1}'+1+1$, we search the rows and we get $\lambda_{2}>\lambda_{3}+1$. Then we label the part that is above the row 3 as red.
    \end{minipage}
    $\Longrightarrow$
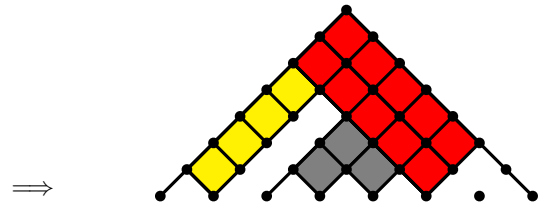
    \begin{minipage}{0.5\textwidth}
    \begin{center}
        \begin{tikzpicture}[scale=.5,rotate=-45]
        \draw[very thick,fill=yellow] (0,0)--(1,0)--(1,4)--(0,4)--(0,0);
        \draw[very thick,fill=gray] (0,1)--(1,1);
        \draw[very thick,fill=gray] (0,2)--(1,2);
        \draw[very thick,fill=gray] (0,3)--(1,3);
        \draw[very thick,fill=gray]
        (2,2)--(3,2)--(3,4)--(2,4)--(2,2);
        \draw[very thick,fill=gray] (2,3)--(3,3);
        \draw[very thick,fill=gray] 
        (3,3)--(4,3)--(4,4)--(3,4)--(3,3);
         \draw[very thick,fill=red](0,4)--(5,4)--(5,6)--(0,6)--(0,4);
        \draw[very thick,fill=red](0,4)--(5,4)--(5,5)--(0,5)--(0,4);
        \draw[very thick,fill=gray](1,4)--(1,6);
        \draw[very thick,fill=gray](2,4)--(2,6);
        \draw[very thick,fill=gray](3,4)--(3,6);
        \draw[very thick,fill=gray](4,4)--(4,6);
        \draw[very thick,fill=gray](0,5)--(5,5);
        \foreach \Point in {(0,-1),(1,0),(2,1),(3,2),(4,3),(5,4),(6,5),(0,0),(1,1),(2,2),(3,3),(4,4),(5,5),(0,1),(1,2),(2,3),(3,4),(4,5),(0,2),(1,3),(2,4),(3,5),(0,3),(1,4),(2,5),(0,4),(1,5),(0,5),(0,6),(1,6),(2,6),(3,6),(4,6),(5,6),(6,6),(7,6)}
        \draw \Point[very thick,fill=black] circle(1mm);
        \draw [color = black, very thick] (0,5)--(0,6)--(7,6);
        \draw [color=black,very thick] (0,4)--(5,4);
        \draw [color=black,very thick] (0,-1)--(0,5);
        \draw [color=black,very thick] (0,0)--(1,0);
        \draw [color=black,very thick] (2,4)--(2,1);
        \end{tikzpicture}
    \end{center}
    For the remainder is the subdiagram that is below the red labeled rows, we have $\lambda_{1}+1+1<\lambda_{1}'+1+1$. Search and get that $\lambda_{1}' > \lambda_{2}'+1$. Then we put a column gap between the columns 1 (marked in yellow) and 2 of size $\lambda_1' -\lambda_2'-1 = 4-2-1 =1 $.
    \end{minipage}
    \\
    \\
    \\
    \\
    $\Longrightarrow$
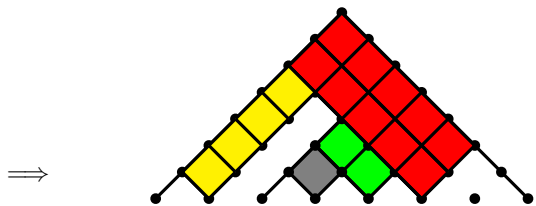
     \begin{minipage}{0.5\textwidth}
     \begin{center}
        \begin{tikzpicture}[scale=.5,rotate=-45]
        \draw[very thick,fill=green]
        (2,3)--(4,3)--(4,4)--(2,4)--(2,3);
        \draw[very thick,fill=gray] (3,3)--(3,4);
        \draw[very thick,fill=gray]
        (2,2)--(3,2)--(3,3)--(2,3)--(2,2);
        \foreach \Point in {(0,-1),(1,0),(2,1),(3,2),(4,3),(5,4),(6,5),(0,0),(1,1),(2,2),(3,3),(4,4),(5,5),(0,1),(1,2),(2,3),(3,4),(4,5),(0,2),(1,3),(2,4),(3,5),(0,3),(1,4),(2,5),(0,4),(1,5),(0,5),(0,6),(1,6),(2,6),(3,6),(4,6),(5,6),(6,6),(7,6)}
        \draw \Point[very thick, fill=black] circle(1mm);
        \draw[very thick,fill=red](0,4)--(5,4)--(5,6)--(0,6)--(0,4);
        \draw[very thick,fill=red](0,4)--(5,4)--(5,5)--(0,5)--(0,4);
        \draw[very thick,fill=gray](1,4)--(1,6);
        \draw[very thick,fill=gray](2,4)--(2,6);
        \draw[very thick,fill=gray](3,4)--(3,6);
        \draw[very thick,fill=gray](4,4)--(4,6);
        \draw[very thick,fill=gray](0,5)--(5,5);
        \draw [color = black, very thick] (0,5)--(0,6)--(7,6);
        \draw [color=black,very thick] (0,4)--(5,4);
        \draw [color=black,very thick] (0,-1)--(0,5);
        \draw [color=black,very thick] (0,0)--(1,0);
        \draw [color=black,very thick] (2,4)--(2,1);
        \draw [color=black,very thick] (2,3)--(4,3);
        \draw[very thick,fill=yellow] (0,0)--(1,0)--(1,4)--(0,4)--(0,0);
        \draw[very thick,fill=gray] (0,1)--(1,1);
        \draw[very thick,fill=gray] (0,2)--(1,2);
        \draw[very thick,fill=gray] (0,3)--(1,3);
        \end{tikzpicture}
        \end{center}
        Here do the same process recursively for the right subdiagram yields a new subdiagram of a single box (stopping condition) below the row marked in green.
    \end{minipage}
    $\Longrightarrow$
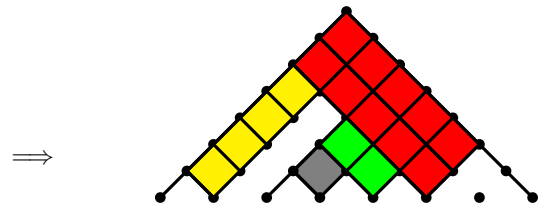
     \begin{minipage}{0.5\textwidth}
     \begin{center}
        \begin{tikzpicture}[scale=.5,rotate=-45]
        \draw[very thick,fill=gray]
        (2,2)--(3,2)--(3,3)--(2,3)--(2,2);
        \draw[very thick,fill=gray] (2,3)--(3,3);
        \foreach \Point in {(0,-1),(1,0),(2,1),(3,2),(4,3),(5,4),(6,5),(0,0),(1,1),(2,2),(3,3),(4,4),(5,5),(0,1),(1,2),(2,3),(3,4),(4,5),(0,2),(1,3),(2,4),(3,5),(0,3),(1,4),(2,5),(0,4),(1,5),(0,5),(0,6),(1,6),(2,6),(3,6),(4,6),(5,6),(6,6),(7,6)}
        \draw \Point[very thick, fill=black] circle(1mm);
        \draw [color = black, very thick] (0,5)--(0,6)--(7,6);
        \draw[very thick,fill=red](0,4)--(5,4)--(5,6)--(0,6)--(0,4);
        \draw[very thick,fill=red](0,4)--(5,4)--(5,5)--(0,5)--(0,4);
        \draw[very thick,fill=gray](1,4)--(1,6);
        \draw[very thick,fill=gray](2,4)--(2,6);
        \draw[very thick,fill=gray](3,4)--(3,6);
        \draw[very thick,fill=gray](4,4)--(4,6);
        \draw[very thick,fill=gray](0,5)--(5,5);
        \draw[very thick,fill=yellow] (0,0)--(1,0)--(1,4)--(0,4)--(0,0);
        \draw[very thick,fill=gray] (0,1)--(1,1);
        \draw[very thick,fill=gray] (0,2)--(1,2);
        \draw[very thick,fill=gray] (0,3)--(1,3);
        \draw [color=black,very thick] (0,4)--(5,4);
        \draw [color=black,very thick] (0,-1)--(0,5);
        \draw [color=black,very thick] (0,0)--(1,0);
        \draw [color=black,very thick] (2,4)--(2,1);
        \draw [color=black,very thick] (2,3)--(4,3);
        \draw [color=black,very thick] (2,2)--(3,2);
        \draw[very thick,fill=green]
        (2,3)--(4,3)--(4,4)--(2,4)--(2,3);
        \draw[very thick,fill=gray] (3,3)--(3,4);
        \end{tikzpicture}
        \end{center}
        Now the diagram is divided into smallest component of a diagram, yielding the final young diagram with gaps.
    \end{minipage}
    \\
    \\
    \\
    \\
    $\Longrightarrow$
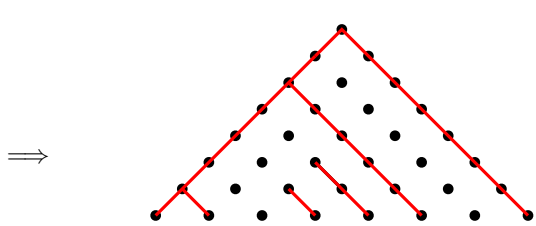
     \begin{minipage}{0.5\textwidth}
     \begin{center}
        \begin{tikzpicture}[scale=.5,rotate=-45]
        \draw[very thick,fill=gray] (2,3)--(3,3);
        \foreach \Point in {(0,-1),(1,0),(2,1),(3,2),(4,3),(5,4),(6,5),(0,0),(1,1),(2,2),(3,3),(4,4),(5,5),(0,1),(1,2),(2,3),(3,4),(4,5),(0,2),(1,3),(2,4),(3,5),(0,3),(1,4),(2,5),(0,4),(1,5),(0,5),(0,6),(1,6),(2,6),(3,6),(4,6),(5,6),(6,6),(7,6)}
        \draw \Point[very thick, fill=black] circle(1mm);
        \draw [color = red, very thick] (0,5)--(0,6)--(7,6);
        \draw [color = red, very thick] (0,4)--(5,4);
        \draw [color=red,very thick] (0,-1)--(0,5);
        \draw [color=red,very thick] (0,0)--(1,0);
       % \draw [color= red,very thick] (2,4)--(2,1);
        \draw [color=red,very thick] (2,3)--(4,3);
        \draw [color=red,very thick] (2,2)--(3,2);
        %\draw [color=blue,very thick] (6,5)--(6,6);
        \end{tikzpicture}
        \end{center}
        Put the shelves (negative sloped branches) below each block we have labeled. 
    \end{minipage}
    $\Longrightarrow$
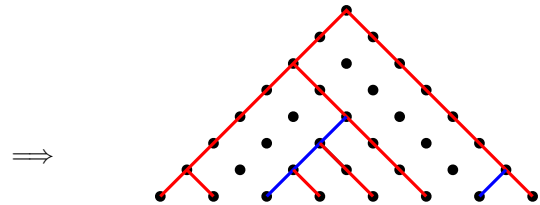
     \begin{minipage}{0.5\textwidth}
     \begin{center}
        \begin{tikzpicture}[scale=.5,rotate=-45]
        \draw[very thick,fill=gray] (2,3)--(3,3);
        \foreach \Point in {(0,-1),(1,0),(2,1),(3,2),(4,3),(5,4),(6,5),(0,0),(1,1),(2,2),(3,3),(4,4),(5,5),(0,1),(1,2),(2,3),(3,4),(4,5),(0,2),(1,3),(2,4),(3,5),(0,3),(1,4),(2,5),(0,4),(1,5),(0,5),(0,6),(1,6),(2,6),(3,6),(4,6),(5,6),(6,6),(7,6)}
        \draw \Point[very thick, fill=black] circle(1mm);
        \draw [color = red, very thick] (0,5)--(0,6)--(7,6);
        \draw [color = red, very thick] (0,4)--(5,4);
        \draw [color=red,very thick] (0,-1)--(0,5);
        \draw [color=red,very thick] (0,0)--(1,0);
        \draw [color= blue,very thick] (2,4)--(2,1);
        \draw [color=red,very thick] (2,3)--(4,3);
        \draw [color=red,very thick] (2,2)--(3,2);
        \draw [color=blue,very thick] (6,5)--(6,6);
        \end{tikzpicture}
        \end{center}
        For the missing children at bottom that are not connected, we connect them by positive sloped line (as the blue line showed).
    \end{minipage}
    \\
    \\
    \\
    \\

\subsubsection{Commutative Mapping between Binary Tree and Young Diagram }
%I think what we can say is each time we do the bookshelf construction is actually equivalent to filling the maximum rectangle on the dyck path(which can be shown easily by that turn), because Dyck path has already given us all of the information of Young diagram above the path, so it doesn't matter the order of filling the box on the path.
The bijection between Binary trees and Young diagrams using Dyck paths are well known. The mapping from binary tree to Dyck path is discussed in detail in Section by Bradley. The mapping from Dyck path to Young diagram is also well known. For a given Dyck path just take the shadow to get the Young diagram. 

The direct mapping form Binary trees to Young diagrams is given by the 'Bookshelf construction' is described in Section (2).

In this section, it will be proven by induction that the following diagram commutes, i.e the Young Diagram given using the bookshelf construction is equal to the one given by the well known mappings between binary trees and Dyck paths and Young diagrams. 

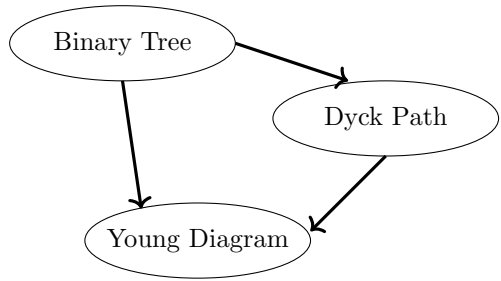
\begin{center}
    \begin{tikzpicture}[scale=.5]
     \draw {(0,0) ellipse (3cm and 1cm)};
     \draw (0,0) node{Binary Tree};
     \draw [->, very thick] (3,0)--(6,-1);
     \draw {(7,-2) ellipse (3cm and 1cm)};
     \draw (7,-2) node{Dyck Path};
     \draw [->,very thick] (7,-3)--(5,-5);
     \draw {(2,-5.25) ellipse (3cm and 1cm)};
     \draw (2,-5.25) node{Young Diagram};
     \draw [<-, very thick] (0.5,-4.4)--(0,-1);
    \end{tikzpicture}
\end{center}

\textbf{Proof:}

Let X denote the left-subtree and Y the right-subtree of a given binary tree. For a given binary tree, there are 4 possible patterns with the subdiagram:

\begin{enumerate}
    \item $X= \emptyset$ and $Y=\emptyset$
    \item $X= \emptyset$ and $Y\neq \emptyset$
    \item $X\neq \emptyset$ and $Y = \emptyset$
    \item $X\neq \emptyset$ and $Y \neq \emptyset$
\end{enumerate}
We prove that the diagram commute in each case:

(1) Consider the base case where both subtrees are empty:

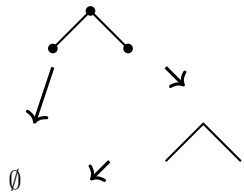
\begin{center}
    \begin{tikzpicture}[scale=.5]
    %Binary
    \draw[ thick] (0,0)--(1,1)--(2,0);
    \draw (0,0) [fill=black,thick] circle (1mm);
    \draw (1,1) [fill=black,thick] circle (1mm);
    \draw (2,0) [fill=black,thick] circle (1mm);
    %Dyck
    \begin{scope}[xshift=3cm,yshift=-3cm] 
    \draw[ thick] (0,0)--(1,1)--(2,0);
    \end{scope}
    %Young
    \begin{scope}[xshift=-1cm,yshift=-3.5cm]
    \draw (0,0) node{$\emptyset$};
    \end{scope}
    \draw [->, very thick] (3,-0.5)--(3.5,-1);
    \draw [->, very thick] (1.5,-3)--(1,-3.5);
    \draw [<-, very thick] (-0.5,-2)--(0,-0.5);
    \end{tikzpicture}
\end{center}
It is shown above that in this case the diagram commute. 

(2) For a binary tree of size $n+1$, i.e. $n$ many children, consider the left subtree, X is non-empty and the right subtree, Y is empty.

Using the Dyck path mapping, we get:
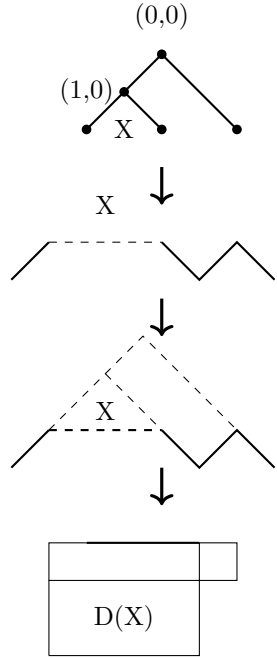
\begin{center}
    \begin{tikzpicture}[scale=.5]
    %Binary
    \begin{scope}[xshift=-3cm]
    \draw[ thick] (0,0)--(2,2)--(4,0);
    \draw[thick] (1,1)--(2,0);
    \draw (0,0) [fill=black,thick] circle (1mm);
    \draw (1,1) [fill=black,thick] circle (1mm);
    \draw (2,0) [fill=black,thick] circle (1mm);
    \draw (4,0) [fill=black,thick] circle (1mm);
    \draw (2,2) [fill=black,thick] circle (1mm);
    \draw (1,0) node{X};
    \draw (2,3) node{(0,0)};
    \draw (0,1) node{(1,0)};
    \end{scope}
    \draw [->, very thick] (-1,-1)--(-1,-2);
    %Dyck
    \begin{scope}[yshift=-4cm] 
    \draw[ thick] (-1,1)--(0,0)--(1,1)--(2,0);
    \draw[thick] (-5,0)--(-4,1);
    \draw[dashed](-4,1)--(-1,1);
    \draw (-2.5,2) node{X};
    \end{scope}
    \draw [->, very thick] (-1,-4.5)--(-1,-5.5);
    \begin{scope}[yshift=-9cm] 
    \draw[ thick] (-1,1)--(0,0)--(1,1)--(2,0);
    \draw[thick] (-5,0)--(-4,1);
    \draw[dashed,thick](-4,1)--(-1,1);
    \draw[dashed](-4,1)--(-1.5,3.5)--(1,1);
    \draw[dashed](-2.5,2.5)--(-1,1);
    \draw (-2.5,1.5) node{X};
    \end{scope}
    \draw [->, very thick] (-1,-9)--(-1,-10);
    %Young
    \begin{scope}[xshift=-4cm, yshift=-12cm] 
    \draw (0,0)--(0,1)--(1,1);
    \draw (2,-1) node{D(X)};
   % \draw (2,0.5) [fill=black,thick] circle (1mm);
    %\draw (2.5,0.5) [fill=black,thick] circle (1mm);
    %\draw (3,0.5) [fill=black,thick] circle (1mm);
    \draw (4,0)--(4,1)--(5,1)--(5,0)--(4,0);
    \draw (0,-2)--(0,0)--(4,0)--(4,-2)--(0,-2);
    \draw[thick] (1,1)--(4,1);
    \end{scope} 
    \end{tikzpicture}
\end{center}
The size of the first row of the diagram is determined by the size of the binary tree, $n+1$, where $n$ is the size of the subtree X.  

The subdiagram $D(X)$ is determined from the properties of the subtree $X$ in a recursive fashion.

Using the bookshelf mapping:

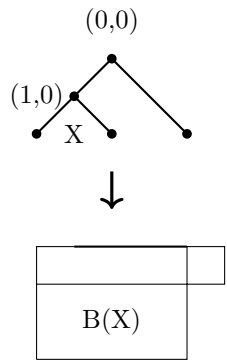
\begin{center}
    \begin{tikzpicture}[scale=.5]
    %Binary
    \draw[ thick] (0,0)--(2,2)--(4,0);
    \draw[thick] (1,1)--(2,0);
    \draw (0,0) [fill=black,thick] circle (1mm);
    \draw (1,1) [fill=black,thick] circle (1mm);
    \draw (2,0) [fill=black,thick] circle (1mm);
    \draw (4,0) [fill=black,thick] circle (1mm);
    \draw (2,2) [fill=black,thick] circle (1mm);
    \draw (1,0) node{X};
    \draw (2,3) node{(0,0)};
    \draw (0,1) node{(1,0)};
    \draw [->, very thick] (2,-1)--(2,-2);
    %Young
    \begin{scope}[yshift=-4cm]
    \draw (0,0)--(0,1)--(1,1);
    \draw (2,-1) node{B(X)};
   % \draw (2,0.5) [fill=black,thick] circle (1mm);
    %\draw (2.5,0.5) [fill=black,thick] circle (1mm);
    %\draw (3,0.5) [fill=black,thick] circle (1mm);
    \draw (4,0)--(4,1)--(5,1)--(5,0)--(4,0);
    \draw (0,-2)--(0,0)--(4,0)--(4,-2)--(0,-2);
    \draw[thick] (1,1)--(4,1);
    \end{scope}
    \end{tikzpicture}
\end{center}
The subdiagram $B(X)$ is determined from the properties of the subtree $X$.

Notice that in both mappings, since Y is empty the first row of the Young diagram is dependent on X but always at least one box longer. Also notice that if X is the base case, then the final Young diagram is a single box since the corresponding Young diagram of X is empty.

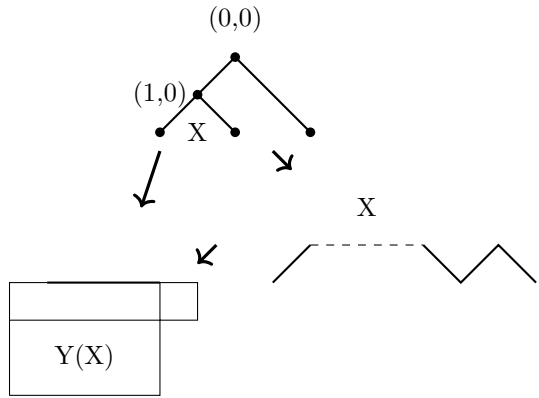
\begin{center}
    \begin{tikzpicture}[scale=.5]
    %Binary
    \draw[ thick] (0,0)--(2,2)--(4,0);
    \draw[thick] (1,1)--(2,0);
    \draw (0,0) [fill=black,thick] circle (1mm);
    \draw (1,1) [fill=black,thick] circle (1mm);
    \draw (2,0) [fill=black,thick] circle (1mm);
    \draw (4,0) [fill=black,thick] circle (1mm);
    \draw (2,2) [fill=black,thick] circle (1mm);
    \draw (1,0) node{X};
    \draw (2,3) node{(0,0)};
    \draw (0,1) node{(1,0)};
    %Dyck
    \begin{scope}[xshift=8cm,yshift=-4cm] 
    \draw[ thick] (-1,1)--(0,0)--(1,1)--(2,0);
    \draw[thick] (-5,0)--(-4,1);
    \draw[dashed](-4,1)--(-1,1);
    \draw (-2.5,2) node{X};
    \end{scope}
    %Young
    \begin{scope}[xshift=-4cm,yshift=-5cm]
    \draw (0,0)--(0,1)--(1,1);
    \draw (2,-1) node{Y(X)};
   % \draw (2,0.5) [fill=black,thick] circle (1mm);
    %\draw (2.5,0.5) [fill=black,thick] circle (1mm);
    %\draw (3,0.5) [fill=black,thick] circle (1mm);
    \draw (4,0)--(4,1)--(5,1)--(5,0)--(4,0);
    \draw (0,-2)--(0,0)--(4,0)--(4,-2)--(0,-2);
    \draw[thick] (1,1)--(4,1);
    \end{scope}
    \draw [->, very thick] (3,-0.5)--(3.5,-1);
    \draw [->, very thick] (1.5,-3)--(1,-3.5);
    \draw [<-, very thick] (-0.5,-2)--(0,-0.5);
    \end{tikzpicture}
\end{center}

The diagram commutes in this case if the subdiagrams $B(X) = D(X)$, denote as $Y(X)$.  

(3) Now consider the case where left subtree, X is empty and the right subtree, Y is non-empty.

Using the Dyck path mapping we get:

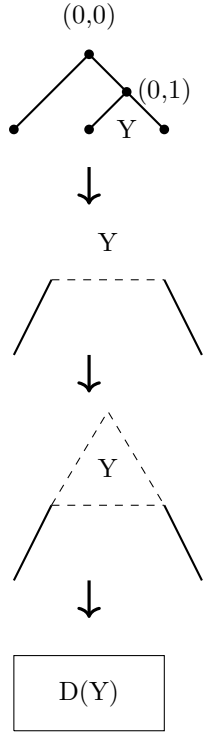
\begin{center}
    \begin{tikzpicture}[scale=.5]
    %Binary
    \begin{scope}[xshift=-3cm]
    \draw[ thick] (0,0)--(2,2)--(4,0);
    \draw[thick] (3,1)--(2,0);
    \draw (0,0) [fill=black,thick] circle (1mm);
    \draw (3,1) [fill=black,thick] circle (1mm);
    \draw (2,0) [fill=black,thick] circle (1mm);
    \draw (4,0) [fill=black,thick] circle (1mm);
    \draw (2,2) [fill=black,thick] circle (1mm);
    \draw (3,0) node{Y};
    \draw (2,3) node{(0,0)};
    \draw (4,1) node{(0,1)};
    \end{scope}
    \draw [->, very thick] (-1,-1)--(-1,-2);
    %Dyck
    \begin{scope}[yshift=-6cm, xshift=2cm] 
    \draw[ thick] (-1,2)--(0,0);
    \draw[thick] (-5,0)--(-4,2);
    \draw[dashed](-4,2)--(-1,2);
    \draw (-2.5,3) node{Y};
    \end{scope}
    \draw [->, very thick] (-1,-6)--(-1,-7);
    %Dyck
    \begin{scope}[yshift=-12cm, xshift=2cm] 
    \draw[ thick] (-1,2)--(0,0);
    \draw[thick] (-5,0)--(-4,2);
    \draw[dashed](-4,2)--(-1,2);
    \draw[dashed](-4,2)--(-2.5,4.5)--(-1,2);
    \draw (-2.5,3) node{Y};
    \end{scope}
    \draw [->, very thick] (-1,-12)--(-1,-13);
    %Young
    \begin{scope}[xshift=-3cm, yshift=-14cm]
    %\draw (0,0)--(0,1)--(1,1)--(1,0)--(0,0);
    \draw (2,-1) node{D(Y)};
    %\draw (2,0.5) [fill=black,thick] circle (1mm);
    %\draw (2.5,0.5) [fill=black,thick] circle (1mm);
   % \draw (3,0.5) [fill=black,thick] circle (1mm);
  %  \draw (4,0)--(4,1)--(5,1)--(5,0)--(4,0);
    \draw (0,-2)--(0,0)--(4,0)--(4,-2)--(0,-2);
    %\draw[thick] (1,1)--(4,1);
    \end{scope}
    \end{tikzpicture}
\end{center}

The Young diagram $D(Y)$ is determined only from the properties of the subtree $Y$.

Using the bookshelf mapping:

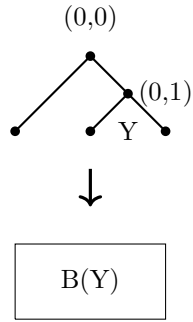
\begin{center}
    \begin{tikzpicture}[scale=.5]
    %Binary
    \begin{scope}[xshift=-3cm]
    \draw[ thick] (0,0)--(2,2)--(4,0);
    \draw[thick] (3,1)--(2,0);
    \draw (0,0) [fill=black,thick] circle (1mm);
    \draw (3,1) [fill=black,thick] circle (1mm);
    \draw (2,0) [fill=black,thick] circle (1mm);
    \draw (4,0) [fill=black,thick] circle (1mm);
    \draw (2,2) [fill=black,thick] circle (1mm);
    \draw (3,0) node{Y};
    \draw (2,3) node{(0,0)};
    \draw (4,1) node{(0,1)};
    \end{scope}
    \draw [->, very thick] (-1,-1)--(-1,-2);
    %Young
    \begin{scope}[xshift=-3cm, yshift=-3cm]
    %\draw (0,0)--(0,1)--(1,1)--(1,0)--(0,0);
    \draw (2,-1) node{B(Y)};
    %\draw (2,0.5) [fill=black,thick] circle (1mm);
    %\draw (2.5,0.5) [fill=black,thick] circle (1mm);
   % \draw (3,0.5) [fill=black,thick] circle (1mm);
  %  \draw (4,0)--(4,1)--(5,1)--(5,0)--(4,0);
    \draw (0,-2)--(0,0)--(4,0)--(4,-2)--(0,-2);
    %\draw[thick] (1,1)--(4,1);
    \end{scope}
    \end{tikzpicture}
\end{center}

Similarly, the Young diagram $B(Y)$ is determined only from the properties of the subtree $Y$ since there are no shelves outside of $B(Y)$.

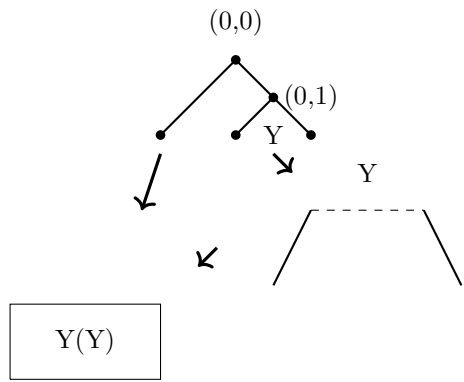
\begin{center}
    \begin{tikzpicture}[scale=.5]
    %Binary
    \draw[ thick] (0,0)--(2,2)--(4,0);
    \draw[thick] (3,1)--(2,0);
    \draw (0,0) [fill=black,thick] circle (1mm);
    \draw (3,1) [fill=black,thick] circle (1mm);
    \draw (2,0) [fill=black,thick] circle (1mm);
    \draw (4,0) [fill=black,thick] circle (1mm);
    \draw (2,2) [fill=black,thick] circle (1mm);
    \draw (3,0) node{Y};
    \draw (2,3) node{(0,0)};
    \draw (4,1) node{(0,1)};
    %Dyck
    \begin{scope}[xshift=8cm,yshift=-4cm] 
    \draw[ thick] (-1,2)--(0,0);
    \draw[thick] (-5,0)--(-4,2);
    \draw[dashed](-4,2)--(-1,2);
    \draw (-2.5,3) node{Y};
    \end{scope}
    %Young
    \begin{scope}[xshift=-4cm,yshift=-4.5cm]
    %\draw (0,0)--(0,1)--(1,1)--(1,0)--(0,0);
    \draw (2,-1) node{Y(Y)};
    %\draw (2,0.5) [fill=black,thick] circle (1mm);
    %\draw (2.5,0.5) [fill=black,thick] circle (1mm);
   % \draw (3,0.5) [fill=black,thick] circle (1mm);
  %  \draw (4,0)--(4,1)--(5,1)--(5,0)--(4,0);
    \draw (0,-2)--(0,0)--(4,0)--(4,-2)--(0,-2);
    %\draw[thick] (1,1)--(4,1);
    \end{scope}
    \draw [->, very thick] (3,-0.5)--(3.5,-1);
    \draw [->, very thick] (1.5,-3)--(1,-3.5);
    \draw [<-, very thick] (-0.5,-2)--(0,-0.5);
    \end{tikzpicture}
\end{center}
The diagram above commutes in this case where the Young diagram only consists of the corresponding subtree Y, if the subdiagrams $B(Y) = D(Y)$, denote as $Y(Y)$. Notice that if Y is the base case this results in an empty Young diagram.

(4) Now consider the left subtree, X is non-empty and the right subtree, Y is non-empty. \footnote{For a binary tree of length $n+m-1$, i.e. $n+m$ children notice that if the right end of X ends at child $(n,m-1)$ then the left end of the right subtree Y must end at child $(n-1,m)$. Otherwise, there would be a child not connected to any branch. Therefore, it must be that X has its root at $(n,0)$ and Y at $(
0,m)$.} 

Using the Dyck path mapping we get:

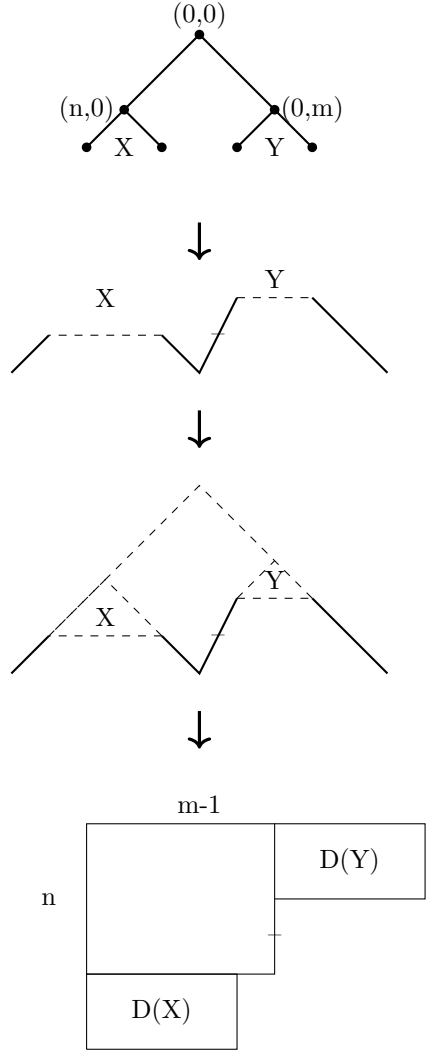
\begin{center}
    \begin{tikzpicture}[scale=.5]
    %Binary
    \begin{scope}[xshift=-6cm]
    \draw[ thick] (0,0)--(3,3)--(6,0);
    \draw[thick] (1,1)--(2,0);
    \draw [thick] (5,1)--(4,0);
    \draw (0,0) [fill=black,thick] circle (1mm);
    \draw (3,3) [fill=black,thick] circle (1mm);
    \draw (2,0) [fill=black,thick] circle (1mm);
    \draw (4,0) [fill=black,thick] circle (1mm);
    \draw (1,1) [fill=black,thick] circle (1mm);
    \draw (5,1) [fill=black,thick] circle (1mm);
    \draw (6,0) [fill=black,thick] circle (1mm);
    \draw (5,0) node{Y};
    \draw (1,0) node{X};
    \draw (3,3.5) node{(0,0)};
    \draw (0,1) node{(n,0)};
    \draw (6,1) node{(0,m)};
    \end{scope}
    \draw [->, very thick] (-3,-2)--(-3,-3);
    %Dyck
    \begin{scope}[xshift =2cm, yshift=-6cm] 
    \draw[ thick] (-2,2)--(0,0);
    \draw[thick] (-6,1)--(-5,0)--(-4,2);
    \draw[dashed](-4,2)--(-2,2);
    \draw[dashed](-6,1)--(-9,1);
    \draw [thick] (-9,1)--(-10,0);
    \draw (-3,2.5) node{Y};
    \draw (-7.5,2) node{X};
    \draw (-4.5, 1) node{--};
    %\draw (-4.5, 1) [fill=black] circle (1mm);
    \end{scope}
    \draw [->, very thick] (-3,-7)--(-3,-8);
    %Dyck
    \begin{scope}[xshift =2cm, yshift=-14cm] 
    \draw[ thick] (-2,2)--(0,0);
    \draw[thick] (-6,1)--(-5,0)--(-4,2);
    \draw[dashed](-4,2)--(-2,2);
    \draw[dashed](-6,1)--(-9,1);
    \draw[dashed](-6,1)--(-7.5,2.5)--(-9,1);
    \draw[dashed](-2,2)--(-5,5)--(-9,1);
    \draw[dashed](-4,2)--(-3,3);
    \draw [thick] (-9,1)--(-10,0);
    \draw (-3,2.5) node{Y};
    \draw (-7.5,1.5) node{X};
    \draw (-4.5, 1) node{--};
    %\draw (-4.5, 1) [fill=black] circle (1mm);
    \end{scope}
    \draw [->, very thick] (-3,-15)--(-3,-16);
    %Young
    \begin{scope}[xshift=-6cm,yshift=-22cm]
    \draw (2,-1) node{D(X)};
    \draw (0,-2)--(0,0)--(4,0)--(4,-2)--(0,-2);
    \draw (0,0)--(0,4)-- (5,4)--(5,0)--(0,0);
    \draw (5,4)--(9,4)--(9,2)--(5,2);
    \draw (7,3) node{D(Y)};
    \draw (-1,2) node{n};
    \draw (3,4.5) node{m-1};
    \draw (5, 1) node{--};
    \end{scope}
    \end{tikzpicture}
\end{center}

Notice that the rectangle between the 2 subdiagrams is of size $n\times(m-1)$. The width of the rectangle $(m-1)$ is the number of 'ups' in the Dyck path construction of X, i.e. the number of 'downs, which is given by the number of children in X not counting the last one. The height of the rectangle $n$ is in a similar fashion is given by Y since Y is a prime path we add one more level thus $n-1+1 = n$. 

By the Dyck path it is clear that the subdiagram determined by X can be at most one row shorter than the rectangle and similarly $D(Y)$ can only be at most 2 columns shorter than the rectangle.

Now using the bookshelf mapping:

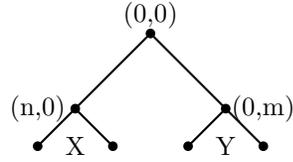
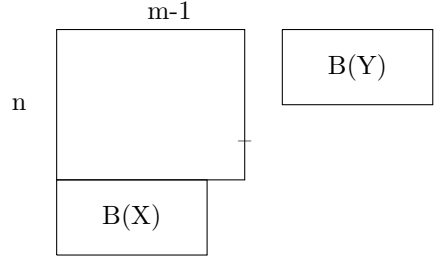
\begin{center}
    \begin{tikzpicture}[scale=.5]
    %Binary
    \begin{scope}[xshift=-6cm]
    \draw[ thick] (0,0)--(3,3)--(6,0);
    \draw[thick] (1,1)--(2,0);
    \draw [thick] (5,1)--(4,0);
    \draw (0,0) [fill=black,thick] circle (1mm);
    \draw (3,3) [fill=black,thick] circle (1mm);
    \draw (2,0) [fill=black,thick] circle (1mm);
    \draw (4,0) [fill=black,thick] circle (1mm);
    \draw (1,1) [fill=black,thick] circle (1mm);
    \draw (5,1) [fill=black,thick] circle (1mm);
    \draw (6,0) [fill=black,thick] circle (1mm);
    \draw (5,0) node{Y};
    \draw (1,0) node{X};
    \draw (3,3.5) node{(0,0)};
    \draw (0,1) node{(n,0)};
    \draw (6,1) node{(0,m)};
    \end{scope}
    \draw [->, very thick] (-3,-1)--(-3,-2);
        %Young
    \begin{scope}[xshift=-6cm,yshift=-8cm]
    \draw (2,-1) node{B(X)};
    \draw (0,-2)--(0,0)--(4,0)--(4,-2)--(0,-2);
    \draw (0,0)--(0,4)-- (5,4)--(5,0)--(0,0);
    \draw (6,4)--(10,4)--(10,2)--(6,2)--(6,4);
    \draw (8,3) node{B(Y)};
    \draw (-1,2) node{n};
    \draw (3,4.5) node{m-1};
    \draw (5, 1) node{--};
    \end{scope}
        \end{tikzpicture}
\end{center}
    The subdiagram $B(Y)$ is shifted by one column. By the gap elimination operation of the bookshelf construction we get the final Young diagram as follows:
    
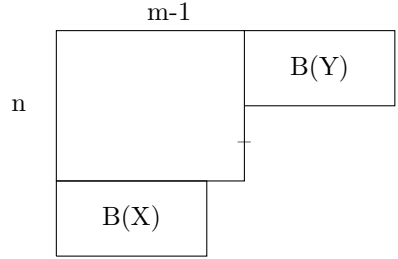
\begin{center}
    \begin{tikzpicture}[scale=.5]
    %Young
    \begin{scope}
    \draw (2,-1) node{B(X)};
    \draw (0,-2)--(0,0)--(4,0)--(4,-2)--(0,-2);
    \draw (0,0)--(0,4)-- (5,4)--(5,0)--(0,0);
    \draw (5,4)--(9,4)--(9,2)--(5,2);
    \draw (7,3) node{B(Y)};
    \draw (-1,2) node{n};
    \draw (3,4.5) node{m-1};
    \draw (5, 1) node{--};
    \end{scope}
    \end{tikzpicture}
\end{center}

Notice that after tilting the tree the rectangle between $B(X)$ and $B(Y)$ of size $n \times (m-1)$ is determined above X: height determined by the root node at $n$ and width by the ceiling of $B(X)$ that is $n+m-1 -n = (m-1)$. We now that no shelf can be as long as the ceiling of $B(X)$, this explains why the row length of the rectangle will be at least one more than that of $B(X)$. There are 2 row gaps between $B(X)$ and $B(Y)$. Similarly one gap is because of the zig-zag construction of a Young diagram another is because the column gap of the Bookshelf construction.\footnote{Each child has one parent and the right most child of X and left most child of Y have to be next to each other, thus creating one row and column gap. The other gap comes from the fact that a box does not fit at the bottom left of Y.}

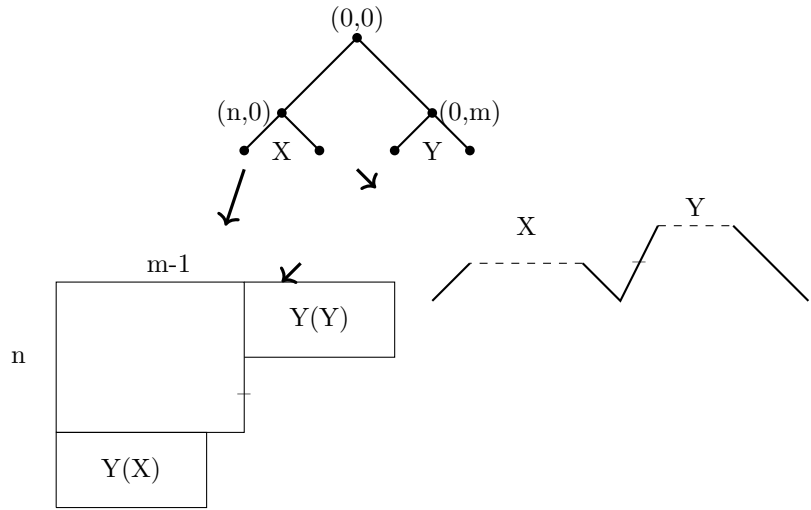
\begin{center}
    \begin{tikzpicture}[scale=.5]
    %Binary
    \draw[ thick] (0,0)--(3,3)--(6,0);
    \draw[thick] (1,1)--(2,0);
    \draw [thick] (5,1)--(4,0);
    \draw (0,0) [fill=black,thick] circle (1mm);
    \draw (3,3) [fill=black,thick] circle (1mm);
    \draw (2,0) [fill=black,thick] circle (1mm);
    \draw (4,0) [fill=black,thick] circle (1mm);
    \draw (1,1) [fill=black,thick] circle (1mm);
    \draw (5,1) [fill=black,thick] circle (1mm);
    \draw (6,0) [fill=black,thick] circle (1mm);
    \draw (5,0) node{Y};
    \draw (1,0) node{X};
    \draw (3,3.5) node{(0,0)};
    \draw (0,1) node{(n,0)};
    \draw (6,1) node{(0,m)};
    %Dyck
    \begin{scope}[xshift=15cm,yshift=-4cm] 
    \draw[ thick] (-2,2)--(0,0);
    \draw[thick] (-6,1)--(-5,0)--(-4,2);
    \draw[dashed](-4,2)--(-2,2);
    \draw[dashed](-6,1)--(-9,1);
    \draw [thick] (-9,1)--(-10,0);
    \draw (-3,2.5) node{Y};
    \draw (-7.5,2) node{X};
    \draw (-4.5, 1) node{--};
    %\draw (-4.5, 1) [fill=black] circle (1mm);
    \end{scope}
    %Young
    \begin{scope}[xshift=-5cm,yshift=-7.5cm]
    \draw (2,-1) node{Y(X)};
    \draw (0,-2)--(0,0)--(4,0)--(4,-2)--(0,-2);
    \draw (0,0)--(0,4)-- (5,4)--(5,0)--(0,0);
    \draw (5,4)--(9,4)--(9,2)--(5,2);
    \draw (7,3) node{Y(Y)};
    \draw (-1,2) node{n};
    \draw (3,4.5) node{m-1};
    \draw (5, 1) node{--};
    \end{scope}
    \draw [->, very thick] (3,-0.5)--(3.5,-1);
    \draw [->, very thick] (1.5,-3)--(1,-3.5);
    \draw [<-, very thick] (-0.5,-2)--(0,-0.5);
    \end{tikzpicture}
\end{center}
The diagram commutes in this case, if the subdiagrams $B(Y) = D(Y)$ and $B(X) = D(X)$ , denote as $Y(X)$. Notice that if X and Y is the base case this results in a rectangle shaped $n\times(m-1)$ Young diagram.

%There is a 2 row difference between X and Y in the Young diagram. We can see this in the Dyck path as the 2 lengthed gap between X and Y occurring due to the prime path. Similarly in the bookshelf construction each child has one parent and the right most child of X and left most child of Y have to be next to each other, thus creating one row gap. The other gap comes from the fact that a box does not fit at the bottom left of Y.

% 2 rows before Y and one column before X.

We have shown that all possible subtrees of any given binary tree maps to the same Young diagram either by the bookshelf construction or by the Dyck path assuming the subdiagrams are equal. Repeating the process for all subtrees using induction concludes that for any binary tree the two maps yield the same Young diagram.

\subsubsection{Conclusion}
We have shown that using the 'bookshelf construction', there exists a bijection between binary trees and Young diagram. Also we have shown the mappings from binary trees with Dyck path and the 'bookshelf' mapping yield the same Young diagram. Both proofs have made use of the recursive properties of binary trees and the partitioning of Young diagrams into subdiagrams.
% \begin{thebibliography}{}
% \bibitem{nelson}
% Nelson, L. (2016). \textit{Toward the enumeration of maximal chains in the Tamari lattices}. Arizona State University.
% \end{thebibliography}

% appendix moved to appendix.tex
\subsection{Appendix}
Here we will briefly describe the proof of the bijection between binary trees and Young diagrams using the bookshelf mapping. 

Notice that there are 4 possible cases of subtrees that define the recursive behavior of binary trees:

\begin{enumerate}
    \item $X= \emptyset$ and $Y=\emptyset$
    \item $X= \emptyset$ and $Y\neq \emptyset$
    \item $X\neq \emptyset$ and $Y = \emptyset$
    \item $X\neq \emptyset$ and $Y \neq \emptyset$
    \end{enumerate}
    
    To prove the bijection we will show that indeed the 'inverse bookshelf' construction is the inverse of the bookshelf construction.

    For a given Young diagram we will take the inverse bookshelf then take the bookshelf construction for each case. If the end result is the same as the given Young diagram, it will show that every Young Diagram is given by the BookShelf construction since the two sets have the same size, so the map is a bijection.  
    
(1) The base case is trivial. 

    
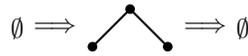
\begin{center}
    \begin{tikzpicture}[scale=.5]
    \draw (0,0) node{$\emptyset$};
    \draw (1,0) node{$\Longrightarrow$};
    \draw[thick] (2,-0.5)--(3,0.5)--(4,-0.5);
    \draw (2,-0.5) [fill=black,thick] circle (1mm);
    \draw (3,0.5) [fill=black,thick] circle (1mm);
    \draw (4,-0.5) [fill=black,thick] circle (1mm);
    \draw (5,0) node{$\Longrightarrow$};
    \draw (6,0) node{$\emptyset$};
    \end{tikzpicture}
    \end{center}
    
(2) Consider the right $Y=\emptyset$:

Notice that this has the following shape of a Young diagram:
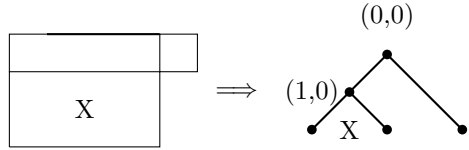
\begin{center}
    \begin{tikzpicture}[scale=.5]
    \draw (0,0)--(0,1)--(1,1);
    \draw (2,-1) node{X};
  % \draw (2,0.5) [fill=black,thick] circle (1mm);
    %\draw (2.5,0.5) [fill=black,thick] circle (1mm);
    %\draw (3,0.5) [fill=black,thick] circle (1mm);
    \draw (4,0)--(4,1)--(5,1)--(5,0)--(4,0);
    \draw (0,-2)--(0,0)--(4,0)--(4,-2)--(0,-2);
    \draw[thick] (1,1)--(4,1);
    \end{tikzpicture}
    \begin{tikzpicture}[scale=.5]
    \draw (-2,1) node{$\Longrightarrow$}; 
    \draw[ thick] (0,0)--(2,2)--(4,0);
    \draw[thick] (1,1)--(2,0);
    \draw (0,0) [fill=black,thick] circle (1mm);
    \draw (1,1) [fill=black,thick] circle (1mm);
    \draw (2,0) [fill=black,thick] circle (1mm);
    \draw (4,0) [fill=black,thick] circle (1mm);
    \draw (2,2) [fill=black,thick] circle (1mm);
    \draw (1,0) node{X};
    \draw (2,3) node{(0,0)};
    \draw (0,1) node{(1,0)};
    \end{tikzpicture}
\end{center}
    Applying the \emph{Inverse Bookshelf Mapping} algorithm, replace the first row with shelf, which is the only row that has length $\lambda_{1}$, and the rest of the part that's below the first row will be the only sub-diagram X. This explicitly shows that the right part cannot have a sub-tree.\\
    \\
    \\
    Notice that this is a spacial case of the inverse mapping for $\lambda_1+n+1 \geq \lambda_1'+k+1$ where $n=1$. If $n>1$ we would have a subtree on the right with no corresponding boxes but rather positive $n-1$ many sloped lines to satisfy the definition of a binary tree.

(3) Consider the case left $X=\emptyset$:\\
suppose we have the following transformation: 
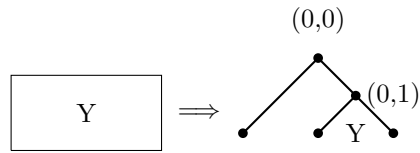
\begin{center}
\begin{tikzpicture}[scale=.5]
\draw (0,-2)--(0,0)--(4,0)--(4,-2)--(0,-2);
\draw (2,-1) node{Y};
\draw (5,-1) node{$\Longrightarrow$};
\end{tikzpicture}
\begin{tikzpicture}[scale=.5]
    \draw[ thick] (0,0)--(2,2)--(4,0);
    \draw[thick] (3,1)--(2,0);
    \draw (0,0) [fill=black,thick] circle (1mm);
    \draw (3,1) [fill=black,thick] circle (1mm);
    \draw (2,0) [fill=black,thick] circle (1mm);
    \draw (4,0) [fill=black,thick] circle (1mm);
    \draw (2,2) [fill=black,thick] circle (1mm);
    \draw (3,0) node{Y};
    \draw (2,3) node{(0,0)};
    \draw (4,1) node{(0,1)};
\end{tikzpicture}
\end{center}
In our inverse bookshelf construction such a case will never occur. This is because we have defined the corresponding binary tree to have size such that putting the top-left corner of the diagram fits the appropriate fixed tree exactly. Thus we will never get a tree such as: 

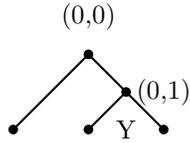
\begin{center}
    \begin{tikzpicture}[scale=.5]
        \draw[ thick] (0,0)--(2,2)--(4,0);
    \draw[thick] (3,1)--(2,0);
    \draw (0,0) [fill=black,thick] circle (1mm);
    \draw (3,1) [fill=black,thick] circle (1mm);
    \draw (2,0) [fill=black,thick] circle (1mm);
    \draw (4,0) [fill=black,thick] circle (1mm);
    \draw (2,2) [fill=black,thick] circle (1mm);
    \draw (3,0) node{Y};
    \draw (2,3) node{(0,0)};
    \draw (4,1) node{(0,1)};
        \end{tikzpicture}
    \end{center}
    
(4) Consider both subtrees are non-empty:
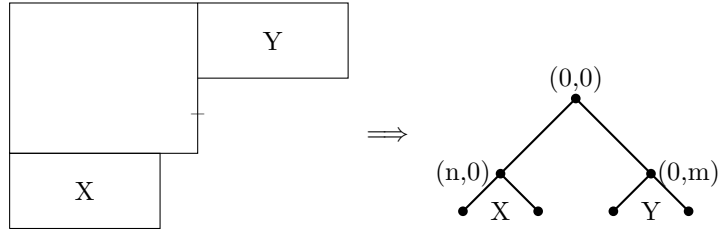
\begin{center}
 \begin{tikzpicture}[scale=.5]
    \begin{scope}[xshift=-5cm,yshift=-7.5cm]
    \draw (2,-1) node{X};
    \draw (0,-2)--(0,0)--(4,0)--(4,-2)--(0,-2);
    \draw (0,0)--(0,4)-- (5,4)--(5,0)--(0,0);
    \draw (5,4)--(9,4)--(9,2)--(5,2);
    \draw (7,3) node{Y};
    \draw (5, 1) node{--};
    \end{scope}
    \end{tikzpicture}
    \begin{tikzpicture}[scale=.5]
    \draw (-2,2) node{$\Longrightarrow$};
    \draw[ thick] (0,0)--(3,3)--(6,0);
    \draw[thick] (1,1)--(2,0);
    \draw [thick] (5,1)--(4,0);
    \draw (0,0) [fill=black,thick] circle (1mm);
    \draw (3,3) [fill=black,thick] circle (1mm);
    \draw (2,0) [fill=black,thick] circle (1mm);
    \draw (4,0) [fill=black,thick] circle (1mm);
    \draw (1,1) [fill=black,thick] circle (1mm);
    \draw (5,1) [fill=black,thick] circle (1mm);
    \draw (6,0) [fill=black,thick] circle (1mm);
    \draw (5,0) node{Y};
    \draw (1,0) node{X};
    \draw (3,3.5) node{(0,0)};
    \draw (0,1) node{(n,0)};
    \draw (6,1) node{(0,m)};
    \end{tikzpicture}
\end{center}
By the \emph{Inverse-Bookshelf Mapping}, we equivalently take out the maximum rectangle of size (i,j) in a given Young Diagram Z where $i$ is the row satisfying $\lambda_i>\lambda_{i+1}$ and $j$ is the column satisfying $\lambda_j'>\lambda_{j+1}+1$. Thus this determines the rectangle between X and Y.

All of the part on the right of Z is Y and the part below the Z is X.  
\\
This describes any Young diagram that has column shifting property defined as $\lambda'_1+k+1> \lambda_1+n+1$ that plots the diagrams into 2 subdiagrams thus creating the left and the right tree. 

Another possibility for this shape is that a diagram with multiple initial rows of maximal length of the same size, i.e.  $\lambda_1+n+1 \geq \lambda_1'+k+1$ where $n>1$ (recall our example for the inverse function with rows marked in red), thus requiring positive sloped subtree on the right to satisfy the definition of a binary tree. This tree will reduce to the base case (1).

%\newpage

%Next, we recall the definition of {Dyck path} and discuss the well-known bijections between Dyck paths and binary trees and Nelson's bijection between Dyck paths and Young diagrams. Nelson's construction is simply to ``fill in'' the space above the Dyck path. We show that these two bijections together give the same bijection as our ``bookshelf'' construction.

%Torsion classes are discussed next. These correspond to binary trees using \ul{bookshelves} (without stacking). 

%\input{Yicheng.tex} %
% Yicheng Tao contribution for Math 47a joint project

%\author{Yicheng Tao}

\chapter*{A bijection between binary trees and torsion classes}

\centerline{Yicheng Tao}

\begin{abstract}
We show the definitions and descriptions of torsion pairs, torsion classes, and torsion-free classes. We give a bijection between the set of binary trees with $n+1$ leaves and the set of torsion classes in $A_{n-1}$ which makes the descending edges in the binary tree correspond to the torsion objects in the category.
\end{abstract}

\bigskip

%\maketitle

%{Introduction}
There are various sets having a Catalan number of objects, such as Young diagrams, Dyck paths, binary trees, torsion classes, and 213 permutations\cite{ref1}. In this paper, we focus on binary trees and torsion classes. In Section 2, we provide a quick introduction to binary trees, ball structures, and zig-zag diagrams. In Section 3, we show the definitions and detailed descriptions of torsion pairs, torsion classes, and torsion-free classes. In Section 4, we first prove that every binary tree gives a torsion pair in Theorem 4.5 and then give a bijection between the set of binary trees with $n+1$ leaves and the set of torsion classes in $A_{n-1}$ in Theorem 4.6.

\subsection{Binary trees, ball structures, and zig-zag diagrams}
\begin{define}
	A \emph{binary tree} (sometimes called a \emph{``full binary tree"}) is a rooted tree in which every internal vertex has exactly two children\cite{ref2}.
\end{define}
In this paper, we put balls into the empty spots of binary trees and always consider binary trees with balls without mentioning it. The balls of a binary tree form a \emph{ball structure}. Binary trees with the same number of leaves have the same ball structure. We also consider the \emph{zig-zag diagram} of a ball structure, which uses dashed lines to encompass every ball. Figure 1 shows an example of them. 
\begin{figure}[htb]
	\centering
	\begin{minipage}[t]{0.36\linewidth}
		\centering
		\begin{tikzpicture}[scale=1]
		\draw (1,0) node[circle,draw,minimum size=4mm] {};
		\draw (2,0) node[circle,draw,minimum size=4mm] {};
		\draw (3,0) node[circle,draw,minimum size=4mm] {};
		\draw (1.5,0.5) node[circle,draw,minimum size=4mm] {};
		\draw (2.5,0.5) node[circle,draw,minimum size=4mm] {};
		\draw (2,1) node[circle,draw,minimum size=4mm] {};
		\draw (2,1.5)--(0,-0.5) (2,1.5)--(4,-0.5) (1,0.5)--(2,-0.5) (0.5,0)--(1,-0.5) (3.5,0)--(3,-0.5);
		\end{tikzpicture}
	\end{minipage}
	\begin{minipage}[t]{0.36\linewidth}
		\centering
		\begin{tikzpicture}[scale=1]
		\draw (1,0) node[circle,draw,minimum size=4mm] {};
		\draw (2,0) node[circle,draw,minimum size=4mm] {};
		\draw (3,0) node[circle,draw,minimum size=4mm] {};
		\draw (1.5,0.5) node[circle,draw,minimum size=4mm] {};
		\draw (2.5,0.5) node[circle,draw,minimum size=4mm] {};
		\draw (2,1) node[circle,draw,minimum size=4mm] {};
		\draw[dashed] (2,1.5)--(0,-0.5) (2,1.5)--(4,-0.5) (2.5,1)--(1,-0.5) (1.5,1)--(3,-0.5) (3,0.5)--(2,-0.5) (1,0.5)--(2,-0.5) (0.5,0)--(1,-0.5) (3.5,0)--(3,-0.5);
		\end{tikzpicture}
	\end{minipage}
	\caption{A binary tree and its ball structure with the zig-zag diagram.}
\end{figure}
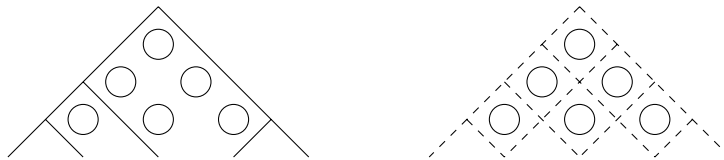

\subsection{Torsion pairs, torsion classes, and torsion-free classes}
\begin{define}\label{def of torsion pair}
	A \emph{torsion pair} in an abelian category $\mathcal A$ is a pair of full subcategories  $(\mathcal{G},\mathcal{F})$ having the following two properties\cite{ref3}.
	\begin{enumerate}
		\item $\mathcal G$ is the class of all objects $X$ so that $\Hom(X,Y)=0$ for all $Y\in\mathcal{F}$;
		\item $\mathcal F$ is the class of all objects $Y$ so that $\Hom(X,Y)=0$ for all $X\in\mathcal{G}$.
	\end{enumerate}
\end{define}

\begin{define}
	A \emph{torsion class} is the subcategory $\mathcal{G}$ in a torsion pair $(\mathcal{G},\mathcal{F})$\cite{ref3}.
\end{define}

\begin{define}
	A \emph{torsion-free class} is the subcategory $\mathcal{F}$ in a torsion pair $(\mathcal{G},\mathcal{F})$\cite{ref3}.
\end{define}

It is immediate from Definition \ref{def of torsion pair} that $\mathcal G$, $\mathcal F$ determine each other.

As we only consider torsion pairs whose objects are balls in ball structures, there is a really easy way to see what are the $\Hom$'s.

\begin{lem}
	For two balls $X$ and $Y$ in a ball structure, $\Hom(X,Y)\neq0$ if and only if the entire rectangle with the left corner $X$ and the right corner $Y$ lies in the zig-zag diagram of the ball structure.
\end{lem}

This is a well-known lemma. Examples below with Figure 2 show how it works.

\begin{enumerate}
	\item $\Hom(X_1,Z_2)\neq0$;
	\item $\Hom(X,X)\neq0$ for any $X$;
	\item $\Hom(X_1,Y_3)=0$ since the lower end of the rectangle is missing;
	\item $\Hom(Y_2,Y_1)=0$ since $Y_2$ and $Y_1$ are at the wrong corners of the rectangle;
	\item $\Hom(X_2,W)=0$ since $X_2$ and $W$ are at the wrong corners of the rectangle.
\end{enumerate}

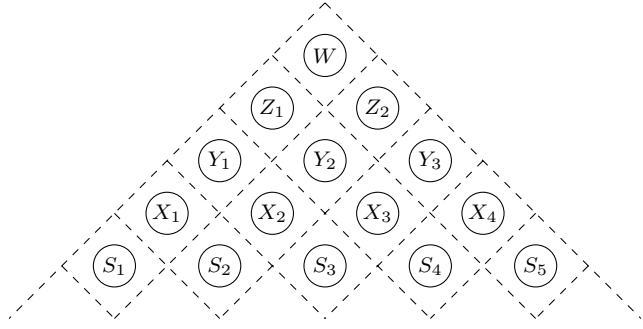
\begin{figure}[htb]
	\centering
	\begin{tikzpicture}[scale=1.4]
	\draw (1,0) node[circle,draw,minimum size=5.6mm] {};
	\node at (1,0) {\footnotesize{$S_1$}};
	\draw (2,0) node[circle,draw,minimum size=5.6mm] {};
	\node at (2,0) {\footnotesize{$S_2$}};
	\draw (3,0) node[circle,draw,minimum size=5.6mm] {};
	\node at (3,0) {\footnotesize{$S_3$}};
	\draw (4,0) node[circle,draw,minimum size=5.6mm] {};
	\node at (4,0) {\footnotesize{$S_4$}};
	\draw (5,0) node[circle,draw,minimum size=5.6mm] {};
	\node at (5,0) {\footnotesize{$S_5$}};
	\draw (1.5,0.5) node[circle,draw,minimum size=5.6mm] {};
	\node at (1.5,0.5) {\footnotesize{$X_1$}};
	\draw (2.5,0.5) node[circle,draw,minimum size=5.6mm] {};
	\node at (2.5,0.5) {\footnotesize{$X_2$}};
	\draw (3.5,0.5) node[circle,draw,minimum size=5.6mm] {};
	\node at (3.5,0.5) {\footnotesize{$X_3$}};
	\draw (4.5,0.5) node[circle,draw,minimum size=5.6mm] {};
	\node at (4.5,0.5) {\footnotesize{$X_4$}};
    \draw (2,1) node[circle,draw,minimum size=5.6mm] {};
    \node at (2,1) {\footnotesize{$Y_1$}};
	\draw (3,1) node[circle,draw,minimum size=5.6mm] {};
	\node at (3,1) {\footnotesize{$Y_2$}};
	\draw (4,1) node[circle,draw,minimum size=5.6mm] {};
	\node at (4,1) {\footnotesize{$Y_3$}};
	\draw (2.5,1.5) node[circle,draw,minimum size=5.6mm] {};
	\node at (2.5,1.5) {\footnotesize{$Z_1$}};
	\draw (3.5,1.5) node[circle,draw,minimum size=5.6mm] {};
	\node at (3.5,1.5) {\footnotesize{$Z_2$}};
	\draw (3,2) node[circle,draw,minimum size=5.6mm] {};
	\node at (3,2) {\footnotesize{$W$}};
	\draw[dashed] (3,2.5)--(0,-0.5) (3,2.5)--(6,-0.5) (3.5,2)--(1,-0.5) (2.5,2)--(5,-0.5) (4,1.5)--(2,-0.5) (2,1.5)--(4,-0.5) (4.5,1)--(3,-0.5) (1.5,1)--(3,-0.5) (5,0.5)--(4,-0.5) (1,0.5)--(2,-0.5) (5.5,0)--(5,-0.5) (0.5,0)--(1,-0.5);
	\end{tikzpicture}
	\caption{A ball structure with the zig-zag diagram.}
\end{figure}

With this lemma, we can determine torsion pairs in a ball structure. For example, in Figure 2 $(\{S_1,S_3,S_5\},\{S_2,X_2,Y_2,Z_2,S_4,X_4\})$ is a torsion pair, where $\{S_1,S_3,S_5\}$ is a torsion class and $\{S_2,X_2,Y_2,Z_2,S_4,X_4\}$ is the corresponding torsion-free class.

\subsection{A bijection between binary trees and torsion classes}

Having shown what binary trees and torsion classes are, we give a bijection between them.

\begin{define} Let $\mathcal D_T$ denote the set of all balls on the descending edges of a binary tree $T$.
\end{define}

\begin{define}
	Let $\mathcal{A}_T$ denote the set of all balls on the ascending edges of a binary tree $T$.
\end{define}

\begin{figure}[htb]
	\centering
	\begin{tikzpicture}[scale=1]
	\draw (1,0) node[circle,draw,fill=blue!50!white,minimum size=4mm] {};
	\draw (2,0) node[circle,draw,fill=blue!50!white,minimum size=4mm] {};
	\draw (3,0) node[circle,draw,fill=red!50!white,minimum size=4mm] {};
	\draw (1.5,0.5) node[circle,draw,fill=blue!50!white,minimum size=4mm] {};
	\draw (2.5,0.5) node[circle,draw,minimum size=4mm] {};
	\draw (2,1) node[circle,draw,minimum size=4mm] {};
	\draw (2,1.5)--(0,-0.5) (2,1.5)--(4,-0.5) (1,0.5)--(2,-0.5) (0.5,0)--(1,-0.5) (3.5,0)--(3,-0.5);
	\end{tikzpicture}
	\caption{$\mathcal{D}_T$ (blue) and $\mathcal{A}_T$ (red)}
\end{figure}
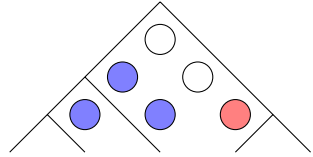

\begin{remark}
Every ball on the bottom row of a binary tree $T$ is either in $\mathcal D_T$ or in $\mathcal A_T$ since there must be an edge to the leaf below it.
\end{remark}

\begin{lem}
	Given any binary tree $T$, $\Hom(X,Y)=0$ for all $X\in\mathcal{D}_T$ and $Y\in\mathcal{A}_T$.
\end{lem}

\begin{proof}
	For any $X\in\mathcal{D}_T$, make the largest rectangle with the left corner $X$ that lies in the corresponding zig-zag diagram. Then for any ball $Y_1$ inside the rectangle, $\Hom(X,Y_1)\neq0$; for any ball $Y_2$ outside the rectangle, $\Hom(X,Y_2)=0$. If there is $Y\in\mathcal{A}_T$ inside the rectangle, the ascending edge $Y$ sits on will cross the descending edge $X$ sits on, which is not allowed in binary trees. So, any $Y\in\mathcal{A}_T$ is outside the rectangle, which means that $\Hom(X,Y)=0$. Therefore, $\Hom(X,Y)=0$ for all $X\in\mathcal{D}_T$ and $Y\in\mathcal{A}_T$. The lemma holds.
\end{proof}

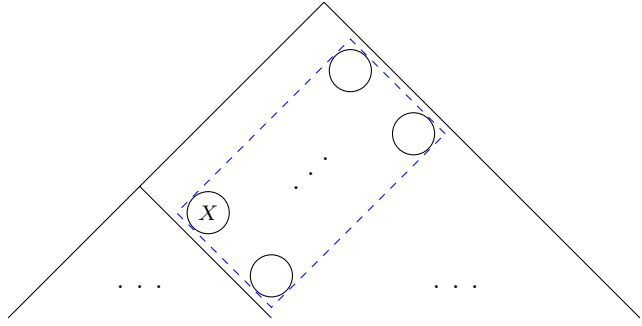
\begin{figure}[htb]
	\centering
	\begin{tikzpicture}[scale=1.4]
	\draw (1.9,0.5) node[circle,draw,minimum size=5.6mm] {};
	\node at (1.9,0.5) {\footnotesize{$X$}};
	\draw (3.25,1.85) node[circle,draw,minimum size=5.6mm] {};
	\draw (3.85,1.25) node[circle,draw,minimum size=5.6mm] {};
	\draw (2.5,-0.1) node[circle,draw,minimum size=5.6mm] {};
	\draw (1.25,-0.2) node[] {.   .   .};
	\draw (4.25,-0.2) node[] {.   .   .};
	\draw (2.875,0.875) node[] {.};
	\draw (3.015,1.015) node[] {.};
	\draw (2.735,0.735) node[] {.};
	\draw (3,2.5)--(0,-0.5) (3,2.5)--(6,-0.5) (1.25,0.75)--(2.5,-0.5);
	\draw[dashed,color=blue] (4.15,1.25)--(3.25,2.15)--(1.6,0.5)--(2.5,-0.4)--(4.15,1.25);
	\end{tikzpicture}
	\caption{The largest rectangle (blue) with the left corner $X$.}
\end{figure}

\begin{lem}
	Given any binary tree $T$, the following statements hold.
	\begin{enumerate}
		\item For any $Z\notin\mathcal{A}_T$, there exists $X\in\mathcal{D}_T$ so that $\Hom(X,Z)\neq0$;
		\item For any $Z\notin\mathcal{D}_T$, there exists $Y\in\mathcal{A}_T$ so that $\Hom(Z,Y)\neq0$.
	\end{enumerate}
\end{lem}

\begin{proof}
	For any $Z\notin\mathcal{A}_T$, consider the lowest ball $Z_1$ in the same ascending row as $Z$. Then $\Hom(Z_1,Z)\neq0$. By Remark, $Z_1$ is either in $\mathcal D_T$ or in $\mathcal A_T$. If $Z_1\in\mathcal{D}_T$, then (1) holds. Otherwise, $Z_1\in\mathcal{A}_T$. Since $Z\notin\mathcal{A}_T$, there must be a descending edge that blocks the ascending edge $Z_1$ sits on from going to $Z$, on which we can find a ball $Z_1'$ in the same ascending row as $Z$ and $Z_1$. Since $Z_1'\in\mathcal{D}_T$ and $\Hom(Z_1',Z)\neq0$, (1) holds. Therefore, (1) holds. 
	
	For any $Z\notin\mathcal{D}_T$, consider the lowest ball $Z_2$ in the same descending row as $Z$. Then $\Hom(Z,Z_2)\neq0$. By Remark, $Z_2$ is either in $\mathcal D_T$ or in $\mathcal A_T$. If $Z_2\in\mathcal{A}_T$, then (2) holds. Otherwise, $Z_2\in\mathcal{D}_T$. Since $Z\notin\mathcal{D}_T$, there must be an ascending edge that blocks the descending edge $Z_2$ sits on from going to $Z$, on which we can find a ball $Z_2'$ in the same descending row as $Z$ and $Z_2$. Since $Z_2'\in\mathcal{A}_T$ and $\Hom(Z,Z_2')\neq0$, (2) holds. Therefore, (2) holds.
\end{proof}

\begin{figure}[htb]
	\centering
	\begin{tikzpicture}[scale=1.4]
	\draw (3,1.5) node[circle,draw,minimum size=5.6mm] {};
	\node at (3,1.5) {\footnotesize{$Z$}};
	\draw (1.3,-0.2) node[circle,draw,minimum size=5.6mm] {};
	\node at (1.3,-0.2) {\footnotesize{$Z_1$}};
	\draw (4.7,-0.2) node[circle,draw,minimum size=5.6mm] {};
	\node at (4.7,-0.2) {\footnotesize{$Z_2$}};
	\draw (1.9,0.4) node[circle,draw,minimum size=5.6mm] {};
	\node at (1.9,0.4) {\footnotesize{$Z_1'$}};
	\draw (4.1,0.4) node[circle,draw,minimum size=5.6mm] {};
	\node at (4.1,0.4) {\footnotesize{$Z_2'$}};
	\draw (3,2.5)--(0,-0.5) (3,2.5)--(6,-0.5); 
	\draw[color=blue] (1.6,0.4)--(2.2,-0.2) (4.4,0.4)--(3.8,-0.2);
	\draw[dashed,color=blue] (1.25,0.75)--(1.6,0.4) (2.2,-0.2)--(2.5,-0.5) (4.75,0.75)--(4.4,0.4) (3.8,-0.2)--(3.5,-0.5);
	\draw[dashed] (3.35,2.15)--(1,-0.2)--(1.3,-0.5) (1.9,0.1)--(3.65,1.85);
	\draw[dashed] (2.65,2.15)--(5,-0.2)--(4.7,-0.5) (4.1,0.1)--(2.35,1.85);
	\draw (1.3,-0.5)--(1.9,0.1) (4.7,-0.5)--(4.1,0.1);
	\end{tikzpicture}
	\caption{The ascending/descending edge is blocked by another edge (blue).}
\end{figure}
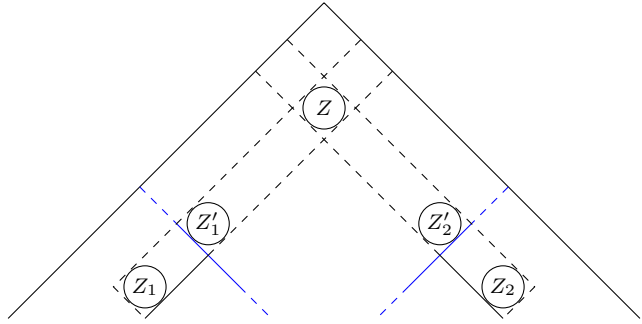

\begin{thm}
	Given any binary tree $T$, the pair of subsets given by $\mathcal{G}=\mathcal{D}_T$ and $\mathcal{F}=\mathcal{A}_T$ form a torsion pair.
\end{thm}

\begin{proof}
	By Lemma 4.3, we know that for any $X\in\mathcal{D}_T$, $\Hom(X,Y)=0$ for all $Y\in\mathcal{A}_T$. By Lemma 4.4, we know that for any $X\notin\mathcal{D}_T$, there exists $Y\in\mathcal{A}_T$ so that $\Hom(X,Y)\neq0$. So, $\mathcal{D}_T$ is the set of all objects $X$ so that $\Hom(X,Y)=0$ for all $Y\in\mathcal{A}_T$. On the other hand, by Lemma 4.3, we know that for any $Y\in\mathcal{A}_T$, $\Hom(X,Y)=0$ for all $X\in\mathcal{D}_T$. By Lemma 4.4, we know that for any $Y\notin\mathcal{A}_T$, there exists $X\in\mathcal{D}_T$ so that $\Hom(X,Y)\neq0$. So, $\mathcal{A}_T$ is the set of all objects $Y$ so that $\Hom(X,Y)=0$ for all $X\in\mathcal{D}_T$. Therefore, according to Definition \ref{def of torsion pair}, the pair of subsets given by $\mathcal{G}=\mathcal{D}_T$ and $\mathcal{F}=\mathcal{A}_T$ form a torsion pair. The theorem holds.
\end{proof}

\begin{thm}
	There exists a bijection $f$ from the set of binary trees with $n+1$ leaves to the set of torsion classes in $A_{n-1}$ so that $f(T)=\mathcal{D}_T$. Furthermore, this bijection is order reversing, i.e., if one binary tree is less than another the torsion class for the smaller tree contains that of the larger tree.
\end{thm}

\begin{proof}
	By theorem 4.5, it is obvious to see that there exists a map $f$ from the set of binary trees with $n+1$ leaves to the set of torsion classes in $A_{n-1}$ so that $f(T)=\mathcal{D}_T$. We prove $f$ is a bijection as follows. For two distinct binary trees $T_1$ and $T_2$ with $n+1$ leaves, $T_1$ must have an edge that $T_2$ does not have. If it is a descending edge, then $f(T_1)=\mathcal{D}_{T_1}\neq{\mathcal{D}_{T_2}=f(T_2)}$. If it is an ascending edge, then $\mathcal{A}_{T_1}\neq{\mathcal{A}_{T_2}}$. Assume that $\mathcal{D}_{T_1}=\mathcal{D}_{T_2}$. By Theorem 4.5 and the fact that $T_1$ and $T_2$ both with $n+1$ leaves, $(\mathcal{D}_{T_1},\mathcal{A}_{T_1})$ and  $(\mathcal{D}_{T_2},\mathcal{A}_{T_2})$ are torsion pairs defined on the same ball structure. By Definition \ref{def of torsion pair}, $\mathcal{A}_{T_1}$ and $\mathcal{A}_{T_2}$ are both the class of all objects $Y$ so that $\Hom(X,Y)=0$ for all $X\in\mathcal{D}_{T_1}$, so $\mathcal{A}_{T_1}={\mathcal{A}_{T_2}}$, which contradicts $\mathcal{A}_{T_1}\neq{\mathcal{A}_{T_2}}$. Thus, the assumption is false and we have $f(T_1)=\mathcal{D}_{T_1}\neq{\mathcal{D}_{T_2}=f(T_2)}$. To sum up, $f(T_1)\neq{f(T_2)}$, which gives that $f$ is an injection. Moreover, it is well known that there are a Catalan number $C_n=\frac{1}{n+1}\tbinom{2n}{n}$ of binary trees with $n+1$ leaves\cite{ref1} as well as torsion classes in $A_{n-1}$\cite{ref3}. Therefore, $f$ is a bijection from the set of binary trees with $n+1$ leaves to the set of torsion pairs in $A_{n-1}$. The theorem holds. 
	
	Finally, to show that this bijection is order reversing, note that the ordering of binary trees is given by shifting descending edges to the right. In that case the ``balls fall off'' of this edge and the torsion class becomes strictly smaller. So, the bijection is order reversing.
\end{proof}

%\newpage

%\input{Max.tex}
% Max's contribution to joint project

\newpage

\chapter*{Torsion Classes and Binary Trees}

\centerline{Max Weinstein}
%\author{Max Weinstein}
%\date{December 2019}

\begin{abstract}
This chapter will define and give examples of torsion classes and their respective torsion free classes as they appear in relation to other Catalan objects. Furthermore a clear bijection will be given between torsion classes and binary trees. 
\end{abstract}

\bigskip

\subsection{Definitions}

\subsubsection{Torsion Classes}
\begin{define}
A {\it torsion class} is collection of objects in a category (for our purposes we will only consider the straight categories $\mathcal{A}_n$) which are closed under extension and quotients. \\
\end{define}

 It is useful to think geometrically of the category as being a collection of balls in an equilateral triangle configuration with objects in the torsion class being colored in blue. The property of being closed under extension simply means that if two objects (balls) are in the torsion class, then if a rectangle with vertices at these two dips at most 1 imaginary ball below the bottom of the configuration, the other vertices within the configuration are also in the torsion class.\\
 \\
\begin{center}
\begin{tikzpicture}[scale=0.5]
%\draw[help lines=.1] (-4,-4) grid (4,4);

%  \draw circle[radius=.2];
 
% \begin{scope}[xshift=.6cm]
%  \draw[fill=blue] circle[radius=.2];
% \end{scope}
 
%  \begin{scope}[xshift=1.2cm]
%  \draw circle[radius=.2];
% \end{scope}

% \begin{scope}[xshift=.3cm, yshift=.6cm]
%  \draw[fill=blue] circle[radius=.2];
% \end{scope}
 
%  \begin{scope}[xshift=.9cm, yshift=.6cm]
%  \draw circle[radius=.2];
% \end{scope}

% \begin{scope}[xshift=.6cm, yshift=1.2cm]
%  \draw circle[radius=.2];
% \end{scope}

\ball{0}{-1}{0.6}
\ball{1}{-2}{0.6}
\ball{-1}{-2}{0.6}[blue]
\ball{-2}{-3}{0.6}
\ball{0}{-3}{0.6}[blue]
\ball{2}{-3}{0.6}
\end{tikzpicture}
\end{center}

\subsubsection{Torsion-Free Class}
\begin{define}
The corresponding {\it torsion-free class} is the collection of balls in the configuration which have the property that there are no homomorphisms from any balls in the torsion class to them. In this case homomorphisms exist if there is a way to draw a complete rectangle within the configuration going up and right from a torsion class ball to another ball. 
\end{define}
We will color objects in the torsion-free class red:
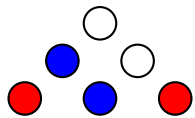
\begin{center}
\begin{tikzpicture}[scale=0.5]
%\draw[help lines=.1] (-4,-4) grid (4,4);

%  \draw[fill=red] circle[radius=.2];
 
% \begin{scope}[xshift=.6cm]
%  \draw[fill=blue] circle[radius=.2];
% \end{scope}
 
%  \begin{scope}[xshift=1.2cm]
%  \draw[fill=red] circle[radius=.2];
% \end{scope}

% \begin{scope}[xshift=.3cm, yshift=.6cm]
%  \draw[fill=blue] circle[radius=.2];
% \end{scope}
 
%  \begin{scope}[xshift=.9cm, yshift=.6cm]
%  \draw circle[radius=.2];
% \end{scope}

% \begin{scope}[xshift=.6cm, yshift=1.2cm]
%  \draw circle[radius=.2];
% \end{scope}

\ball{0}{-1}{0.6}
\ball{1}{-2}{0.6}
\ball{-1}{-2}{0.6}[blue]
\ball{-2}{-3}{0.6}[red]
\ball{0}{-3}{0.6}[blue]
\ball{2}{-3}{0.6}[red]

\end{tikzpicture}
\end{center}

\subsubsection{Torsion Pair}
The torsion pair is just a natural algebraic way of describing the torsion class and the corresponding torsion free class simultaneously.
\begin{define}
 A {\it torsion pair} is a pair $(\mathcal{G},\mathcal{F})$ of full subcategories such that:
$$
\mathcal{G}=\{x\,|\, Hom(x,y)=0\, \, \forall y\in \mathcal{F}\}
$$
and 
$$
\mathcal{F}=\{y\, | \, Hom(x,y)=0\, \, \forall x\in \mathcal{G}\}
$$
\end{define}
Where $Hom=0$ means that no rectangle can be drawn from an object $x$ to $y$ where all four vertices are also objects in the category. 

\subsection{Constructing a Torsion Class}

Torsion classes can be generated by the inclusion of an arbitrary number of objects in the category (If no objects are included this is just the empty torsion class) and by considering in advance the Torsion Pair. Call these objects added $x_n$. Then find the set of all $Y$ such that $Hom(x_n, Y)=0$. $Y$ is the torsion free class. Then find the set of all $X$ such that $Hom(X,Y)=0$ which is the full torsion class. \\
\\
1. Start with an arbitrary selection of objects filled in blue ($x_n$).

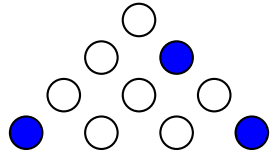
\begin{center}
\begin{tikzpicture}[scale=0.5]
%\draw[help lines=.1] (-4,-4) grid (4,4);

%  \draw[fill=blue] circle[radius=.2];
 
% \begin{scope}[xshift=.6cm]
%  \draw circle[radius=.2];
% \end{scope}
 
%  \begin{scope}[xshift=1.2cm]
%  \draw circle[radius=.2];
% \end{scope}

% \begin{scope}[xshift=.3cm, yshift=.6cm]
%  \draw circle[radius=.2];
% \end{scope}
 
%  \begin{scope}[xshift=.9cm, yshift=.6cm]
%  \draw circle[radius=.2];
% \end{scope}

% \begin{scope}[xshift=.6cm, yshift=1.2cm]
%  \draw circle[radius=.2];
% \end{scope}

% \begin{scope}[xshift=1.8cm]
%  \draw[fill=blue] circle[radius=.2];
% \end{scope}

% \begin{scope}[xshift=1.5cm, yshift=.6cm]
%  \draw circle[radius=.2];
% \end{scope}

% \begin{scope}[xshift=1.2cm, yshift=1.2cm]
%  \draw[fill=blue] circle[radius=.2];
% \end{scope}

% \begin{scope}[xshift=.9cm, yshift=1.8cm]
%  \draw circle[radius=.2];
% \end{scope}

\ball{0}{-1}{0.6}
\ball{1}{-2}{0.6}[blue]
\ball{-1}{-2}{0.6}
\ball{-2}{-3}{0.6}
\ball{0}{-3}{0.6}
\ball{2}{-3}{0.6}
\ball{-3}{-4}{0.6}[blue]
\ball{-1}{-4}{0.6}
\ball{1}{-4}{0.6}
\ball{3}{-4}{0.6}[blue]

\end{tikzpicture}
\end{center}

\noindent 2. Then locate objects where no rectangle can be drawn up/right or down right from a blue object ($Hom=0$) and fill them in red ($Y$). 

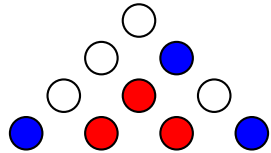
\begin{center}
\begin{tikzpicture}[scale=0.5]
%\draw[help lines=.1] (-4,-4) grid (4,4);

%  \draw[fill=blue] circle[radius=.2];
 
% \begin{scope}[xshift=.6cm]
%  \draw[fill=red] circle[radius=.2];
% \end{scope}
 
%  \begin{scope}[xshift=1.2cm]
%  \draw[fill=red] circle[radius=.2];
% \end{scope}

% \begin{scope}[xshift=.3cm, yshift=.6cm]
%  \draw circle[radius=.2];
% \end{scope}
 
%  \begin{scope}[xshift=.9cm, yshift=.6cm]
%  \draw[fill=red] circle[radius=.2];
% \end{scope}

% \begin{scope}[xshift=.6cm, yshift=1.2cm]
%  \draw circle[radius=.2];
% \end{scope}

% \begin{scope}[xshift=1.8cm]
%  \draw[fill=blue] circle[radius=.2];
% \end{scope}

% \begin{scope}[xshift=1.5cm, yshift=.6cm]
%  \draw circle[radius=.2];
% \end{scope}

% \begin{scope}[xshift=1.2cm, yshift=1.2cm]
%  \draw[fill=blue] circle[radius=.2];
% \end{scope}

% \begin{scope}[xshift=.9cm, yshift=1.8cm]
%  \draw circle[radius=.2];
% \end{scope}

\ball{0}{-1}{0.6}
\ball{1}{-2}{0.6}[blue]
\ball{-1}{-2}{0.6}
\ball{-2}{-3}{0.6}
\ball{0}{-3}{0.6}[red]
\ball{2}{-3}{0.6}
\ball{-3}{-4}{0.6}[blue]
\ball{-1}{-4}{0.6}[red]
\ball{1}{-4}{0.6}[red]
\ball{3}{-4}{0.6}[blue]
\end{tikzpicture}
\end{center}

\noindent 3. Fill in objects blue which have no homomorphism to any red objects. 

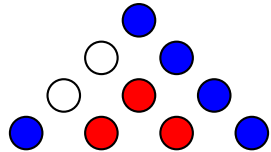
\begin{center}
\begin{tikzpicture}[scale=0.5]
%\draw[help lines=.1] (-4,-4) grid (4,4);

%  \draw[fill=blue] circle[radius=.2];
 
% \begin{scope}[xshift=.6cm]
%  \draw[fill=red] circle[radius=.2];
% \end{scope}
 
%  \begin{scope}[xshift=1.2cm]
%  \draw[fill=red] circle[radius=.2];
% \end{scope}

% \begin{scope}[xshift=.3cm, yshift=.6cm]
%  \draw circle[radius=.2];
% \end{scope}
 
%  \begin{scope}[xshift=.9cm, yshift=.6cm]
%  \draw[fill=red] circle[radius=.2];
% \end{scope}

% \begin{scope}[xshift=.6cm, yshift=1.2cm]
%  \draw circle[radius=.2];
% \end{scope}

% \begin{scope}[xshift=1.8cm]
%  \draw[fill=blue] circle[radius=.2];
% \end{scope}

% \begin{scope}[xshift=1.5cm, yshift=.6cm]
%  \draw[fill=blue] circle[radius=.2];
% \end{scope}

% \begin{scope}[xshift=1.2cm, yshift=1.2cm]
%  \draw[fill=blue] circle[radius=.2];
% \end{scope}

% \begin{scope}[xshift=.9cm, yshift=1.8cm]
%  \draw[fill=blue] circle[radius=.2];
% \end{scope}

\ball{0}{-1}{0.6}[blue]
\ball{1}{-2}{0.6}[blue]
\ball{-1}{-2}{0.6}
\ball{-2}{-3}{0.6}
\ball{0}{-3}{0.6}[red]
\ball{2}{-3}{0.6}[blue]
\ball{-3}{-4}{0.6}[blue]
\ball{-1}{-4}{0.6}[red]
\ball{1}{-4}{0.6}[red]
\ball{3}{-4}{0.6}[blue]
\end{tikzpicture}
\end{center}

\noindent The total collection of blue objects is the complete torsion class.

\subsection{Bijection Between Binary Trees and Torsion Classes}

\subsubsection{Geometric Correspondence}
For a binary tree with $n$ leaves, take the torsion classes corresponding with triangular configurations of balls having $n-2$ balls along each edge. (i.e. if a binary tree has 5 leaves, take a triangle configuration having 3 balls on each edge). Each ball should be located directly above an "inside leaf" (just a leaf that's not on the end). \\

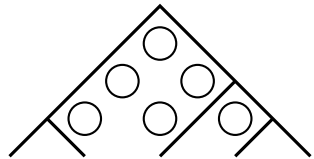
\begin{center}
\begin{tikzpicture}[scale=0.5]

% \draw[thick] (-.8,-.4)--(.6, 1.8)--(2,-.4);

% \draw[thick] (-.5,.15)--(0,-.4);

% \draw[thick] (.6,-.4)--(1.35,.75);

% \draw[thick] (1.2,-.4)--(1.7,.35);

%  \draw circle[radius=.2];
 
% \begin{scope}[xshift=.6cm]
%  \draw circle[radius=.2];
% \end{scope}
 
%  \begin{scope}[xshift=1.2cm]
%  \draw circle[radius=.2];
% \end{scope}

% \begin{scope}[xshift=.3cm, yshift=.6cm]
%  \draw circle[radius=.2];
% \end{scope}
 
%  \begin{scope}[xshift=.9cm, yshift=.6cm]
%  \draw circle[radius=.2];
% \end{scope}

% \begin{scope}[xshift=.6cm, yshift=1.2cm]
%  \draw circle[radius=.2];
% \end{scope}

\draw[very thick] (-4, -4) -- (0,0) -- (4, -4);
\draw[very thick] (-3, -3) -- (-2, -4);
\draw[very thick] (2, -2) -- (0, -4);
\draw[very thick] (3, -3) -- (2, -4);
\ball{0}{-1}{0.6}
\ball{-1}{-2}{0.6}
\ball{1}{-2}{0.6}
\ball{-2}{-3}{0.6}
\ball{0}{-3}{0.6}
\ball{2}{-3}{0.6}

\end{tikzpicture}
\end{center}
The bijection is given geometrically like so: Directly above all branches of the binary tree which are descending left to right, color the objects blue. directly above the branches ascending left to right color the objects red. \\
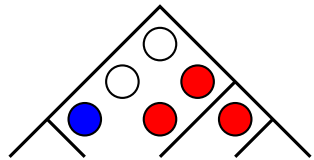
\begin{center}
\begin{tikzpicture}[scale=0.5]

% \draw[thick] (-.8,-.4)--(.6, 1.8)--(2,-.4);

% \draw[thick] (-.5,.15)--(0,-.4);

% \draw[thick] (.6,-.4)--(1.35,.75);

% \draw[thick] (1.2,-.4)--(1.7,.35);

%  \draw[fill=blue] circle[radius=.2];
 
% \begin{scope}[xshift=.6cm]
%  \draw[fill=red] circle[radius=.2];
% \end{scope}
 
%  \begin{scope}[xshift=1.2cm]
%  \draw[fill=red] circle[radius=.2];
% \end{scope}

% \begin{scope}[xshift=.3cm, yshift=.6cm]
%  \draw circle[radius=.2];
% \end{scope}
 
%  \begin{scope}[xshift=.9cm, yshift=.6cm]
%  \draw[fill=red] circle[radius=.2];
% \end{scope}

% \begin{scope}[xshift=.6cm, yshift=1.2cm]
%  \draw circle[radius=.2];
% \end{scope}

\draw[very thick] (-4, -4) -- (0,0) -- (4, -4);
\draw[very thick] (-3, -3) -- (-2, -4);
\draw[very thick] (2, -2) -- (0, -4);
\draw[very thick] (3, -3) -- (2, -4);
\ball{0}{-1}{0.6}
\ball{-1}{-2}{0.6}
\ball{1}{-2}{0.6}[red]
\ball{-2}{-3}{0.6}[blue]
\ball{0}{-3}{0.6}[red]
\ball{2}{-3}{0.6}[red]

\end{tikzpicture}
\end{center}
\begin{thm}
The collection of objects directly above the descending edges (blue) gives a torsion class and the collection of objects directly above the ascending edges (red) form the corresponding torsion-free class. 

\end{thm}
\noindent {\bf{Proof}} Take our proposed correspondence from binary trees to torsion classes. Above all descending branches should be objects in the torsion class, but why do they form a torsion class? Notice that along descending branches are objects that if one were in the torsion class, they would all have to be, since by our construction of torsion classes objects in the torsion class will fall diagonally below the already included object, since these objects will share the same torsion free objects in common. This is so because there are always homomorphisms between objects on the same diagonal line (so they will never include both torsion and torsion free objects) and because if an object is in the torsion free class with respect to a given torsion object, then an object diagonally downward cannot have any homomorphisms to any of those torsion free objects. Torsion free objects will always be to the left of the descending line or on the bottom row, all of which will not have a homomorphism to objects on the line. Hence we get the descending branches. \\

To go the other direction, from an arrangement of balls draw descending branches underneath objects in the torsion class and ascending branches under objects in the torsion free class. \\

\begin{thm}
The binary tree got from drawing descending branches underneath torsion objects is unique. 
\end{thm}

\noindent{\bf{Proof}} We know that torsion objects give clear instructions for drawing descending branches of binary trees, but how do we know this tree is unique? The proof follows from the fact that binary trees are entirely determined by their descending (or ascending edges). Suppose you fixed the descending edges of a binary tree. Then for all leaves left unconnected, ascending branches must be linked to them, giving the rest of the tree. Since there are only two options for a branch (it is descending or ascending) giving one gives the other since they are complements. It turns out it is always possible to connect all remaining spaces with ascending edges once descending edges are defined. \\
\\

\begin{thm}
The correspondence described above is a bijection.
\end{thm}
\noindent {\bf{Proof}} This follows from the well known fact that there are a Catalan number $C_n$ of both binary trees with $n+1$ leaves and torsion classes with $n-1$ balls in the bottom row. So since a unique tree is given by every torsion class, and a unique torsion class is given by every binary tree, and there are exactly the same number of each, they are in a bijective relationship and the bijection is given by our rule. \\

\subsubsection{Induction}

It may still be supposed that our correspondence fails when building larger binary trees but the following induction shows that any new tree formed from smaller trees in fact yields a proper torsion pair. \\

 Connecting two trees gives a new descending branch, a new ascending branch and potentially some objects without branches underneath. Objects along the new descending branch will be filled in blue by our correspondence, but notice that they should also be included to the torsion class by our construction of torsion classes. There will be no homomorphism between objects along the new descending branch and objects down and to the left of them which are in the torsion free class of the previous tree. Therefore the objects along the new descending line MUST be filled blue. Objects along the new ascending branch will be filled red by our correspondence, but notice that the new ascending branch must connect to a leaf at the bottom which is at least one space away from the right-most leaf of the left subtree. This guarantees there is no way to draw a rectangle, and thus find a homomorphism, from any torsion object in the left subtree to the new ascending branch, and obviously any torsion object in the right subtree cannot have a homomorphism going left. The objects along neither new branch are exactly those who are diagonally upward from the objects along the descending branch, meaning that there are existing homomorphism from torsion objects to them, thus meaning they do not lie in either the torsion class or torsion free class. \\
 \\
 Thus binary trees with $n+1$ leaves and torsion classes of $\mathcal{A}_{n-1}$ are in bijection with exactly a $C_n$ number of each, and our correspondence gives the bijective mapping.

%\newpage

%We show that this construction gives back \ul{Young diagrams with gaps} which can be pushed together to give standard Young diagrams without gaps.

%\input{Ruiyang.tex} %

% Ruiyang's contribution 

%\author{Ruiyang Hu}

\chapter*{Young diagrams with gaps and torsion classes}

\centerline{Ruiyang Hu}

%\begin{document}

\begin{abstract}
The purpose of this short paper is to show an explicit bijection between subdiagrams of a full Young diagram with gaps and torsion classes with the same size $n$.
\end{abstract}

\bigskip
%\maketitle

\subsubsection{Introduction}
It is well-known that there are $C_{n+1}$ number of Young diagrams that are subdiagrams of a full Young diagram without gaps of size $n$. ($C_{n+1}$ denotes for the $n+1$th Catalan number.)  
So is the number of torsion class for $A_n$. Since the cardinality of two sets are the same, if we want to prove a given mapping between Young diagrams with gaps and torsion classes of the same size is a biijection, showing that such mapping is a injection or surjection will be enough.  

In the following part of this paper, we will first have to define all concepts we are going to use in this paper. Then we will show some examples of correspondencfe between Young diagrams with gaps and torsion classes when $n$ is small. Finally, we will use a strong Mathematical induction to construct an explicit bijection between two sets of objects.  
\bigskip

\subsection{Definitions and Examples}
Since this paper tries to prove there is a bijection between Young diagrams with gaps and torsion class that both have the same size $n$, we will first define what is a Young diagram with gaps.
\bigskip

\begin{define}\label{YDG}
A \emph{Young diagram with gaps} $Y$ is a visual reprensentation of a partition, just like normal Young diagrams without gaps, where:
\begin{enumerate}[a)]
\item\label{YDGa} $Y$ is transformed from the corresponding Young diagram without gaps $Y'$ that contains $n$ boxes and $k$ rows. Each row of $Y'$ has various number of boxes in it where number of boxes in $k$ rows sum up to $n$. ($k,n\in \mathcal{Z}^+$)
\item\label{YDGb} Let $\lambda_i$ denotes number of boxes in the $i$th row of $Y$. If $Y'$ is a valid normal Young diagram, then number of boxes in each row has to statisfies following rule: $\lambda_1\ge\lambda_2\ge\dots \lambda_k >0$ and $\lambda_{k+1}=0$.
\item\label{YDGc} $Y$ is transfromed from the corresponding normal Young diagram $Y'$ by first rotating clockwise by 45 degrees around the upper-left corner of $Y'$. Then the rotated diagram will have ascending rectangles and descending rectangles, which are originally columns and rows of Young diagrams.
\item\label{YDGd} Depends on number of rows and columns in $Y'$, there are four different ways for four cases to complete the transformation from $Y'$ to $Y$.
\begin{enumerate}[1)]
\item\label{Case1} If $Y'$ contains more rows than columns, then we have to the leftmost rectangle in $Y$ that its lowest box is in the lowest position. Use the lowest corner in that box as start point, draw a horizontal line which is perpendicular to the straight line formed by connecting lowest and highest corners of the box. For all other ascending rectangles on the left and right of this big rectangle, slide them down without breaking the integrity of each ascending rectangles one by one until the lowest corner of the lowest box in each ascending rectangles falls on the horizontal line. 
\item\label{Case2} If $Y'$ contains more columns than rows, then we need to find the rightmost rectangle in $Y$ that its lowest box is the lowest position. Just like the previous case, we can draw a horizontal line with the lowest corner in the box. Finally, we slide all other ascending rectangles on the left and right of the big rectangles downward without breaking their integrity until the lowest box in each asending rectangles falls on the horizontal line.
\item\label{Case3} If $Y'$ contains the same number of columns and rows, and the leftmost and rightmost boxes are both the lowest boxes in this diagram. Our next step is to keep the row and the column that contain these two boxes unchanged, and slide all inner ascending rectangles downward so that the lowest box in each inner ascending rectangle lows on the horizontal line.
\item\label{Case4} If $Y'$ contains the same number of columns and rows, and the leftmost and rightmost boxes are not in the lowest position. Then we have to find rectangles that contain boxes lay in the lowest position in this diagram. Draw a horizontal line on the lower corner of these boxes, and slide all ascending edges on the left and right of these rectangles that already falls on the horizontal line downwards without breaking integrity of each ascending edge.   
\end{enumerate}
\item\label{YDGe} To transform from Young diagrams with gaps back to Young diagrams without gaps, we first have to sort ascending rectangles by number of boxes in each rectangle with descending order. After that, the ascending rectangle with smallest number of boxes will be put in the rightmost position. Then we will remove gaps between rectangles. Finally, rotate the diagram by 45 degree counterclockwise, and the result is a Young diagram without gaps.
\end{enumerate}
\end{define}

To clarify the definition, here are some examples of valid Young diagrams without gaps and their corresponding Young diagrams with gaps:
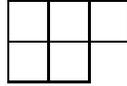
\newline
\begin{center}
\ydiagram{3,2}
\end{center}
\par
For the normal Young diagram Y1' above, it has 5 boxes so $n=5$. It has 3 boxes in the first row and 2 boxes in the second one, thus $\lambda_1=3$, $\lambda_2=2$ and $\lambda_3=0$. Since the diagram has two rows, $k=2$. It is easy to see that all $\lambda$ here statisfy rule \ref{YDGb} in definition 2.1. Next step is to transform Y1' into a Young diagram with gap.

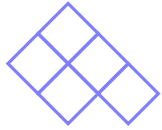
\newcommand{\YPthreeb}{{\draw[very thick,color=blue!50!white,xshift=5mm,yshift=-5mm] (4,2)--(3,3)--(4,4)--(5,3)--cycle;}}%(4,2)--(3,3);}}

\newcommand{\YPtwob}{{\draw[very thick,color=blue!50!white,xshift=5mm,yshift=-5mm] (3,1)--(2,2)--(3,3)--(4,2)--cycle
;\YPthreeb}}%(3,1)--(2,2);}}

\newcommand{\YSoneb}{{\draw[very thick,color=blue!50!white,xshift=5mm,yshift=-5mm] (2,0)--(1,1)--(2,2)--(3,1)--cycle;\YPtwob}}%(2,0)--(1,1);}}
\newcommand{\YSonew}{{\draw[thick,color=red!40!white] (2.5,.5)  circle[radius=.3cm];}}%(2,0)--(3,1);}}
\newcommand{\YItwob}{{\draw[very thick,color=blue!50!white,xshift=5mm,yshift=-5mm] (5,1)--(4,2)--(5,3)--(6,2)--cycle;}}%(5,1)--(4,2);}}

\newcommand{\YStwob}{{\draw[very thick,color=blue!50!white,xshift=5mm,yshift=-5mm] (4,0)--(3,1)--(4,2)--(5,1)--cycle;\YItwob}}%(4,0)--(3,1);}}
\newcommand{\YSthreeb}{{\draw[very thick,color=blue!50!white,xshift=5mm,yshift=-5mm] (6,0)--(5,1)--(6,2)--(7,1)--cycle;}}%(6,0)--(5,1);}}
\newcommand{\YStwow}{{\draw[thick,color=red!40!white]  (4.5, .5) circle[radius=.3cm];}}%(4,0)--(5,1);}}
\newcommand{\YSthreew}{{\draw[thick,color=red!40!white]  (6.5, .5) circle[radius=.3cm];}}%(6,0)--(7,1);}}
\newcommand{\YPtwow}{{\draw[thick,color=red!40!white]  (3.5, 1.5) circle[radius=.3cm];}}%(3,1)--(4,2);}}
\newcommand{\YItwow}{{\draw[thick,color=red!40!white]  (5.5, 1.5) circle[radius=.3cm];}}%(5,1)--(6,2);}}
\newcommand{\YPthreew}{{\draw[thick,color=red!40!white]  (4.5, 2.5) circle[radius=.3cm];}}%(4,2)--(5,3);}}

\newcommand{\YPthreed}{{\draw[very thick,color=blue!50!white,xshift=5mm,yshift=-5mm] (5,-1)--(4,0)--(5,1)--(6,0)--cycle;}}
\newcommand{\YPthreec}{{\draw[very thick,color=blue!50!white,xshift=5mm,yshift=-5mm] (6,0)--(5,1)--(6,2)--(7,1)--cycle
;\YPthreed}}
\newcommand{\YPthreee}{{\draw[very thick,color=blue!50!white,xshift=5mm,yshift=-5mm] (6,0)--(5,1)--(6,2)--(7,1)--cycle;}}
\newcommand{\YPfoura}{{\draw[very thick,color=blue!50!white,xshift=5mm,yshift=-5mm](1,-1)--(0,0)--(1,1)--(2,0)--cycle;}}
\newcommand{\YPfourb}{{\draw[very thick,color=blue!50!white,xshift=5mm,yshift=-5mm](7,-1)--(6,0)--(7,1)--(8,0)--cycle;}}
\newcommand{\YPfourc}{{\draw[very thick,color=blue!50!white,xshift=5mm,yshift=-5mm](5,-1)--(4,0)--(5,1)--(6,0)--cycle;}}
\newcommand{\BLa}{{\draw[very thick,color=red!50!white,xshift=5mm,yshift=-5mm] (2,0)--(8,0);}}
\newcommand{\BLb}{{\draw[very thick,color=red!50!white,xshift=5mm,yshift=-5mm] (2,1)--(8,1);}}
\newcommand{\BLc}{{\draw[very thick,color=red!50!white,xshift=5mm,yshift=-5mm] (0,-1)--(8,-1);}}

\begin{center}
\begin{tikzpicture}[scale=.4]
\YStwob
\YSthreeb
\YPtwob
\end{tikzpicture}
\end{center}
\par
First we rotate Y1' clockwise by 45 degree, where the result is shown by the graph above. For this example, we have already transformed Y1' into a Young diagram with gaps because number of rows is larger than number of columns, and it is case \ref{Case2}. By drawing a horizontal line(shown a red line by the diagram below) at the lowest corner of the lowest box in rightmost descending rectangle, we can observe that all lowest corners of the lowest boxes of other descending rectangles already lay on the horizontal line. This means we do not have to slide anything downward.
\iffalse
\begin{center}
\begin{tikzpicture}[scale=.4]
\YPtwob
\YItwob
\YPthreec
\end{tikzpicture}
\end{center}
\fi
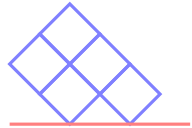
\begin{center}
\begin{tikzpicture}[scale=.4]
\YStwob
\YSthreeb
\YPtwob
\BLa
\end{tikzpicture}
\end{center}
\par
The Young diagram $Y2'$ without gaps below belongs to case \ref{Case1} since it has 3 rows and only 2 columns. 
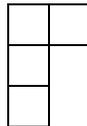
\begin{center}
\ydiagram{2,1,1}
\end{center}
\par
By rotating $Y2'$, we will get a graph that is shown below. This graph is not a valid Young diagram with gaps because the lowest box of one ascending rectangle does not fall on the horizontal line, which is shown below. 
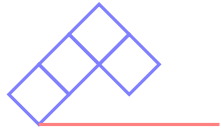
\begin{center}
\begin{tikzpicture}[scale=0.4]
\YSoneb
\YItwob
\BLa
\end{tikzpicture}
\end{center} 
\par
The valid Young diagram $Y2$ with gaps should look like this after we shift all other ascending rectangles downwards:
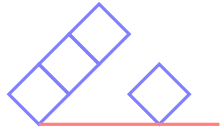
\begin{center}
\begin{tikzpicture}[scale=0.4]
\YSoneb
\YSthreeb
\BLa
\end{tikzpicture}
\end{center} 
\par
The Young diagram $Y3'$ without gaps below belongs to case \ref{Case3} because it has 3 rows and 3 columns.
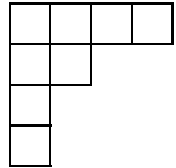
\begin{center}
\ydiagram{4,2,1,1}
\end{center}
\par
After we rotate $Y3'$ clockwise by 45 degree, we now have this diagram below, which is currently not a valid Young diagram with gaps:
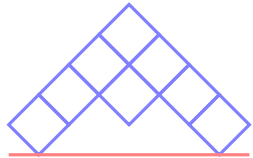
\begin{center}
\begin{tikzpicture}[scale=0.4]
\YSoneb
\YPfoura
\YStwob
\YPthreee
\YPfourb
\BLc
\end{tikzpicture}
\end{center} 
\par
We can find that although the lowest box of the leftmost ascending rectangle and that of the rightmost descending rectangle falls on the red horizontal line, there is one inner ascending rectangle(which in this case, only contains one box) not laying on the horizontal line. So we will have to slide it down to get a valid Young diagram with gaps, and the result looks like this:
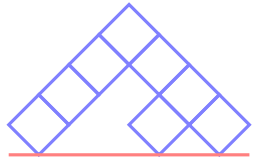
\begin{center}
\begin{tikzpicture}[scale=0.4]
\YSoneb
\YPfoura
\YItwob
\YPthreee
\YPfourb
\YPfourc
\BLc
\end{tikzpicture}
\end{center} 
\par
For more examples of transforming between Young diagram without gaps and Young diagrams with gaps, please check Olly Liang and Serra's paper.
\medskip\par
After Young diagrams with gaps are defined, we shall give definitions and examples of \emph{full Young diagrams without gaps} and \emph{full Young diagrams with gaps}.
\begin{define}\label{FYDWG}
A \emph{full Young diagram without gaps} $Y_{fw}$ is a special type of \emph{Young diagrams without gaps} that not only satisfies all requirements of Young diagram without gaps, but also satisfies following rules:
\begin{enumerate}[a)]
\item\label{FYDWGa} Total number of boxes $N_{b}$ in the full Young diagrams belongs to the Triangular number sequence, which means that $N_{b}={{a+1} \choose 2}$ and $a$ is the number of boxes in the longest row of $Y_{fw}$.
\item\label{FYDWGb} Numbers of rows and columns in $Y_{fw}$ all equal to $a$. 
\end{enumerate}  
\end{define}
\par
For example, the diagram below is a full Young diagram without gaps, where $a=2$ and ${N_b}=3$:
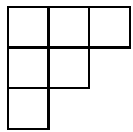
\begin{center}
\ydiagram{3,2,1}
\end{center}
\begin{define}\label{FYDG}
A \emph{full Young diagram with gaps} $Y_{f}$ is a special type of \emph{Young diagrams with gaps} that not only satisfies all requirements of Young diagram with gaps, but also satisfies following rules:
\begin{enumerate}[a)]
\item\label{FYDGa} Total number of boxes $N_{b}$ also belongs to the Triangular number sequence, and $a$ is the number of boxes in the longest ascending rectangle of $Y_{f}$.
\item\label{FYDGb} $Y_{f}$ has $a$ ascending rectangles and $a$ descending rectangles. 
\end{enumerate}  
\end{define}
As usual, we need an exmaple to make the definition clear:
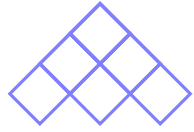
\begin{center}
\begin{tikzpicture}[scale=.4]
\YSoneb
\YStwob
\YSthreeb
\YPtwob
\YItwob
\YPthreeb
\end{tikzpicture}
\end{center} 
\par
The diagram above is an example of a full Young diagram with gaps, where $N_{b}=6$, $a=3$ and ${4 \choose 2}=6=N_{b}$ which satisfies rule \ref{FYDGa} in definition 2.2. Besides that, the full Young diagram with gaps above corresponds to the full Young diagram without gaps above. By simple counting we can find out that number of boxes in a row decreases by one each time from the highest row to the lowest row.($3\rightarrow2\rightarrow1$)  

Definition of torsion classes will be defined in the paper written by Max and Yicheng Tao, hence we will skip this part this paper.
\medskip\newline\par
In introduction section of this paper we mentioned that there exists a correspondence between subdiagrams of a Young diagram with gaps and torsion classes with the same \emph{size}, so so it is necessary to define the term \emph{size}. 
\begin{define}\label{size}
If a torsion class or a full Young diagram with gaps has a \emph{size} of $n$, it means that:
\begin{enumerate}[a)]
\item\label{sizea} The torsion class contains $\frac {n^{2}+n}{2}$ balls, which contains 1 ball in first row(or 0 ball if $n=0$), then two balls in the second row and so on. 
\item\label{sizeb} The full Young diagram with gaps contains $\frac {n^{2}+n}{2}$ boxes.
\end{enumerate}
\end{define}

Since we already have an example of full Young diagram with gaps that has a size of 3, let us take a look at an example of a torsion class with size of 3:
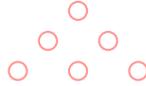
\newline
\begin{center}
\begin{tikzpicture}[scale=0.4]
\YSonew
\YStwow
\YSthreew
\YPtwow
\YItwow
\YPthreew
\end{tikzpicture}
\end{center}
\par
The diagram above is an example of a torsion class which has a size of 3. As we can see, there are ${\frac {3^{2}+2} {2}}=6={4\choose 2}$ balls.

Because we have already defined all definitions we will need in this paper, we will move to next section, which shows how to transform from a torsion class to the corresponding Young diagram with gaps.

\newpage
\subsection{Method and Examples of Transforming a Torsion Class to Its Correspondence Young Diagram}
To construct an explicit bijective correspondence between torsion classes and subdiagrams of a full Young diagram with gaps that both have the same size, we first have to know how the transformation between two kinds of objects works. As we have already mentioned in the introduction part, there are $C_{n+1}$ number of subdiagram of a full Young diagram with gaps of size $n$ and the same number of torsion classes with size $n$. Hence, we will only have to construct an injection or surjection from one set to another and the bijection relationship will be proved. In this paper we will approach from torsion classes to Young diagrams with gaps, so we have to explain how to transform a random torsion class into its corresponding Young diagram with gaps.

Let's start with how to draw and complete a torsion class:
\begin{enumerate}[1)]
\item\label{Trulea} First fill in some balls in the torsion class with color blue.
\item\label{Truleb} For balls we already filled with blue, check whether there exist a ball that lays in the lower right position relative to a ball that has already been colored with blue. If there is, color that ball with blue too.
\item\label{Trulec} Keep executing the second step until there is no more balls that is not colored while lays in the lower right position of a colored ball.
\item\label{Truled} For each pair of colored balls, check whether they could be put into a rectangle with four corners of four different balls where each of these two balls are neither in the highest nor in the lowest corner of the rectangle. If they could, check whether the ball in the lowest corner of this rectangle is inside the torsion class or could be added at one line below the lowest row of balls in the torsion class. If the answer is yes, color the ball in the highest corner of this rectangle with blue and we call such ball as an "extension". Then go back to step 2 again.  
\end{enumerate}

\newcommand{\TSoneb}{{\draw[fill,color=blue!50!white] (2.5,.5) circle[radius=.3cm];}}%(2,0)--(1,1);}}
\newcommand{\TSonew}{{\draw[thick,color=red!40!white] (2.5,.5)  circle[radius=.3cm];}}%(2,0)--(3,1);}}
\newcommand{\TStwob}{{\draw[fill,color=blue!50!white] (4.5, .5) circle[radius=.3cm];;}}%(4,0)--(3,1);}}
\newcommand{\TStwow}{{\draw[thick,color=red!40!white]  (4.5, .5) circle[radius=.3cm];;}}%(4,0)--(5,1);}}
\newcommand{\TSthreeb}{{\draw[fill,color=blue!50!white] (6.5, .5) circle[radius=.3cm];;}}%(6,0)--(5,1);}}
\newcommand{\TSthreew}{{\draw[thick,color=red!40!white]  (6.5, .5) circle[radius=.3cm];;}}%(6,0)--(7,1);}}
\newcommand{\TPtwob}{{\draw[fill,color=blue!50!white] (3.5, 1.5) circle[radius=.3cm];;}}%(3,1)--(2,2);}}
\newcommand{\TPtwow}{{\draw[thick,color=red!40!white]  (3.5, 1.5) circle[radius=.3cm];;}}%(3,1)--(4,2);}}
\newcommand{\TItwob}{{\draw[fill,color=blue!50!white] (5.5, 1.5) circle[radius=.3cm];;}}%(5,1)--(4,2);}}
\newcommand{\TItwow}{{\draw[thick,color=red!40!white]  (5.5, 1.5) circle[radius=.3cm];;}}%(5,1)--(6,2);}}
\newcommand{\TPthreeb}{{\draw[fill,color=blue!50!white] (4.5, 2.5) circle[radius=.3cm];;}}%(4,2)--(3,3);}}
\newcommand{\TPthreew}{{\draw[thick,color=red!40!white]  (4.5, 2.5) circle[radius=.3cm];;}}%(4,2)--(5,3);}}
\newcommand{\TPfouraw}{{\draw[thick,color=red!40!white] (1.5,-0.5) circle[radius=.3cm];;}}
\newcommand{\TPfourab}{{\draw[fill,color=blue!40!white] (1.5,-0.5) circle[radius=.3cm];;}}
\newcommand{\TPfourbw}{{\draw[thick,color=red!40!white] (3.5,-0.5) circle[radius=.3cm];;}}
\newcommand{\TPfourbb}{{\draw[fill,color=blue!40!white] (3.5,-0.5) circle[radius=.3cm];;}}
\newcommand{\TPfourcw}{{\draw[thick,color=red!40!white] (5.5,-0.5) circle[radius=.3cm];;}}
\newcommand{\TPfourcb}{{\draw[fill,color=blue!40!white] (5.5,-0.5) circle[radius=.3cm];;}}
\newcommand{\TPfourdw}{{\draw[thick,color=red!40!white] (7.5,-0.5) circle[radius=.3cm];;}}
\newcommand{\TPfourdb}{{\draw[fill,color=blue!40!white] (7.5,-0.5) circle[radius=.3cm];;}}
\newcommand{\TPfivecg}{{\draw[fill,color=green!40!white] (4.5,-1.5) circle[radius=.3cm];;}}
\newcommand{\TPsixdB}{{\draw[fill,color=black!40!white] (5.5,-2.5) circle[radius=.3cm];;}}

\newcommand{\TSoneB}{{\draw[fill,color=black!40!white] (2.5,.5) circle[radius=.3cm];;}}
\newcommand{\TPtwoB}{{\draw[fill,color=black!40!white] (3.5,1.5) circle[radius=.3cm];;}}
\newcommand{\TPthreeB}{{\draw[fill,color=black!50!white] (4.5, 2.5) circle[radius=.3cm];;}}
\newcommand{\TPfouraB}{{\draw[fill,color=black!40!white] (1.5,-0.5) circle[radius=.3cm];;}}

Here is an example that helps illustrate these rules above:
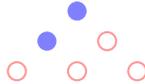
\begin{center}
\begin{tikzpicture}[scale=.4]
\TSonew
\TStwow
\TSthreew
\TPtwob
\TItwow
\TPthreeb
\end{tikzpicture}
\end{center}
\par
The diagram above is currently an incomplete torsion class with size of 3. So let us follow steps of completeing a torision class listed above. There are already two blue balls in the incomplete torsion class, and for each ball there is an uncolored ball in the lower right corner. So according to step \ref{Truleb}, we will have to color these two balls with blue color too, and the result will look like this:
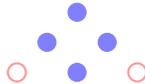
\begin{center}
\begin{tikzpicture}[scale=.4]
\TSonew
\TStwob
\TSthreew
\TPtwob
\TItwob
\TPthreeb
\end{tikzpicture}
\end{center}
\par
However, when we enter step \ref{Trulec}, we will have to repeat step \ref{Truleb} again since there still eixsts a ball uncolored that also lays in lower right position of a colored ball. When we enter step 4 and test every existing pair of blue balls, we can observe that all possible extension are already colored witrh blue. Here is the complete torsion class:
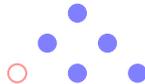
\begin{center}
\begin{tikzpicture}[scale=.4]
\TSonew
\TStwob
\TSthreeb
\TPtwob
\TItwob
\TPthreeb
\end{tikzpicture}
\end{center}
\medskip\par 
Let us have another example in order to make torsion classes completion steps clearer:
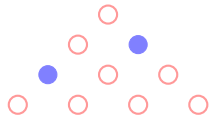
\begin{center}
\begin{tikzpicture}[scale=.4]
\TSoneb
\TStwow
\TSthreew
\TPtwow
\TItwob
\TPthreew
\TPfouraw
\TPfourbw
\TPfourcw
\TPfourdw
\end{tikzpicture}
\end{center}
\par
The diagram above is an incomplete torsion class with size of 4 this time. Just like the previous example, we will have to check for empty balls that should be colored. This is what we have before we begin to execute step \ref{Truled}:
\begin{center}
\begin{tikzpicture}[scale=.4]
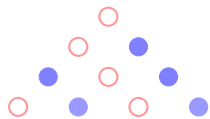
\TSoneb
\TStwow
\TSthreeb
\TPtwow
\TItwob
\TPthreew
\TPfouraw
\TPfourbb
\TPfourcw
\TPfourdb
\end{tikzpicture}
\end{center}
\par
We can find that the pair formed by two blue balls on the third row can form a rectangle where the ball in lowest corner(shown by the yellow ball in the diagram below) is only one line below the lowest row of balls. Here is an invalid rectangle: the pair of balls formed by two balls on the $4th$ row can not form a valid rectangle because the lowest ball(colored with black) is two lines below the lowest row of this torsion class.   
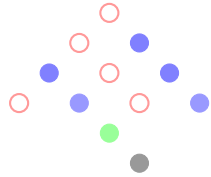
\begin{center}
\begin{tikzpicture}[scale=.4]
\TSoneb
\TStwow
\TSthreeb
\TPtwow
\TItwob
\TPthreew
\TPfouraw
\TPfourbb
\TPfourcw
\TPfourdb
\TPfivecg
\TPsixdB
\end{tikzpicture}
\end{center}
\par
Hence, after we colored the only extension in this example, we have the complete torsion class:
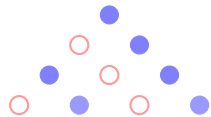
\begin{center}
\begin{tikzpicture}[scale=.4]
\TSoneb
\TStwow
\TSthreeb
\TPtwow
\TItwob
\TPthreeb
\TPfouraw
\TPfourbb
\TPfourcw
\TPfourdb
\end{tikzpicture}
\end{center}
\bigskip\par
After rules of constructing a complete torsion class is described, we will show how to transform from a torsion class to its corresponding Young diagram with gaps. Similar to the transformation from a binary tree to a Young diagram, which will be introduced by Olly Liang and Serra, we will use to "bookshelf" method too. First we can observe that a torsion class has 'columns' that are formed by balls. For example, these four black balls in the diagram below represents the first column and the diagram has four columns:
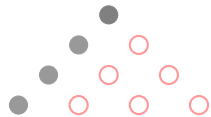
\begin{center}
\begin{tikzpicture}[scale=.4]
\TSoneB
\TStwow
\TSthreew
\TPtwoB
\TItwow
\TPthreeB
\TPfouraB
\TPfourbw
\TPfourcw
\TPfourdw
\end{tikzpicture}
\end{center}
\bigskip\par
Given a torsion class with some balls colored with blue in some columns, we will need to traverse every columns. Starting from the leftmost column, each time we will move to right by one column. And finally we will reach the rightmost column which will always contains only one ball. For each column, traverse it from the lowest ball to the highest ball in the column and find the first ball that is colored with blue. Put boxes on this ball as well as all balls above which are in the same column. Let us use examples to demonstrate this process:
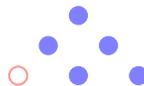
\begin{center}
\begin{tikzpicture}[scale=.4]
\TSonew
\TStwob
\TSthreeb
\TPtwob
\TItwob
\TPthreeb
\end{tikzpicture}
\end{center}
\par
In the diagram above, we can observe that in all three acending edges there exist at least one blue ball, so the corresponding Young diagram with gaps should be:
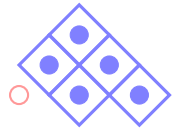
\begin{center}
\begin{tikzpicture}[scale=.4]
\TSonew
\TStwob
\TSthreeb
\TPtwob
\TItwob
\TPthreeb
\YStwob
\YSthreeb
\YPtwob
\end{tikzpicture}
\end{center}
\par Let us have one more example, which this time we will use the torsion class above that has a size of 4. Its correspondence Young diagram with gaps should be like this:

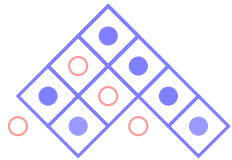
\newcommand{\YSfourd}{{\draw[very thick,color=blue!50!white,xshift=5mm,yshift=-5mm] (7,-1)--(6,0)--(7,1)--(8,0)--cycle;}}
\newcommand{\YSfourb}{{\draw[very thick,color=blue!50!white,xshift=5mm,yshift=-5mm] (3,-1)--(2,0)--(3,1)--(4,0)--cycle;}}
\begin{center}
\begin{tikzpicture}[scale=.4]
\TSoneb
\TStwow
\TSthreeb
\TPtwow
\TItwob
\TPthreeb
\TPfouraw
\TPfourbb
\TPfourcw
\TPfourdb
\YSoneb
\YStwob
\YSthreeb
\YPtwob
\YItwob
\YPthreeb
\YSfourd
\YSfourb
\end{tikzpicture}
\end{center}
\par
Now we have shown how to transform from a torsion class to a Young diagram with gaps, we will move to prove that there exists a bijection between torsion classes and subdiagrams of a full Young diagram with gaps of the same size $n$.

%\newpage
\subsection{Proof of the Bijection between Torsion Classes and Subdiagrams of a Full Young Diagram with Gaps of the Same Size $n$}
To construct an bijection, we will use strong induction here. We can define a recursively mapping from a torsion class with size $n$ to its corresponding Young diagram with gaps.

According to the torsion class rules, as long as this torsion class is not empty, there will always be at least one colored ball in the lowest row of the torsion class. Find the leftmost ball on the lowest row of the torsion class such that the corresponding box of this ball and the highest box in the Young diagram with gaps are contained in the same rectangle in the Young diagram with gaps. And it is the rectangle with most number of boxes(balls) in it. Then this big rectangle will divide the torsion class into two parts: we call the part on the left of this rectangle $X$ and the part on the right $Y$($X$ and $Y$ could be empty). Then we will do the same operations on $X$ and $Y$, until it reaches the unit torsion class, which has only one ball in it.   

Let $P(n)$ denotes that such recursively defined mapping between torsion classes of size $n$ and subdiagrams of a full Young diagram with gaps of size $n$ is a bijection.

First we will have to prove the basis step $P(1)$. There are two torsion class with a size of one. Here are these two torsion classes and two corresponding Young diagram with gaps: 
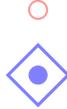
\begin{center}
\begin{tikzpicture}[scale=.4]
\TStwow
\end{tikzpicture}
\end{center}

\begin{center}
\begin{tikzpicture}[scale=.4]
\TPthreeb
\YPthreeb
\end{tikzpicture}
\end{center}
\par
As we can see, each torsion class corresponds to a unique Young diagram with gaps and all subdiagrams of the full Young diagram with gaps of size 1 are mapped to, and for both cases $X$ and $Y$ equal to $\phi$. So $P(1)$ is true.

Then, we assume that for all integers smaller than $K$, where $K\geq 1$, $P(1),P(2),...P(K-1)$ are all true. We then have to prove that when size equals to $K$, the mapping defined recursively is still a bijective correspondence. By applying the same method on torsion classes, we will have several cases:
\begin{enumerate}[1)]
\item The first case is that the leftmost column does not contain any colored balls, this means we can not form a big rectangle that contains the uppermost ball which lays in the first row. However, this also means that we can just omit the uncolored column and focus on those columns that contains colored balls on the right. By omiting the leftmost column, we decrease the size of torsion class by at least one without changing the torsion class and affect its corresponding Young diagram with gaps. Since all $P(C)$($1\leq C<K$) are true according to assumption, in this case such recursively defined mapping is still a bijection.
\item The second case is that we could find a big rectangle that contains most number of balls where the upper corner is the ball in the first row and the lowest corner is the leftmost ball we could find while keeps number of balls in the rectangle being largest. Then we can divide the torsion class into this rectangle, $X$ and $Y$. Obviously $X$ and $Y$ are all smaller than $K$ so that we could apply the assumption on them and breaks them into even smaller pieces such as $X_1$ and $X_2$ with the recursively defined mapping until they are splitted into the base case. 
\end{enumerate}
\par
Thus, as long as $P(C)$($1\leq C<K$) are true, $P(K) is true$. Since we already proved that $P(1)$ is true, so $P(K)$ is true for all $K\geq 1$. 
\par
Here are some examples of breaking a torsion class into smaller torsion class:
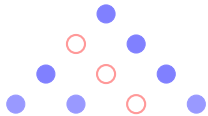
\begin{center}
\begin{tikzpicture}[scale=.4]
\TSoneb
\TStwow
\TSthreeb
\TPtwow
\TItwob
\TPthreeb
\TPfourab
\TPfourbb
\TPfourcw
\TPfourdb
\end{tikzpicture}
\end{center}
\par
First we could find that the leftmost ball in the lowest row of this torsion class that forms a rectangle that contains most number of balls and the uppermost ball is the second ball in the lowest row. So we can draw a rectangle to divide this torsion class into smaller pieces:
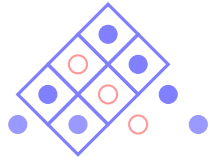
\begin{center}
\begin{tikzpicture}[scale=.4]
\TSoneb
\TStwow
\TSthreeb
\TPtwow
\TItwob
\TPthreeb
\TPfourab
\TPfourbb
\TPfourcw
\TPfourdb
\YSoneb
\YStwob
\YSfourb
\end{tikzpicture}
\end{center}
\par
Here we have $X$ and $Y$, where $X$ belongs to the base case since it has a size of 1. Let's take a look on $Y$, we can find that there is still a rectangle in $Y$, where the leftmost ball is the second ball in the lowest row:
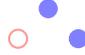
\begin{center}
\begin{tikzpicture}[scale=.4]
\TStwow
\TSthreeb
\TItwob
\end{tikzpicture}
\end{center}
\par
We could find that this new rectangle divides $Y$ into $Y_1$ and $Y_2$, where both of them belongs to the base case. Hence we are done here.
\bigskip
\subsubsection{Conclusion}
By strong induction, we have shown that such recursively defined mapping is an explicit bijection between torsion classes and subdiagrams of a full Young diagram with gaps of the same size.    

%\end{document}

%\newpage

%Next, we go to pattern avoiding permutations. By a 213-avoiding permutation we mean a permutation of $n$ letters so that (Something). We use the \ul{baseball} construction to give the bijection between torsion classes and 213-avoiding permutations.

%\input{MikeAndZhaonan.tex} %https://www.overleaf.com/project/5fb3f737c40b291947be87aa%

% Mike and Zhaonan's contribution

\newpage

\chapter*{213-Avoiding Permutations and Binary Trees}

%(Mike, Leo)}
\centerline{Zhaonan Li,  Michael Richard}
%\date{December 2019}

%\begin{document}

%\maketitle
\begin{abstract}
The aim of this paper is to define and describe 213-avoiding permutations, and to prove a bijection with binary trees. Such a bijection was known to have existed as each of these objects are counted by Catalan numbers, however this paper offers one explicitly.
\end{abstract}
\bigskip

\subsubsection{Introduction}
It is well known that 213-avoiding permutations and binary trees are in bijection due to the fact that there are a Catalan number of each for a given size. However, this paper seeks to give a new bijective map that allows one to convert 213-avoiding permutations into binary trees and vice versa. The drawing of these binary trees is done so according to another paper and relates to the bookshelf construction (olly \& serra reference). This novel method of expressing binary trees also allows for the use of the ``baseball'' construction, which maps these binary trees to 213-avoiding permutations.
\subsection{213 Avoiding Permutations}

\begin{define}\label{213ap}
A \emph{213-avoiding permutation} is a permutation $p = (m_1, m_2, m_3, ..., m_n)$ such that for any $m_i$ and $m_j \ (j>i)$, if $m_i>m_j$ there is no $m_k \ (k>j)$ such that $m_k>m_i$
\end{define}

\begin{Ex}
$p=(5,1,2,3,4)$ is a 213-avoiding permutation of 5.
\end{Ex}

\begin{Ex}
$p=(2,3,1,4,5)$ is not a 213-avoiding permutation of 5, because $m_1 > m_3$, but $m_4 > m_1$. 
\end{Ex}

The following theorem states a well-known fact that we are going to use in this paper:
\begin{thm}\label{pnum}
The number of 213-avoiding permutations with n terms $P_n$ is equal to the n-th Catalan number
$$P_n = C_n = \frac{1}{n+1}\binom{2n}{n}$$
\end{thm}

\comment{
\begin{proof}
It's obvious that when $n$ is less than $3$, all permutations are $213$ avoiding. Therefore, 
$$
    P_0 = 1 = C_0. 
$$
$$
    P_1 = 1 = C_1.
$$
$$
    P_2 = 2 = C_2.
$$
Now we consider a permutation $p = (m_1, m_2, ... , m_n)$ with $n$ terms when $n \geq 3$. Let the smallest term be $m_i$, and let the permutation consisting of all terms preceding to $m_i$ be $p_{\alpha}$, and let the permutation consisting of all terms following $m_i$ be $p_{\beta}$. In other words,
$$
    p_\alpha = (m_1, m_2, ... m_{i-1}),\  p_\beta = (m_{i+1}, m_{i+2}, ..., m_n).
$$
In order for $p$ to be $213$-avoiding, by definition $\ref{213ap}$, 
\begin{enumerate}
  \item all terms in $p_{\alpha}$ need to be greater than all terms in $p_{\beta}$, and
  \item $p_{\alpha}$ and $p_{\beta}$ need to be $213$-avoiding permutations.
\end{enumerate}
From $(1)$, we know that 

\end{proof}
}

\smallskip
\subsection{Converting 213-avoiding Permutations to Binary Trees}

\begin{define}
Let $|p|$ be defined by the cardinality of, or the number of terms contained within, a permutation $p$.
\end{define}

\begin{Ex}
For the permutation $p=(4,1,2,3)$, $|p|=4$
\end{Ex}

\subsubsection{Coordinate system for binary trees}
\hfill\\
In this section, we introduce the coordinate system used to describe binary trees in this paper. The two edges extending from the root are regarded as two axes. Every node in the binary tree is represented by two numbers $x$ and $y$, measuring the distance along the two axes. The root of the binary tree is positioned at the origin $(0, 0)$. Every edge extends to the bottom. This coordinate system was used in a paper which can be referred to for examples (Olly and Serra).\\

\subsubsection{Constructing a binary tree given a $213$-avoiding permutation}\label{p2b}
\hfill\\
In this subsection, we define the procedure to convert a $213$-avoiding permutation with $n$ terms to a binary tree with $n+1$ leaves. Given some $213$-avoiding permutation $p=(\phi_1,\phi_2,\phi_3,...,\phi_n)$ with $n$ terms. Let $\phi_a$ be the smallest term in $p$. Let $y$ be defined as the sequence of all terms in $p$, $\phi_b$ such that $b<a$. Let $x$ be defined as the sequence of all terms in $p$, $\phi_c$ such that $c>a$. Now, to construct the binary tree which corresponds to $p$, draw an ascending line from $(n,0)$ to the root, $(0,0)$ and a descending line from $(0,0)$ to $(0,n)$. Then, draw an ascending line from $(|y|, n - |y|)$ to $(0,n - |y|)$ and a descending line from $(n - |x|, 0)$ to $(n - |x|, |x|)$. This second pair of lines creates a new binary tree, a sub-tree with a new root on which this process can be reiterated. If $x$ or $y$ contain no terms, then their respective sub-tree is complete. Otherwise, redefine $x$ and $y$ as their own permutations $p$, and repeat the process until both $x$ and $y$ contain no terms. 
\\

According to the construction defined in section $\ref{p2b}$, a lemma immediately follows.

\begin{lem}\label{map}
Given any 213-avoiding permutation of $n$, there exists exactly one corresponding binary tree with $n+1$ leaves.
\end{lem}

\begin{Ex}\label{51234}
Using the procedure described above, and the coordinate system defined in Section $\ref{ss: coordinate system}$ the permutation $p=(5,1,2,3,4)$ can be converted into the following binary tree.

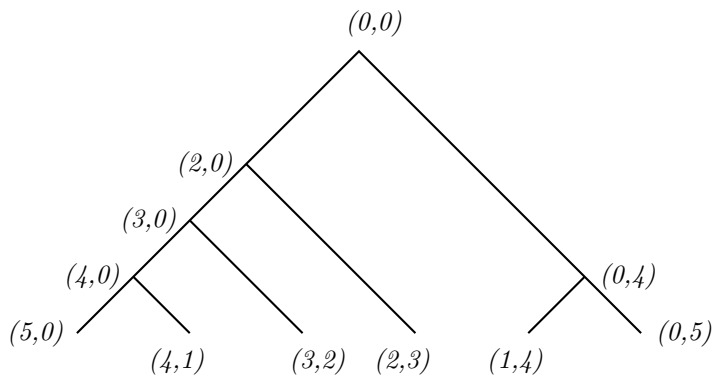
\begin{center}
    \begin{tikzpicture}[scale=0.75]
    \node[text width=1em] at (0, 0.5) {(0,0)};
    \node[text width=1em] at (-6, -5) {(5,0)};
    \node[text width=1em] at (5.5, -5) {(0,5)};
    \draw [thick] (-5,-5) -- (0,0) -- (5,-5);
    \node[text width=1em] at (-3, -2) {(2,0)};
    \draw [thick] (-2,-2) -- (1,-5);
    \node[text width=1em] at (-4, -3) {(3,0)};
    \node[text width=1em] at (4.5, -4) {(0,4)};
    \draw [thick] (-3,-3) -- (-1,-5);
    \node[text width=1em] at (-5, -4) {(4,0)};
    \draw [thick] (-4,-4) -- (-3,-5);
    \draw [thick] (4,-4) -- (3,-5);
    
    \node[text width=1em] at (-3.5, -5.5) {(4,1)};
    
    \node[text width=1em] at (-1, -5.5) {(3,2)};
    
    \node[text width=1em] at (0.5, -5.5) {(2,3)};
    
    \node[text width=1em] at (2.5, -5.5) {(1,4)};
    
    \end{tikzpicture}
\end{center}

\begin{enumerate}
    \item
    Given $p=(5,1,2,3,4)$, an ascending line from $(5, 0)$ to $(0, 0)$ is formed, and a descending line from $(0, 0)$ to $(0, 5)$ is formed. The smallest term in $p$ is $p_2=1$. Thus the sequence of preceding terms $y = (5)$ contains $1$ terms, and the sequence of following terms $x = (2, 3, 4)$ contains $3$ terms. An ascending line from $(1, 4)$ to $(0, 4)$ is formed, and a descending line from $(2, 0)$ to $(2, 3)$ is formed. Reiterate $x$ and $y$ at their corresponding origins, which are $(2, 0)$ and $(0, 4)$ respectively.
    \item
    Given $p = (5)$, the ascending and descending lines overlap with the existing drawings, and given $x$ and $y$ in this case are all empty, the construction for the current $p$ stops.
    \item
    Given $p = (2, 3, 4)$, the ascending and descending lines overlap with the existing drawings. The smallest term is this case is $p_1=2$. Hence, $y$ is empty and $x = (3, 4)$ which contains $2$ terms. A descending line from $(1, 0)$ to $(1, 2)$ is formed. This line is adjusted to its origin $(2, 0)$ and a descending line from $(3, 0)$ to $(3, 2)$ in the original drawing is obtained.
    \item
    The case for $p = (3, 4)$ follows similarly from step $(3)$, and $p = (4)$ follows similarly from step $(2)$.
    
\end{enumerate}

\end{Ex}

\smallskip

\smallskip

\subsection{The "Baseball" Construction: Converting Binary Trees to 213-avoiding Permutations}
\hfill\\
Here we describe a construction which can be used to convert binary trees into 213-avoiding permutations. In a binary tree with $n$ leaves, a triangle of circles with $n-2$ circles at the base is placed within the binary tree. Circles that lie immediately above a descending line are called baseballs, while circles that do not are called cross-balls. The wire diagram is given by connecting lines from right to left through the baseballs and cross-balls.
\begin{Ex}
Below is the wire-diagram for a ``cross-ball''.
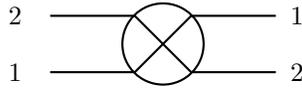
\smallskip
\begin{center}
\begin{tikzpicture}[scale=0.75]
\cball{0}{0}{1}
\draw[thick] (0.5,0.5) -- (2,0.5);
\draw[thick] (0.5,-0.5) -- (2, -0.5);
\draw[thick] (-0.5,0.5) -- (-2,0.5);
\draw[thick] (-0.5,-0.5) -- (-2,-0.5);
\node[text width=1em] at (2.5, 0.5) {$1$};
\node[text width=1em] at (2.5, -0.5) {$2$};
\node[text width=1em] at (-2.5, 0.5) {$2$};
\node[text width=1em] at (-2.5, -0.5) {$1$};
\end{tikzpicture}
\end{center}
\end{Ex}

\begin{Ex}
Below is the wire-diagram for a ``baseball''.
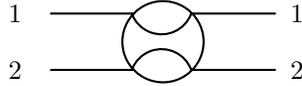
\begin{center}
\begin{tikzpicture}[scale=0.75]
\bball{0}{0}{1}
\draw[thick] (0.5,0.5) -- (2,0.5);
\draw[thick] (0.5,-0.5) -- (2, -0.5);
\draw[thick] (-0.5,0.5) -- (-2,0.5);
\draw[thick] (-0.5,-0.5) -- (-2,-0.5);
\node[text width=1em] at (2.5, 0.5) {$1$};
\node[text width=1em] at (2.5, -0.5) {$2$};
\node[text width=1em] at (-2.5, 0.5) {$1$};
\node[text width=1em] at (-2.5, -0.5) {$2$};
\end{tikzpicture}
\end{center}
\end{Ex}

\begin{Ex}\label{bballcball}
Wire diagrams for two simple cases.
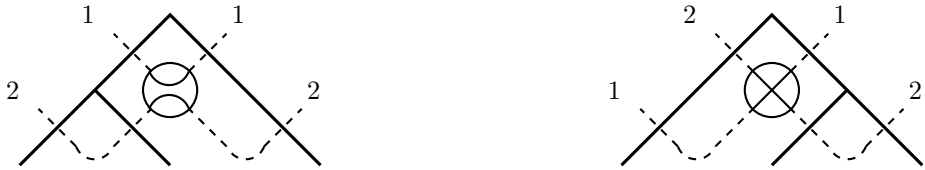
\begin{center}
\begin{tikzpicture}
\begin{scope}[xshift=-4cm, yshift=1cm]
\bball{0}{-1}{0.5}{black}{black}{1}
\node[text width=1em] at (1, 0) {$1$};
\node[text width=1em] at (2, -1) {$2$};
\node[text width=1em] at (-1, 0) {$1$};
\node[text width=1em] at (-2, -1) {$2$};
\draw[very thick] (-2, -2) -- (0,0) -- (2, -2);
\draw[very thick] (-1,-1) -- (0, -2);
\draw[thick, dashed] (0.25, -0.75) -- (0.75, -0.25);
\draw[thick, dashed] (1.25, -1.75) -- (1.75, -1.25);
\draw[thick, dashed] (-0.75, -1.75) -- (-0.25, -1.25);
\draw[thick, dashed] (-0.25, -0.75) -- (-0.75, -0.25);
\draw[thick, dashed] (-1.25, -1.75) -- (-1.75, -1.25);
\draw[thick, dashed] (0.75, -1.75) -- (0.25, -1.25);

\draw[thick, dashed] (-1.25, -1.75) arc (-160: -20:0.26);
\draw[thick, dashed] (1.25, -1.75) arc (-20: -160:0.26);

\end{scope}
\begin{scope}[xshift=4cm, yshift=1cm]
\cball{0}{-1}{0.5}{black}{black}
\node[text width=1em] at (1, 0) {$1$};
\node[text width=1em] at (2, -1) {$2$};
\node[text width=1em] at (-1, 0) {$2$};
\node[text width=1em] at (-2, -1) {$1$};
\draw[very thick] (-2, -2) -- (0,0) -- (2, -2);
\draw[very thick] (1,-1) -- (0, -2);
\draw[thick, dashed] (0.25, -0.75) -- (0.75, -0.25);
\draw[thick, dashed] (1.25, -1.75) -- (1.75, -1.25);
\draw[thick, dashed] (-0.75, -1.75) -- (-0.25, -1.25);
\draw[thick, dashed] (-0.25, -0.75) -- (-0.75, -0.25);
\draw[thick, dashed] (-1.25, -1.75) -- (-1.75, -1.25);
\draw[thick, dashed] (0.75, -1.75) -- (0.25, -1.25);

\draw[thick, dashed] (-1.25, -1.75) arc (-160: -20:0.26);
\draw[thick, dashed] (1.25, -1.75) arc (-20: -160:0.26);
\end{scope}
\end{tikzpicture}
\end{center}
The wire diagram on the left corresponds to the permutation $(1, 2)$, and the wire diagram on the right corresponds to the permutation $(2, 1)$. In either case, the permutation is $213$-avoiding.
\end{Ex}

\begin{Ex} Below, a binary tree constructed through the procedure described in subsection \ref{p2b} from the 213-avoiding permutation $p=(1,3,4,2)$ is shown. Through the binary tree, a wire diagram formed using the baseball construction demonstrates that the correct permutation is obtained.
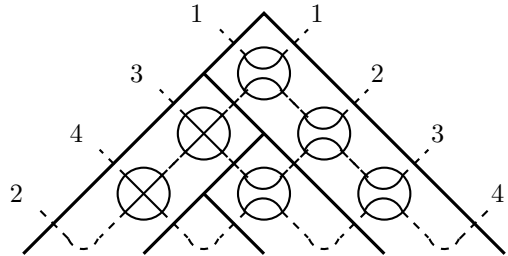
\begin{center}
\begin{tikzpicture}[scale=0.8]
\draw[very thick] (-4, -4) -- (0,0) -- (4, -4);
\draw[very thick] (-1, -1) -- (2, -4);
\draw[very thick] (0, -2) -- (-2, -4);
\draw[very thick] (-1, -3) -- (0, -4);
\bball{0}{-1}{0.6}[black][black][1]
\bball{1}{-2}{0.6}[black][black][1]
\bball{2}{-3}{0.6}[black][black][1]
\cball{-1}{-2}{0.6}[black][black][1]
\cball{-2}{-3}{0.6}[black][black][1]
\bball{0}{-3}{0.6}[black][black][1]

\node[text width=1em] at (1, 0) {$1$};
\node[text width=1em] at (2, -1) {$2$};
\node[text width=1em] at (3, -2) {$3$};
\node[text width=1em] at (4, -3) {$4$};
\node[text width=1em] at (-1, 0) {$1$};
\node[text width=1em] at (-2, -1) {$3$};
\node[text width=1em] at (-3, -2) {$4$};
\node[text width=1em] at (-4, -3) {$2$};

\draw[thick, dashed] (3.25, -3.75) -- (3.75, -3.25);
\draw[thick, dashed] (-3.25, -3.75) -- (-3.75, -3.25);

\draw[thick, dashed] (-3.25, -3.75) arc (-160: -20:0.26);
\draw[thick, dashed] (-1.25, -3.75) arc (-160: -20:0.26);
\draw[thick, dashed] (0.75, -3.75) arc (-160: -20:0.26);
\draw[thick, dashed] (2.75, -3.75) arc (-160: -20:0.26);

\end{tikzpicture}
\end{center}
\end{Ex}

\begin{thm}
The number of binary trees with $n + 1$ leaves $B_n$ is equal to the $n$-th Catalan number
$$B_{n} = C_n = \frac{1}{n+1}\binom{2n}{n}$$
\end{thm}

\begin{thm}\label{bnum}
The procedure provided within the proof of Lemma \ref{map} describes a bijective map between 213-avoiding permutations and binary trees.
\end{thm}

\begin{proof}
We start by presenting a generalization of the baseball construction. The figure below functions as a model for the baseball construction representation of every binary tree. This is because the structure of the area between the $x$ and $y$ sub-trees is the same for all binary trees. This follows from the definition of the baseball construction, as baseballs are placed only above descending lines, and there are no descending lines not contained in either $x$ $y$, so all balls outside of $x$ and $y$ that do not rest on the descending branch of $x$ must be cross-balls. The figure demonstrates that for a sequence $s=(1,2,...,k,...,n)$ all terms greater than one and less than or equal to $k$ are sent to the $x$ sub-tree, which sends terms to positions after 1. Additionally, all terms greater than $k$ and less than or equal to $n$ are sent from the $y$ sub-tree to positions before 1. Therefore, for $p$ to be a 213-avoiding permutation, it is sufficient to prove that $x$ and $y$ are 213-avoiding permutations. This can be done by recognizing that $x$ and $y$ themselves can be modeled by the generalized baseball construction representation, but will have fewer terms than $p$. So, the process can be reiterated until $x$ and $y$ contain only two terms, which as demonstrated in Example \ref{bballcball}, are always 213-avoiding. Therefore, for any binary tree, there exists a 213-avoiding permutation that maps to it, making the map surjective. In addition, based on Theorem $\ref{pnum}$ and Theorem $\ref{bnum}$, there are the same number of 213-avoiding permutations with $n$ terms as binary trees with $n+1$ leaves, so the map is bijective.

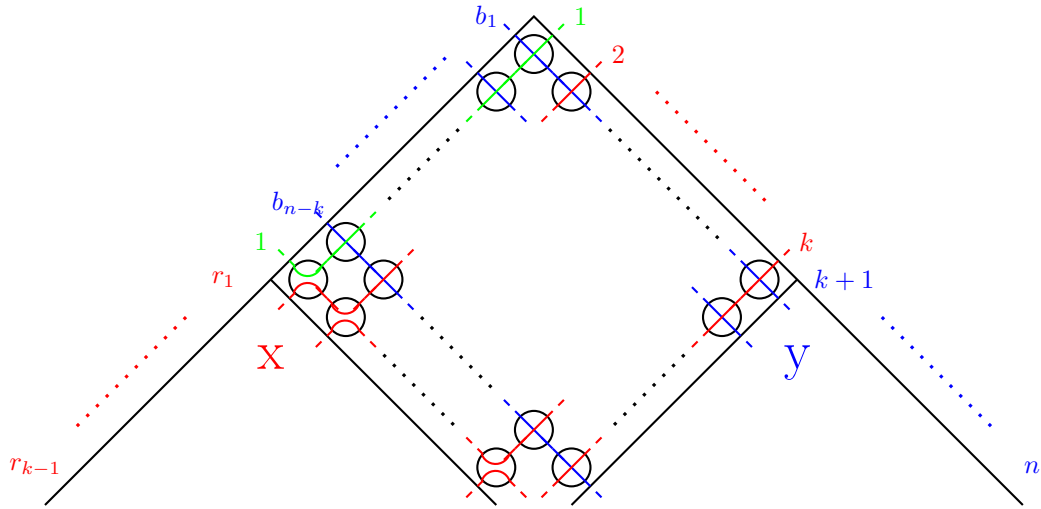
\begin{center}
\begin{tikzpicture}[scale=0.5]
\draw [thick] (-13,-13) -- (0,0) -- (13,-13);
\draw [thick] (7,-7) -- (1,-13);
\draw [thick] (-7,-7) -- (-1,-13);
\node [below, red] at (-7,-8.5) {\Huge x};
\node [below, blue] at (7,-8.5) {\Huge y};
% top three cross balls
\cball{0}{-1}{0.7}[blue][green][1]
\cball{1}{-2}{0.7}[blue][red][1]
\cball{-1}{-2}{0.7}[blue][green][1]
% baseballs on top of X
\bball{-6}{-7}{0.7}[green][red][1]
\bball{-5}{-8}{0.7}[red][red][1]
\bball{-1}{-12}{0.7}[red][red][1]
% crossballs on top of baseballs
\cball{-5}{-6}{0.7}[blue][green][1]
\cball{-4}{-7}{0.7}[blue][red][1]
\cball{0}{-11}{0.7}[blue][red][1]
% crossballs on top of Y
\cball{6}{-7}{0.7}[blue][red][1]
\cball{5}{-8}{0.7}[blue][red][1]
\cball{1}{-12}{0.7}[blue][red][1]
\draw [loosely dotted, very thick] (-3,-8) -- (-1,-10);
\draw [loosely dotted, very thick] (-4,-9) -- (-2,-11);
\draw [loosely dotted, very thick] (4,-9) -- (2,-11);
\draw [loosely dotted, very thick] (-2,-3) -- (-4,-5);
\draw [loosely dotted, very thick] (2,-3) -- (5,-6);

% green lines & numbers
% \draw [thick, green] (0.75, -0.25) -- (-1.75, -2.75);
\node[text=green] at (1.25, 0) {$1$};
\node[text=green] at (-7.25, -6) {$1$};

% red lines & numbers
% \draw [thick, red] (1.75, -1.25) -- (-0.75, -3.75);
\node[text=red] at (2.25, -1) {$2$};
% \draw [thick, red] (6.75, -6.25) -- (4.25, -8.75);
\node[text=red] at (7.25, -6) {$k$};
\draw [loosely dotted, very thick, red] (3.25,-2) -- (6.25,-5);
\node[text=red] at (-8.25, -7) {$r_1$};
\node[text=red] at (-13.25, -12) {$r_{k-1}$};
\draw [loosely dotted, very thick, red] (-9.25,-8) -- (-12.25,-11);

%blue lines & numbers
\node[text=blue] at (8.25, -7) {$k + 1$};
\draw [loosely dotted, very thick, blue] (9.25,-8) -- (12.25,-11);
\node[text=blue] at (13.25, -12) {$n$};
\node[text=blue] at (-6.25, -5) {$b_{n-k}$};
\node[text=blue] at (-1.25, 0) {$b_1$};
\draw [loosely dotted, very thick, blue] (-5.25,-4) -- (-2.25,-1);
\end{tikzpicture}
\end{center}
\end{proof}

\begin{Ex}\label{51234tree}
The 213-avoiding permutation $p=(5,1,2,3,4)$ corresponds to the binary tree shown in example \ref{51234}. The wire diagram below demonstrates that this binary tree can be converted to obtain the correct permutation. The wire corresponding to $1$ which is drawn in green connects to the position directly above $x$. The wires corresponding to $2$ through $k$, which is in this case equal to $4$, drawn in red, connect to the positions within $x$. And, the wire corresponding to $k+1$ through $n$, in this case both are equal to $5$, connects to a position above the position to which $1$ is connected.
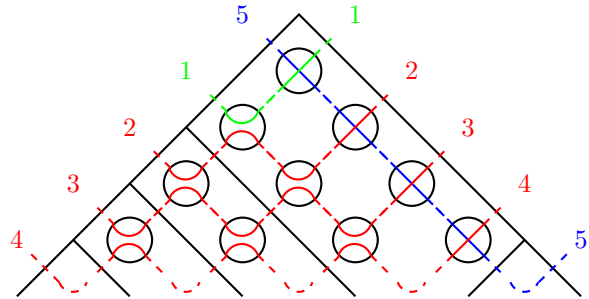
\begin{center}
    \begin{tikzpicture}[scale=0.75]
    \draw [thick] (-5,-5) -- (0,0) -- (5,-5);
    \draw [thick] (-2,-2) -- (1,-5);
    \draw [thick] (-3,-3) -- (-1,-5);
    \draw [thick] (-4,-4) -- (-3,-5);
    \draw [thick] (4,-4) -- (3,-5);
    \cball{0}{-1}{0.55}[blue][green][1]
    \cball{1}{-2}{0.55}[blue][red][1]
    \cball{2}{-3}{0.55}[blue][red][1]
    \cball{3}{-4}{0.55}[blue][red][1]
    \bball{-1}{-2}{0.55}[green][red][1]
    \bball{0}{-3}{0.55}[red][red][1]
    \bball{1}{-4}{0.55}[red][red][1]
    \bball{-2}{-3}{0.55}[red][red][1]
    \bball{-1}{-4}{0.55}[red][red][1]
    \bball{-3}{-4}{0.55}[red][red][1]
    
    \node[text=green] at (1, 0) {$1$};
    \node[text=red] at (2, -1) {$2$};
    \node[text=red] at (3, -2) {$3$};
    \node[text=red] at (4, -3) {$4$};
    \node[text=blue] at (5, -4) {$5$};
    
    \node[text=blue] at (-1, 0) {$5$};
    \node[text=green] at (-2, -1) {$1$};
    \node[text=red] at (-3, -2) {$2$};
    \node[text=red] at (-4, -3) {$3$};
    \node[text=red] at (-5, -4) {$4$};
    
    \draw[thick, blue, dashed] (4.25, -4.75) -- (4.75, -4.25);
    \draw[thick, red, dashed] (-4.25, -4.75) -- (-4.75, -4.25);
    
    \draw[thick, red, dashed] (-4.25, -4.75) arc (-160: -20:0.26);
    \draw[thick, red, dashed] (-2.25, -4.75) arc (-160: -20:0.26);
    \draw[thick, red, dashed] (-0.25, -4.75) arc (-160: -20:0.26);
    \draw[thick, red, dashed] (1.75, -4.75) arc (-160: -20:0.26);
    \draw[thick, blue, dashed] (3.75, -4.75) arc (-160: -20:0.26);
    \end{tikzpicture}
\end{center}
\end{Ex}

\remark{Torsion classes to 213-avoiding permutations:}
Torsion classes correspond to 213-avoiding permutations in the following way. The placement of torsion classes is identical to the placement of baseballs in the baseball construction. That is, both torsion classes and baseballs are to be placed on the descending shelves of the bookshelf construction. To map directly from torsion classes to 213-avoiding permutations, develop a baseball construction from the torsion class with baseballs only on torsion classes and cross-balls in all other locations. Then, use the baseball construction to determine the permutation. The reverse mapping is provided similarly; generate a baseball construction from the 213-avoiding permutation using the process described in this paper, and place torsion classes on exactly the locations occupied by baseballs.
\newline
\subsubsection{Conclusion}
In this paper, we used a coordinate system for binary trees defined in an earlier paper in this volume and used this to derive a rigorous procedure by which 213-avoiding permutations can be converted into binary trees. Then we used a representation called the ``Baseball'' construction to derive a procedure for converting binary trees into $213$-avoiding permutation. These findings all built off the equality between the number of 213-avoiding permutations and the number of binary trees, and hence a bijection between $213$-avoiding permutations and the binary trees is established. An example is given by the permutation $p = (5, 1, 2, 3, 4)$ illustrated in Example $\ref{51234}$ and Example $\ref{51234tree}$. While the existence of the bijection between these objects is well known, the procedures put forth in this paper for the conversions between them are novel.
%\subsection{References}

%\end{document}

%\newpage

%\section{Concluding comments} Something...

%\input{Appendix}

%\input{Example}

%%%%%%%%

\newcommand{\tAL}{{\draw[fill,color=blue!50!white] (2,1)circle[radius=4mm];}}
\newcommand{\tAR}{{\draw[very thick, color=red!50!white] %(3,1)--(2,0)
(2,1)circle[radius=4mm];}}
\newcommand{\tAF}{{\draw[fill,color=red!60!white] %(3,1)--(2,0)
(2,1)circle[radius=4mm];}}
\newcommand{\tBL}{{%\draw[very thick, color=blue] (3,1)--(4,0);
\draw[fill,color=blue!50!white] (4,1)circle[radius=4mm];}}
\newcommand{\tBR}{{\draw[very thick, color=red!50!white] %(5,1)--(4,0)
(4,1)circle[radius=4mm];}}
\newcommand{\tBF}{{\draw[fill,color=red!60!white] %(5,1)--(4,0)
(4,1)circle[radius=4mm];}}
\newcommand{\tCL}{{%\draw[very thick, color=blue] (2,2)--(3,1);
\draw[fill,color=blue!50!white] 
(3,2)circle[radius=4mm];}}
\newcommand{\tCR}{{\draw[very thick, color=red!50!white] %(4,2)--(3,1)
(3,2)circle[radius=4mm];}}
\newcommand{\tCF}{{\draw[fill,color=red!60!white] %(4,2)--(3,1)
(3,2)circle[radius=4mm];}}

%%%%%%%%%%

\newcommand{\bAL}{{\draw[very thick, color=blue!50!white] (1,1)--(2,0);\draw[fill,color=blue!50!white] (2,1)circle[radius=4mm];}}
\newcommand{\bAR}{{\draw[very thick, color=red!50!white] (3,1)--(2,0)
(2,1)circle[radius=4mm];}}
\newcommand{\bAF}{{\draw[fill,color=red!60!white] 
(2,1)circle[radius=4mm];
\draw[very thick, color=red!50!white] (3,1)--(2,0);}}
\newcommand{\bBL}{{\draw[very thick, color=blue] (3,1)--(4,0);
\draw[fill,color=blue!50!white] (4,1)circle[radius=4mm];}}
\newcommand{\bBR}{{\draw[very thick, color=red!50!white] (5,1)--(4,0)(4,1)circle[radius=4mm];}}
\newcommand{\bBF}{{\draw[fill,color=red!60!white] (4,1)circle[radius=4mm];
\draw[very thick, color=red!50!white] (5,1)--(4,0);}}
\newcommand{\bCL}{{\draw[very thick, color=blue!50!white] (2,2)--(3,1);\draw[fill,color=blue!50!white] 
(3,2)circle[radius=4mm];}}
\newcommand{\bCR}{{\draw[very thick, color=red!50!white] (4,2)--(3,1)(3,2)circle[radius=4mm];}}
\newcommand{\bCF}{{\draw[very thick, color=red!50!white] (4,2)--(3,1); \draw[fill,color=red!60!white] (3,2)circle[radius=4mm];}}

%
%%%%%%%%%%%%%%%%%
%%%%%%%%%%%%%%%%%
%%%%%%%%%%%%%%%%%

%\section{Binary trees}
%%%%%%%%%%%%%%%%%%%%%%%%%%%
\newpage

\newcommand{\Soneb}{{\draw[thick,color=blue!50!white] (2,0)--(1,1);}}
\newcommand{\Sonew}{{\draw[thick,color=red!40!white] (2,0)--(3,1);}}
\newcommand{\Stwob}{{\draw[thick,color=blue!50!white] (4,0)--(3,1);}}
\newcommand{\Stwow}{{\draw[thick,color=red!40!white] (4,0)--(5,1);}}
\newcommand{\Sthreeb}{{\draw[thick,color=blue!50!white] (6,0)--(5,1);}}
\newcommand{\Sthreew}{{\draw[thick,color=red!40!white] (6,0)--(7,1);}}
\newcommand{\Ptwob}{{\draw[thick,color=blue!50!white] (3,1)--(2,2);}}
\newcommand{\Ptwow}{{\draw[thick,color=red!40!white] (3,1)--(4,2);}}
\newcommand{\Itwob}{{\draw[thick,color=blue!50!white] (5,1)--(4,2);}}
\newcommand{\Itwow}{{\draw[thick,color=red!40!white] (5,1)--(6,2);}}
\newcommand{\Pthreeb}{{\draw[thick,color=blue!50!white] (4,2)--(3,3);}}
\newcommand{\Pthreew}{{\draw[thick,color=red!40!white] (4,2)--(5,3);}}

\begin{figure}[htbp]
\begin{center}
\begin{tikzpicture}[scale=.35]
%\draw[help lines=1,thick] (-5,-5) grid (5,4);
%\foreach \x in {-8,-6,...,8}\draw (\x,0) node{\x};\foreach \y in {-6,-4,...,8}\draw (0,\y) node{\y};
\begin{scope}[xshift=4cm,yshift=2cm]
\coordinate (Top) at (0, 29);
\coordinate (X) at (-12, 15);
\coordinate (Y) at (0, 19);
\coordinate (Z) at (13, 19);
\coordinate (XZ) at (-6, 5);
\coordinate (XZp) at (6, 10);
\coordinate (YZ) at (18, 4);
\coordinate (AC) at (6, 0);
\coordinate (Ap) at (-6, -5);
\coordinate (AB) at (-17, -1);
\coordinate (C) at (12, -12);
\coordinate (B) at (0, -15);
\coordinate (A) at (-12, -15);
\coordinate (Bottom) at (0, -25);
\end{scope}

\draw[very thick, color=green!20!white] (A)--(Bottom)--(C)
 (Bottom)--(B)--(AB)--
 (A)--(Ap)--(AC)--(C)
 (Ap)--(XZ)--(X)--(AB)
 (X)--(Top)--(Z)--(XZp)--(AC)
 (C)--(YZ)--(Z)
 (XZ)--(XZp)
 (Top)--(Y)--(B) (Y)--(YZ);
\begin{scope}[yshift=-25cm]% bottom: mod L
\draw[fill,color=white] (4,1.4) ellipse[x radius=3.5cm,y radius=3cm];
\coordinate (T) at (4,4);
\coordinate (L) at (0,0);
\coordinate (R) at (8,0);
\draw[thick] (L)--(T)--(R);
\Soneb
\Stwob
\Sthreeb
\Ptwob
\Itwob
\Pthreeb
\end{scope}
\begin{scope}[yshift=-15cm,xshift=-12cm]% a
\draw[fill,color=white] (4,1.4) ellipse[x radius=3.5cm,y radius=3cm];
\coordinate (T) at (4,4);
\coordinate (L) at (0,0);
\coordinate (R) at (8,0);
\draw[thick] (L)--(T)--(R);
\Sonew
\Stwob
\Sthreeb
\Ptwob
\Itwob
\Pthreeb
\end{scope}
\begin{scope}[yshift=-15cm,xshift=0cm]% b
\draw[fill,color=white] (4,1.4) ellipse[x radius=3.5cm,y radius=3cm];
\coordinate (T) at (4,4);
\coordinate (L) at (0,0);
\coordinate (R) at (8,0);
\draw[thick] (L)--(T)--(R);
\Soneb
\Stwob
\Sthreew
\Ptwob
%\Itwob
%\Pthreeb
\end{scope}
\begin{scope}[yshift=-12cm,xshift=12cm]% c
\draw[fill,color=white] (4,1.4) ellipse[x radius=3.5cm,y radius=3cm];
\coordinate (T) at (4,4);
\coordinate (L) at (0,0);
\coordinate (R) at (8,0);
\draw[thick] (L)--(T)--(R);
\Soneb
\Stwow
\Sthreeb
%\Ptwob
\Itwob
\Pthreeb
\end{scope}
\begin{scope}[yshift=-1cm,xshift=-17cm]% ab
\draw[fill,color=white] (4,1.4) ellipse[x radius=3.5cm,y radius=3cm];
\coordinate (T) at (4,4);
\coordinate (L) at (0,0);
\coordinate (R) at (8,0);
\draw[thick] (L)--(T)--(R);
\Sonew
\Stwob
\Sthreew
\Ptwob
%\Itwob
%\Pthreeb
\end{scope}
\begin{scope}[yshift=-5cm,xshift=-6cm]% a+
\draw[fill,color=white] (4,1.4) ellipse[x radius=3.5cm,y radius=3cm];
\coordinate (T) at (4,4);
\coordinate (L) at (0,0);
\coordinate (R) at (8,0);
\draw[thick] (L)--(T)--(R);
\Sonew
\Stwob
\Sthreeb
\Ptwow
\Itwob
\Pthreeb
\end{scope}
\begin{scope}[yshift=0cm,xshift=6cm]% ac
\draw[fill,color=white] (4,1.4) ellipse[x radius=3.5cm,y radius=3cm];
\coordinate (T) at (4,4);
\coordinate (L) at (0,0);
\coordinate (R) at (8,0);
\draw[thick] (L)--(T)--(R);
\Sonew
\Stwow
\Sthreeb
\Ptwow
\Itwob
\Pthreeb
\end{scope}
\begin{scope}[yshift=5cm,xshift=-6cm]% xz
\draw[fill,color=white] (4,1.4) ellipse[x radius=3.5cm,y radius=3cm];
\coordinate (T) at (4,4);
\coordinate (L) at (0,0);
\coordinate (R) at (8,0);
\draw[thick] (L)--(T)--(R);
\Sonew
\Stwob
\Sthreeb
\Ptwow
\Itwob
\Pthreew
\end{scope}
\begin{scope}[yshift=10cm,xshift=6cm]% xz+
\draw[fill,color=white] (4,1.4) ellipse[x radius=3.5cm,y radius=3cm];
\coordinate (T) at (4,4);
\coordinate (L) at (0,0);
\coordinate (R) at (8,0);
\draw[thick] (L)--(T)--(R);
\Sonew
\Stwow
\Sthreeb
\Ptwow
\Itwob
\Pthreew
\end{scope}
\begin{scope}[yshift=15cm,xshift=-12cm]% x
\draw[fill,color=white] (4,1.4) ellipse[x radius=3.5cm,y radius=3cm];
\coordinate (T) at (4,4);
\coordinate (L) at (0,0);
\coordinate (R) at (8,0);
\draw[thick] (L)--(T)--(R);
\Sonew
\Stwob
\Sthreew
\Ptwow
%\Itwob
\Pthreew
\end{scope}
\begin{scope}[yshift=19cm,xshift=0cm]% y
\draw[fill,color=white] (4,1.4) ellipse[x radius=3.5cm,y radius=3cm];
\coordinate (T) at (4,4);
\coordinate (L) at (0,0);
\coordinate (R) at (8,0);
\draw[thick] (L)--(T)--(R);
\Soneb
\Stwow
\Sthreew
%\Ptwow
\Itwow
%\Pthreew
\end{scope}
\begin{scope}[yshift=19cm,xshift=13cm]% z
\draw[fill,color=white] (4,1.4) ellipse[x radius=3.5cm,y radius=3cm];
\coordinate (T) at (4,4);
\coordinate (L) at (0,0);
\coordinate (R) at (8,0);
\draw[thick] (L)--(T)--(R);
\Sonew
\Stwow
\Sthreeb
\Ptwow
\Itwow
\Pthreew
\end{scope}
\begin{scope}[yshift=4cm,xshift=18cm]% yz
\draw[fill,color=white] (4,1.4) ellipse[x radius=3.5cm,y radius=3cm];
\coordinate (T) at (4,4);
\coordinate (L) at (0,0);
\coordinate (R) at (8,0);
\draw[thick] (L)--(T)--(R);
\Soneb
\Stwow
\Sthreeb
%\Ptwow
\Itwow
%\Pthreew
\end{scope}
\begin{scope}[yshift=29cm,xshift=0cm]% top
\draw[fill,color=white] (4,1.4) ellipse[x radius=3.5cm,y radius=3cm];
\coordinate (T) at (4,4);
\coordinate (L) at (0,0);
\coordinate (R) at (8,0);
\draw[thick] (L)--(T)--(R);
\Sonew
\Stwow
\Sthreew
\Ptwow
\Itwow
\Pthreew
\end{scope}
%
%\coordinate (A) at (0,0);
%\draw[thick,color=blue!50!white] (A) ellipse [x radius=2.8cm,y radius=2.1cm];
\end{tikzpicture}
\caption{Above is a Tamari lattice of order 4 in which each node is a binary tree with five leaves. One can see that the covering relation is given by replacing a descending edge with an ascending edge.}
%\label{Figure99}
\end{center}
\end{figure}
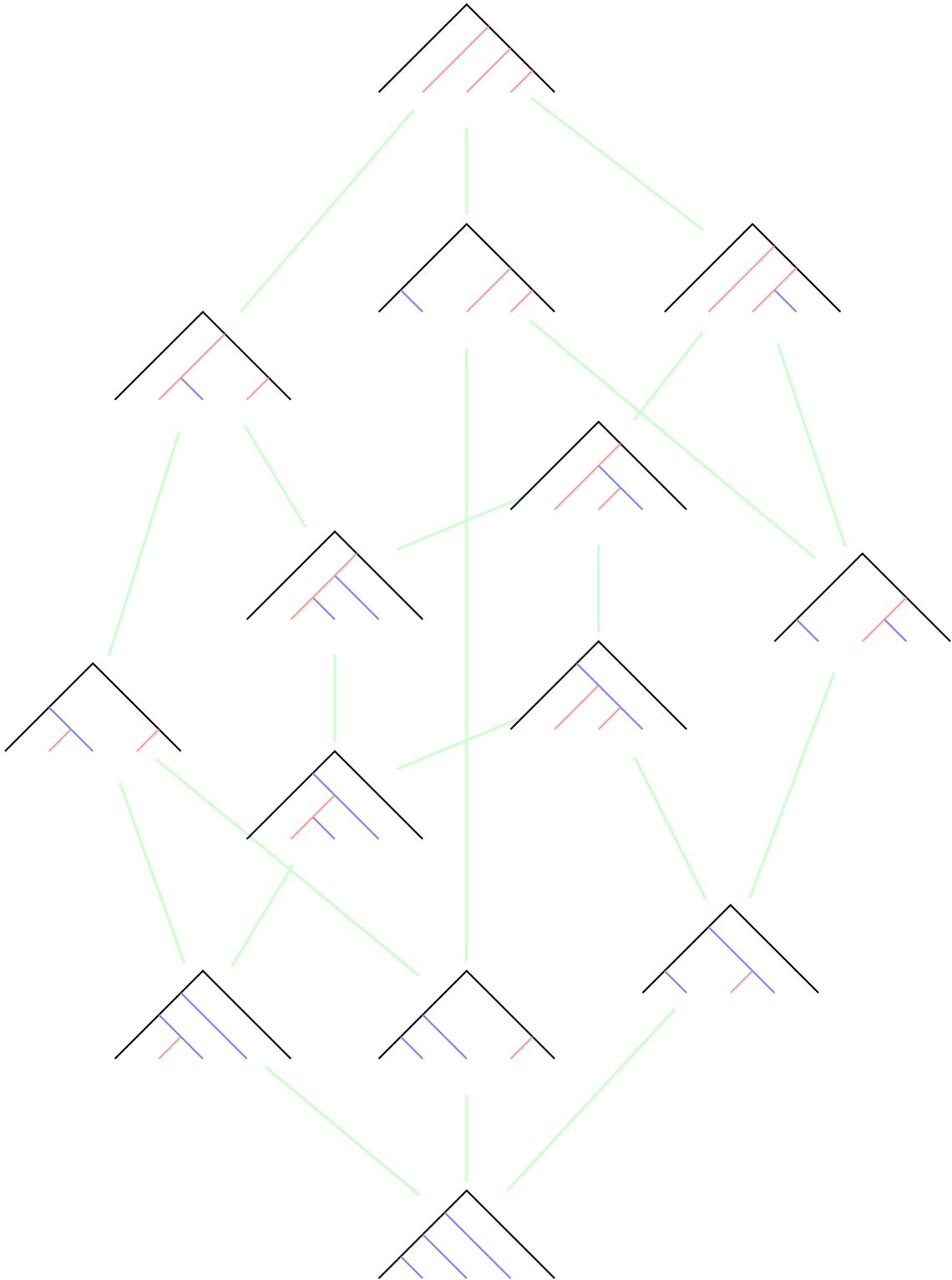
%

%%%%%%%%%%%%%%%%%%%%%%%%%%%

%\section{Torsion classes}
%%%%%%%%%%%%%%%%%%%%%%%%%%%
\newpage

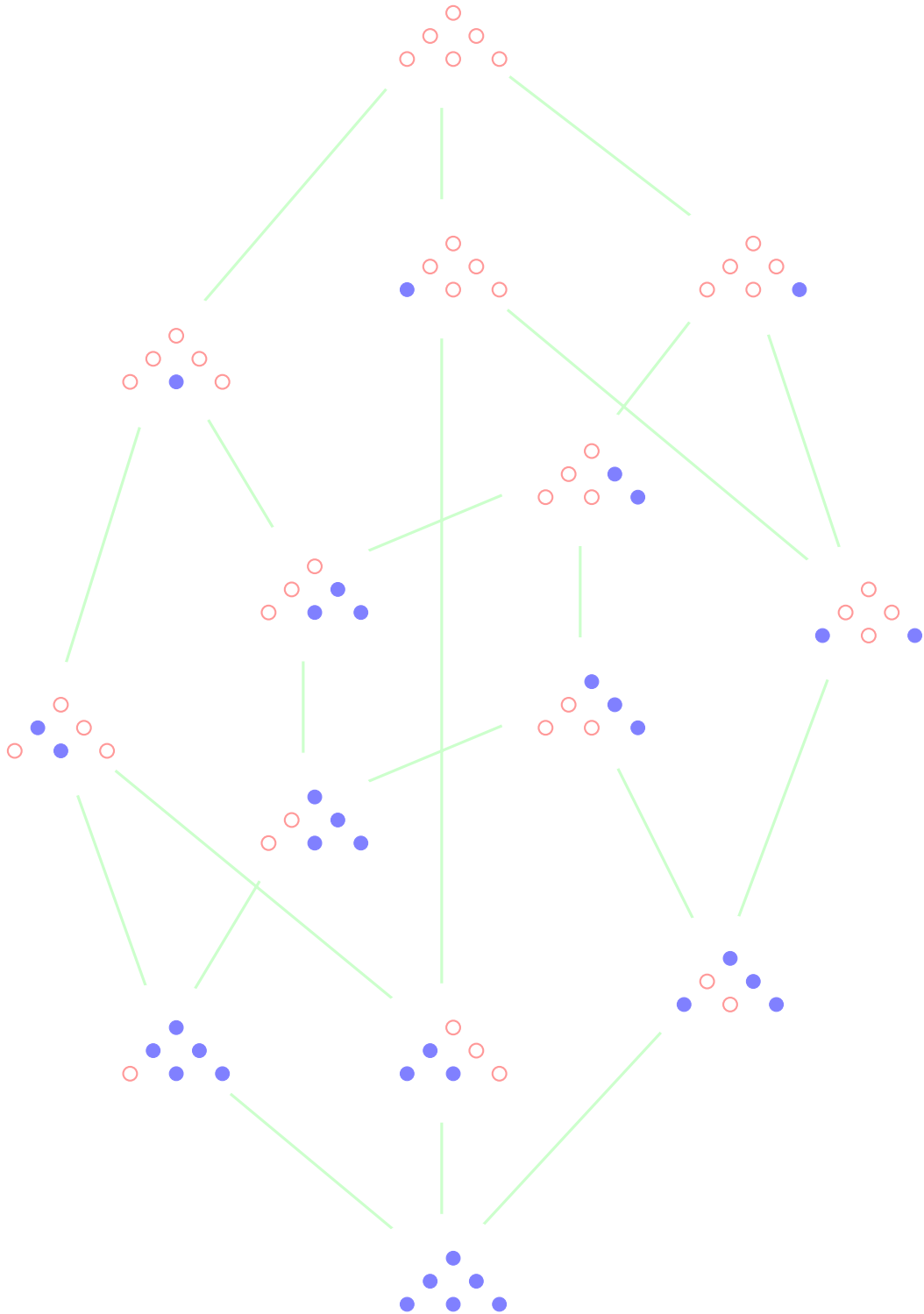
\begin{figure}[htbp]
\begin{center}
\begin{tikzpicture}[scale=.35]
%\draw[help lines=1,thick] (-5,-5) grid (5,4);
%\foreach \x in {-8,-6,...,8}\draw (\x,0) node{\x};\foreach \y in {-6,-4,...,8}\draw (0,\y) node{\y};
\begin{scope}[xshift=4cm,yshift=2cm]
\coordinate (Top) at (0, 29);
\coordinate (X) at (-12, 15);
\coordinate (Y) at (0, 19);
\coordinate (Z) at (13, 19);
\coordinate (XZ) at (-6, 5);
\coordinate (XZp) at (6, 10);
\coordinate (YZ) at (18, 4);
\coordinate (AC) at (6, 0);
\coordinate (Ap) at (-6, -5);
\coordinate (AB) at (-17, -1);
\coordinate (C) at (12, -12);
\coordinate (B) at (0, -15);
\coordinate (A) at (-12, -15);
\coordinate (Bottom) at (0, -25);
\end{scope}

\draw[very thick, color=green!20!white] (A)--(Bottom)--(C)
 (Bottom)--(B)--(AB)--
 (A)--(Ap)--(AC)--(C)
 (Ap)--(XZ)--(X)--(AB)
 (X)--(Top)--(Z)--(XZp)--(AC)
 (C)--(YZ)--(Z)
 (XZ)--(XZp)
 (Top)--(Y)--(B) (Y)--(YZ);
\begin{scope}[yshift=-25cm]% bottom: mod L
%\begin{scope}[rotate=180]
\draw[fill,color=white] (4,1.4) ellipse[x radius=3.5cm,y radius=3cm];
\coordinate (T) at (4,4);
\coordinate (L) at (0,0);
\coordinate (R) at (8,0);
%\draw[thick] (L)--(T)--(R);
\TSoneb
\TStwob
\TSthreeb
\TPtwob
\TItwob
\TPthreeb
%\end{scope}
\end{scope}
\begin{scope}[yshift=-15cm,xshift=-12cm]% a
\draw[fill,color=white] (4,1.4) ellipse[x radius=3.5cm,y radius=3cm];
\coordinate (T) at (4,4);
\coordinate (L) at (0,0);
\coordinate (R) at (8,0);
%\draw[thick] (L)--(T)--(R);
\TSonew
\TStwob
\TSthreeb
\TPtwob
\TItwob
\TPthreeb
\end{scope}
\begin{scope}[yshift=-15cm,xshift=0cm]% b
\draw[fill,color=white] (4,1.4) ellipse[x radius=3.5cm,y radius=3cm];
\coordinate (T) at (4,4);
\coordinate (L) at (0,0);
\coordinate (R) at (8,0);
%\draw[thick] (L)--(T)--(R);
\TSoneb
\TStwob
\TSthreew
\TPtwob
\TItwow%
\TPthreew%
\end{scope}
\begin{scope}[yshift=-12cm,xshift=12cm]% c
\draw[fill,color=white] (4,1.4) ellipse[x radius=3.5cm,y radius=3cm];
\coordinate (T) at (4,4);
\coordinate (L) at (0,0);
\coordinate (R) at (8,0);
%\draw[thick] (L)--(T)--(R);
\TSoneb
\TStwow
\TSthreeb
\TPtwow%
\TItwob
\TPthreeb
\end{scope}
\begin{scope}[yshift=-1cm,xshift=-17cm]% ab
\draw[fill,color=white] (4,1.4) ellipse[x radius=3.5cm,y radius=3cm];
\coordinate (T) at (4,4);
\coordinate (L) at (0,0);
\coordinate (R) at (8,0);
%\draw[thick] (L)--(T)--(R);
\TSonew
\TStwob
\TSthreew
\TPtwob
\TItwow%
\TPthreew%
\end{scope}
\begin{scope}[yshift=-5cm,xshift=-6cm]% a+
\draw[fill,color=white] (4,1.4) ellipse[x radius=3.5cm,y radius=3cm];
\coordinate (T) at (4,4);
\coordinate (L) at (0,0);
\coordinate (R) at (8,0);
%\draw[thick] (L)--(T)--(R);
\TSonew
\TStwob
\TSthreeb
\TPtwow
\TItwob
\TPthreeb
\end{scope}
\begin{scope}[yshift=0cm,xshift=6cm]% ac
\draw[fill,color=white] (4,1.4) ellipse[x radius=3.5cm,y radius=3cm];
\coordinate (T) at (4,4);
\coordinate (L) at (0,0);
\coordinate (R) at (8,0);
%\draw[thick] (L)--(T)--(R);
\TSonew
\TStwow
\TSthreeb
\TPtwow
\TItwob
\TPthreeb
\end{scope}
\begin{scope}[yshift=5cm,xshift=-6cm]% xz
\draw[fill,color=white] (4,1.4) ellipse[x radius=3.5cm,y radius=3cm];
\coordinate (T) at (4,4);
\coordinate (L) at (0,0);
\coordinate (R) at (8,0);
%\draw[thick] (L)--(T)--(R);
\TSonew
\TStwob
\TSthreeb
\TPtwow
\TItwob
\TPthreew
\end{scope}
\begin{scope}[yshift=10cm,xshift=6cm]% xz+
\draw[fill,color=white] (4,1.4) ellipse[x radius=3.5cm,y radius=3cm];
\coordinate (T) at (4,4);
\coordinate (L) at (0,0);
\coordinate (R) at (8,0);
%\draw[thick] (L)--(T)--(R);
\TSonew
\TStwow
\TSthreeb
\TPtwow
\TItwob
\TPthreew
\end{scope}
\begin{scope}[yshift=15cm,xshift=-12cm]% x
\draw[fill,color=white] (4,1.4) ellipse[x radius=3.5cm,y radius=3cm];
\coordinate (T) at (4,4);
\coordinate (L) at (0,0);
\coordinate (R) at (8,0);
%\draw[thick] (L)--(T)--(R);
\TSonew
\TStwob
\TSthreew
\TPtwow
\TItwow%
\TPthreew
\end{scope}
\begin{scope}[yshift=19cm,xshift=0cm]% y
\draw[fill,color=white] (4,1.4) ellipse[x radius=3.5cm,y radius=3cm];
\coordinate (T) at (4,4);
\coordinate (L) at (0,0);
\coordinate (R) at (8,0);
%\draw[thick] (L)--(T)--(R);
\TSoneb
\TStwow
\TSthreew
\TPtwow%
\TItwow
\TPthreew%
\end{scope}
\begin{scope}[yshift=19cm,xshift=13cm]% z
\draw[fill,color=white] (4,1.4) ellipse[x radius=3.5cm,y radius=3cm];
\coordinate (T) at (4,4);
\coordinate (L) at (0,0);
\coordinate (R) at (8,0);
%\draw[thick] (L)--(T)--(R);
\TSonew
\TStwow
\TSthreeb
\TPtwow
\TItwow
\TPthreew
\end{scope}
\begin{scope}[yshift=4cm,xshift=18cm]% yz
\draw[fill,color=white] (4,1.4) ellipse[x radius=3.5cm,y radius=3cm];
\coordinate (T) at (4,4);
\coordinate (L) at (0,0);
\coordinate (R) at (8,0);
%\draw[thick] (L)--(T)--(R);
\TSoneb
\TStwow
\TSthreeb
\TPtwow
\TItwow
\TPthreew
\end{scope}
\begin{scope}[yshift=29cm,xshift=0cm]% top
\draw[fill,color=white] (4,1.4) ellipse[x radius=3.5cm,y radius=3cm];
\coordinate (T) at (4,4);
\coordinate (L) at (0,0);
\coordinate (R) at (8,0);
%\draw[thick] (L)--(T)--(R);
\TSonew
\TStwow
\TSthreew
\TPtwow
\TItwow
\TPthreew
\end{scope}
%
%\coordinate (A) at (0,0);
%\draw[thick,color=blue!50!white] (A) ellipse [x radius=2.8cm,y radius=2.1cm];
\end{tikzpicture}
\caption{Above is the the Tamari lattice in which each node indicates a torsion class (in blue) within the category of the representation of the quiver of type $A_3$. The torsion classes occupy the points that are supported by the descending lines of the binary trees in the previous figure.}
%\label{Figure99}
\end{center}
\end{figure}

\newpage
%%%%%%%%%%%%%%%%%%%%%%%%%%%

%\section{Permutations}

\newcommand{\perBottom}{{\draw[ thick, color=blue] (2.5,.5) arc[start angle=-135,end angle=-40, radius=6.5mm]
(4.5,.5) arc[start angle=-135,end angle=-40, radius=6.5mm]
(2.5,2.5) arc[start angle=40,end angle=100, radius=6.5mm]--(1.5,2.7)
(3.5,3.5) arc[start angle=40,end angle=100, radius=6.5mm]--(1.5,3.7)
(4.5,3.5) arc[start angle=135,end angle=80, radius=6.5mm]--(6.5,3.7)
(5.5,2.5) arc[start angle=135,end angle=80, radius=6.5mm]--(6.5,2.7)
;}}%(2,0)--(1,1);}}

\newcommand{\perSoneb}{{\draw[ thick, color=blue] %(2,0)--(1,1)--(2,2)--(3,1)--cycle
(1.5,1.5) arc[start angle=-135,end angle=-40, radius=6.5mm]
(1.5,.5) arc[start angle=135,end angle=40, radius=6.5mm]
;}}%(2,0)--(1,1);}}
\newcommand{\perSonew}{{\draw[thick,color=blue] %(2,0)--(1,1)--(2,2)--(3,1)--cycle
(1.5,1.5)--(2.5,.5)
(1.5,.5)--(2.5,1.5)
;}}%(2,0)--(3,1);}}
\newcommand{\perStwob}{{\draw[thick,color=blue] 
%(4,0)--(3,1)--(4,2)--(5,1)--cycle
(3.5,1.5) arc[start angle=-135,end angle=-40, radius=6.5mm]
(3.5,.5) arc[start angle=135,end angle=40, radius=6.5mm]
;}}%(4,0)--(3,1);}}
\newcommand{\perStwow}{{\draw[thick,color=blue]  
%(4,0)--(3,1)--(4,2)--(5,1)--cycle
(3.5,1.5)--(4.5,.5)
(3.5,.5)--(4.5,1.5)
;}}%(4,0)--(5,1);}}
\newcommand{\perSthreeb}{{\draw[thick,color=blue] 
%(6,0)--(5,1)--(6,2)--(7,1)--cycle
(5.5,1.5) arc[start angle=-135,end angle=-40, radius=6.5mm]
(5.5,.5) arc[start angle=135,end angle=40, radius=6.5mm]
;}}%(6,0)--(5,1);}}
\newcommand{\perSthreew}{{\draw[thick,color=blue]  %(6,0)--(5,1)--(6,2)--(7,1)--cycle
(5.5,1.5)--(6.5,.5)
(5.5,.5)--(6.5,1.5)
;}}%(6,0)--(7,1);}}
\newcommand{\perPtwob}{{\draw[thick,color=blue] 
%(3,1)--(2,2)--(3,3)--(4,2)--cycle
(2.5,2.5) arc[start angle=-135,end angle=-40, radius=6.5mm]
(2.5,1.5) arc[start angle=135,end angle=40, radius=6.5mm]
;}}%(3,1)--(2,2);}}
\newcommand{\perPtwow}{{\draw[thick,color=blue]  
%(3,1)--(2,2)--(3,3)--(4,2)--cycle
(2.5,2.5)--(3.5,1.5)
(2.5,1.5)--(3.5,2.5)
;}}%(3,1)--(4,2);}}
\newcommand{\perItwob}{{\draw[thick,color=blue] 
%(5,1)--(4,2)--(5,3)--(6,2)--cycle
(4.5,2.5) arc[start angle=-135,end angle=-40, radius=6.5mm]
(4.5,1.5) arc[start angle=135,end angle=40, radius=6.5mm]
;}}%(5,1)--(4,2);}}
\newcommand{\perItwow}{{\draw[thick,color=blue]  
%(5,1)--(4,2)--(5,3)--(6,2)--cycle
(4.5,2.5)--(5.5,1.5)
(4.5,1.5)--(5.5,2.5)
;}}%(5,1)--(6,2);}}
\newcommand{\perPthreeb}{{\draw[thick,color=blue] 
%(4,2)--(3,3)--(4,4)--(5,3)--cycle
(3.5,3.5) arc[start angle=-135,end angle=-40, radius=6.5mm]
(3.5,2.5) arc[start angle=135,end angle=40, radius=6.5mm]
;}}%(4,2)--(3,3);}}
\newcommand{\perPthreew}{{\draw[thick,color=blue]  
%(4,2)--(3,3)--(4,4)--(5,3)--cycle
(3.5,3.5)--(4.5,2.5)
(3.5,2.5)--(4.5,3.5)
;}}%(4,2)--(5,3);}}

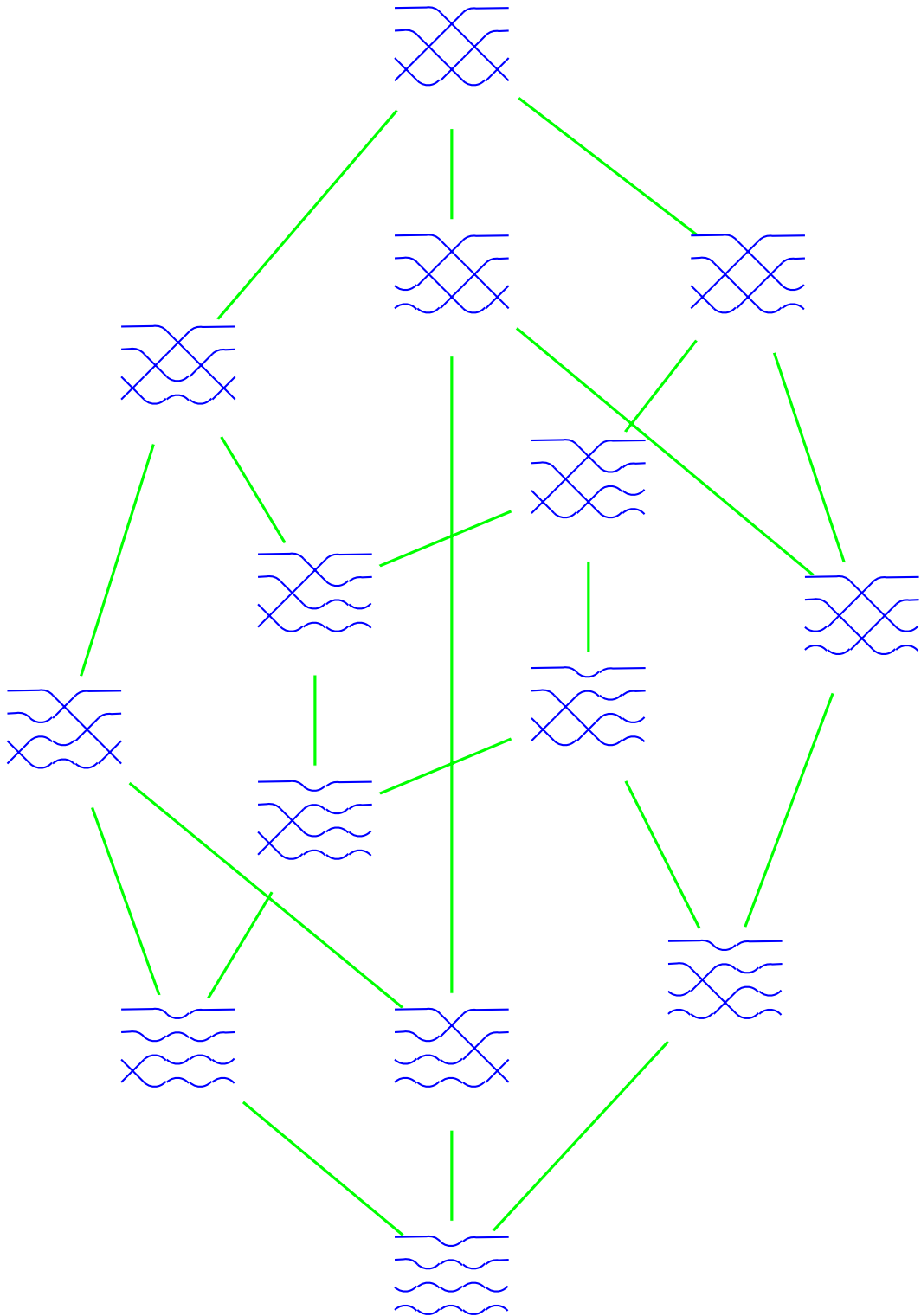
\begin{figure}[htbp]
\begin{center}
\begin{tikzpicture}[scale=.35]
%\draw[help lines=1,thick] (-5,-5) grid (5,4);
%\foreach \x in {-8,-6,...,8}\draw (\x,0) node{\x};\foreach \y in {-6,-4,...,8}\draw (0,\y) node{\y};
\begin{scope}[xshift=4cm,yshift=2cm]
\coordinate (Top) at (0, 29);
\coordinate (X) at (-12, 15);
\coordinate (Y) at (0, 19);
\coordinate (Z) at (13, 19);
\coordinate (XZ) at (-6, 5);
\coordinate (XZp) at (6, 10);
\coordinate (YZ) at (18, 4);
\coordinate (AC) at (6, 0);
\coordinate (Ap) at (-6, -5);
\coordinate (AB) at (-17, -1);
\coordinate (C) at (12, -12);
\coordinate (B) at (0, -15);
\coordinate (A) at (-12, -15);
\coordinate (Bottom) at (0, -25);
\end{scope}

\draw[very thick, color=green] (A)--(Bottom)--(C)
 (Bottom)--(B)--(AB)--
 (A)--(Ap)--(AC)--(C)
 (Ap)--(XZ)--(X)--(AB)
 (X)--(Top)--(Z)--(XZp)--(AC)
 (C)--(YZ)--(Z)
 (XZ)--(XZp)
 (Top)--(Y)--(B) (Y)--(YZ);
\begin{scope}[yshift=-25cm]% bottom: mod L
\draw[fill,color=white] (4,1.4) ellipse[x radius=3.5cm,y radius=3cm];
\coordinate (T) at (4,4);
\coordinate (L) at (0,0);
\coordinate (R) at (8,0);
%\draw[thick] (L)--(T)--(R);
\perSoneb
\perStwob
\perSthreeb
\perPtwob
\perItwob
\perPthreeb
\perBottom
\end{scope}
\begin{scope}[yshift=-15cm,xshift=-12cm]% a
\draw[fill,color=white] (4,1.4) ellipse[x radius=3.5cm,y radius=3cm];
\coordinate (T) at (4,4);
\coordinate (L) at (0,0);
\coordinate (R) at (8,0);
%\draw[thick] (L)--(T)--(R);
\perSonew
\perStwob
\perSthreeb
\perPtwob
\perItwob
\perPthreeb
\perBottom
\end{scope}
\begin{scope}[yshift=-15cm,xshift=0cm]% b
\draw[fill,color=white] (4,1.4) ellipse[x radius=3.5cm,y radius=3cm];
\coordinate (T) at (4,4);
\coordinate (L) at (0,0);
\coordinate (R) at (8,0);
%\draw[thick] (L)--(T)--(R);
\perSoneb
\perStwob
\perSthreew
\perPtwob
\perItwow%
\perPthreew%
\perBottom
\end{scope}
\begin{scope}[yshift=-12cm,xshift=12cm]% c
\draw[fill,color=white] (4,1.4) ellipse[x radius=3.5cm,y radius=3cm];
\coordinate (T) at (4,4);
\coordinate (L) at (0,0);
\coordinate (R) at (8,0);
%\draw[thick] (L)--(T)--(R);
\perSoneb
\perStwow
\perSthreeb
\perPtwow%
\perItwob
\perPthreeb
\perBottom
\end{scope}
\begin{scope}[yshift=-1cm,xshift=-17cm]% ab
\draw[fill,color=white] (4,1.4) ellipse[x radius=3.5cm,y radius=3cm];
\coordinate (T) at (4,4);
\coordinate (L) at (0,0);
\coordinate (R) at (8,0);
%\draw[thick] (L)--(T)--(R);
\perSonew
\perStwob
\perSthreew
\perPtwob
\perItwow%
\perPthreew%
\perBottom
\end{scope}
\begin{scope}[yshift=-5cm,xshift=-6cm]% a+
\draw[fill,color=white] (4,1.4) ellipse[x radius=3.5cm,y radius=3cm];
\coordinate (T) at (4,4);
\coordinate (L) at (0,0);
\coordinate (R) at (8,0);
%\draw[thick] (L)--(T)--(R);
\perSonew
\perStwob
\perSthreeb
\perPtwow
\perItwob
\perPthreeb
\perBottom
\end{scope}
\begin{scope}[yshift=0cm,xshift=6cm]% ac
\draw[fill,color=white] (4,1.4) ellipse[x radius=3.5cm,y radius=3cm];
\coordinate (T) at (4,4);
\coordinate (L) at (0,0);
\coordinate (R) at (8,0);
%\draw[thick] (L)--(T)--(R);
\perSonew
\perStwow
\perSthreeb
\perPtwow
\perItwob
\perPthreeb
\perBottom
\end{scope}
\begin{scope}[yshift=5cm,xshift=-6cm]% xz
\draw[fill,color=white] (4,1.4) ellipse[x radius=3.5cm,y radius=3cm];
\coordinate (T) at (4,4);
\coordinate (L) at (0,0);
\coordinate (R) at (8,0);
%\draw[thick] (L)--(T)--(R);
\perSonew
\perStwob
\perSthreeb
\perPtwow
\perItwob
\perPthreew
\perBottom
\end{scope}
\begin{scope}[yshift=10cm,xshift=6cm]% xz+
\draw[fill,color=white] (4,1.4) ellipse[x radius=3.5cm,y radius=3cm];
\coordinate (T) at (4,4);
\coordinate (L) at (0,0);
\coordinate (R) at (8,0);
%\draw[thick] (L)--(T)--(R);
\perSonew
\perStwow
\perSthreeb
\perPtwow
\perItwob
\perPthreew
\perBottom
\end{scope}
\begin{scope}[yshift=15cm,xshift=-12cm]% x
\draw[fill,color=white] (4,1.4) ellipse[x radius=3.5cm,y radius=3cm];
\coordinate (T) at (4,4);
\coordinate (L) at (0,0);
\coordinate (R) at (8,0);
%\draw[thick] (L)--(T)--(R);
\perSonew
\perStwob
\perSthreew
\perPtwow
\perItwow%
\perPthreew
\perBottom
\end{scope}
\begin{scope}[yshift=19cm,xshift=0cm]% y
\draw[fill,color=white] (4,1.4) ellipse[x radius=3.5cm,y radius=3cm];
\coordinate (T) at (4,4);
\coordinate (L) at (0,0);
\coordinate (R) at (8,0);
%\draw[thick] (L)--(T)--(R);
\perSoneb
\perStwow
\perSthreew
\perPtwow%
\perItwow
\perPthreew%
\perBottom
\end{scope}
\begin{scope}[yshift=19cm,xshift=13cm]% z
\draw[fill,color=white] (4,1.4) ellipse[x radius=3.5cm,y radius=3cm];
\coordinate (T) at (4,4);
\coordinate (L) at (0,0);
\coordinate (R) at (8,0);
%\draw[thick] (L)--(T)--(R);
\perSonew
\perStwow
\perSthreeb
\perPtwow
\perItwow
\perPthreew
\perBottom
\end{scope}
\begin{scope}[yshift=4cm,xshift=18cm]% yz
\draw[fill,color=white] (4,1.4) ellipse[x radius=3.5cm,y radius=3cm];
\coordinate (T) at (4,4);
\coordinate (L) at (0,0);
\coordinate (R) at (8,0);
%\draw[thick] (L)--(T)--(R);
\perSoneb
\perStwow
\perSthreeb
\perPtwow
\perItwow
\perPthreew
\perBottom
\end{scope}
\begin{scope}[yshift=29cm,xshift=0cm]% top
\draw[fill,color=white] (4,1.4) ellipse[x radius=3.5cm,y radius=3cm];
\coordinate (T) at (4,4);
\coordinate (L) at (0,0);
\coordinate (R) at (8,0);
%\draw[thick] (L)--(T)--(R);
\perSonew
\perStwow
\perSthreew
\perPtwow
\perItwow
\perPthreew
\perBottom
\end{scope}
%
%\coordinate (A) at (0,0);
%\draw[thick,color=blue] (A) ellipse [x radius=2.8cm,y radius=2.1cm];
\end{tikzpicture}
\caption{Above is the Tamari lattice in which each node represents a wire diagram corresponding to a 213-avoiding permutation in which the baseballs are placed on the location of the torsion classes of the previous figure.}
%\label{Figure99}
\end{center}
\end{figure}
%

%%%%%%%%%%%%%%%%%%%%%%%%%%%

%\section{Dyck paths}

\newpage

%\section{Young diagrams}

%%%%%%%%%%%%%%%%%%%%%%%%%%%
%
%\newcommand{\YPthreeb}{{\draw[fill,color=blue!50!white,xshift=5mm,yshift=-5mm] (4,3)  circle[radius=.3cm];}}%(4,2)--(3,3);}}

%\newcommand{\YPtwob}{{\draw[fill,color=blue!50!white,xshift=5mm,yshift=-5mm] (3,2)  circle[radius=.3cm];}}%(3,1)--(2,2);}}

%\newcommand{\YSoneb}{{\draw[fill,color=blue!50!white,xshift=5mm,yshift=-5mm] (2,1)  circle[radius=.3cm];}}%(2,0)--(1,1);}}
%
%\newcommand{\YItwob}{{\draw[fill,color=blue!50!white,xshift=5mm,yshift=-5mm] (5,2)  circle[radius=.3cm];}}%(5,1)--(4,2);}}

%\newcommand{\YStwob}{{\draw[fill,color=blue!50!white,xshift=5mm,yshift=-5mm] (4,1)  circle[radius=.3cm];}}%(4,0)--(3,1);}}
%
%\newcommand{\YSthreeb}{{\draw[fill,color=blue!50!white,xshift=5mm,yshift=-5mm] (6,1)  circle[radius=.3cm];}}%(6,0)--(5,1);}}
%
%%%%%%%%%%%%%%%%%%%%%%%%%%
%
%\newcommand{\YSonew}{{\draw[fill,color=white] (2.5,.5)  circle[radius=.4cm];\draw[thick,color=red!40!white] (2.5,.5)  circle[radius=.3cm];}}%(2,0)--(3,1);}}

\newcommand{\YSonef}{{\draw[fill,color=white] (2.5,.5)  circle[radius=.4cm];
\draw[fill,color=red!60!white] (2.5,.5)  circle[radius=.3cm];}}%(2,0)--(3,1);}}

\newcommand{\YStwof}{{\draw[fill,color=white]  (4.5, .5) circle[radius=.4cm];
\draw[fill,color=red!60!white]  (4.5, .5) circle[radius=.3cm];}}%(4,0)--(5,1);}}
\newcommand{\YSthreef}{{\draw[fill,color=white]  (6.5, .5) circle[radius=.4cm];
\draw[fill,color=red!60!white]  (6.5, .5) circle[radius=.3cm];}}%(6,0)--(7,1);}}
\newcommand{\YPtwof}{{\draw[fill,color=white]  (3.5, 1.5) circle[radius=.4cm];
\draw[fill,color=red!60!white]  (3.5, 1.5) circle[radius=.3cm];}}%(3,1)--(4,2);}}
\newcommand{\YItwof}{{\draw[fill,color=white]  (5.5, 1.5) circle[radius=.4cm];
\draw[fill,color=red!60!white]  (5.5, 1.5) circle[radius=.3cm];}}%(5,1)--(6,2);}}
\newcommand{\YPthreef}{{\draw[fill,color=white]  (4.5, 2.5) circle[radius=.4cm];
\draw[fill,color=red!60!white]  (4.5, 2.5) circle[radius=.3cm];}}%(4,2)--(5,3);}}

\begin{figure}[htbp]
\begin{center}
\begin{tikzpicture}[scale=.35]
%\draw[help lines=1,thick] (-5,-5) grid (5,4);
%\foreach \x in {-8,-6,...,8}\draw (\x,0) node{\x};\foreach \y in {-6,-4,...,8}\draw (0,\y) node{\y};
\begin{scope}[xshift=4cm,yshift=2cm]
\coordinate (Top) at (0, 29);
\coordinate (X) at (-12, 15);
\coordinate (Y) at (0, 19);
\coordinate (Z) at (13, 19);
\coordinate (XZ) at (-6, 5);
\coordinate (XZp) at (6, 10);
\coordinate (YZ) at (18, 4);
\coordinate (AC) at (6, 0);
\coordinate (Ap) at (-6, -5);
\coordinate (AB) at (-17, -1);
\coordinate (C) at (12, -12);
\coordinate (B) at (0, -15);
\coordinate (A) at (-12, -15);
\coordinate (Bottom) at (0, -25);
\end{scope}

\draw[very thick, color=green!20!white] (A)--(Bottom)--(C)
 (Bottom)--(B)--(AB)--
 (A)--(Ap)--(AC)--(C)
 (Ap)--(XZ)--(X)--(AB)
 (X)--(Top)--(Z)--(XZp)--(AC)
 (C)--(YZ)--(Z)
 (XZ)--(XZp)
 (Top)--(Y)--(B) (Y)--(YZ);
\begin{scope}[yshift=-25cm]% bottom: mod L
\draw[fill,color=white] (4,1.4) ellipse[x radius=3.5cm,y radius=3cm];
\begin{scope}[xshift=5mm,yshift=-5mm]
\coordinate (T) at (4,4);
\coordinate (L) at (0,0);
\coordinate (R) at (8,0);
\draw[thick] (L)--(T)--(R);
\Soneb
\Stwob
\Sthreeb
\Ptwob
\Itwob
\Pthreeb
\end{scope}
\YSoneb
\YStwob
\YSthreeb
\YPtwob
\YItwob
\YPthreeb
\end{scope}
\begin{scope}[yshift=-15cm,xshift=-12cm]% a
\draw[fill,color=white] (4,1.4) ellipse[x radius=3.5cm,y radius=3cm];
\begin{scope}[xshift=5mm,yshift=-5mm]
\coordinate (T) at (4,4);
\coordinate (L) at (0,0);
\coordinate (R) at (8,0);
\draw[thick] (L)--(T)--(R);
\Sonew
\Stwob
\Sthreeb
\Ptwob
\Itwob
\Pthreeb
\end{scope}
\YSonef
\YStwob
\YSthreeb
\YPtwob
\YItwob
\YPthreeb
\end{scope}
\begin{scope}[yshift=-15cm,xshift=0cm]% b
\draw[fill,color=white] (4,1.4) ellipse[x radius=3.5cm,y radius=3cm];
\begin{scope}[xshift=5mm,yshift=-5mm]
\coordinate (T) at (4,4);
\coordinate (L) at (0,0);
\coordinate (R) at (8,0);
\draw[thick] (L)--(T)--(R);
\Soneb
\Stwob
\Sthreew
\Ptwob
%\Itwob
%\Pthreeb
\end{scope}
\YSoneb
\YStwob
\YSthreef
\YPtwob
\YItwow%
\YPthreew%
\end{scope}
\begin{scope}[yshift=-12cm,xshift=12cm]% c
\draw[fill,color=white] (4,1.4) ellipse[x radius=3.5cm,y radius=3cm];
\begin{scope}[xshift=5mm,yshift=-5mm]
\coordinate (T) at (4,4);
\coordinate (L) at (0,0);
\coordinate (R) at (8,0);
\draw[thick] (L)--(T)--(R);
\Soneb
\Stwow
\Sthreeb
%\Ptwob
\Itwob
\Pthreeb
\end{scope}
\YSoneb
\YStwof
\YSthreeb
\YPtwow%
\YItwob
\YPthreeb
\end{scope}
\begin{scope}[yshift=-1cm,xshift=-17cm]% ab
\draw[fill,color=white] (4,1.4) ellipse[x radius=3.5cm,y radius=3cm];
\begin{scope}[xshift=5mm,yshift=-5mm]
\coordinate (T) at (4,4);
\coordinate (L) at (0,0);
\coordinate (R) at (8,0);
\draw[thick] (L)--(T)--(R);
\Sonew
\Stwob
\Sthreew
\Ptwob
%\Itwob
%\Pthreeb
\end{scope}
\YSonef
\YStwob
\YSthreef
\YPtwob
\YItwow%
\YPthreew%
\end{scope}
\begin{scope}[yshift=-5cm,xshift=-6cm]% a+
\draw[fill,color=white] (4,1.4) ellipse[x radius=3.5cm,y radius=3cm];
\begin{scope}[xshift=5mm,yshift=-5mm]
\coordinate (T) at (4,4);
\coordinate (L) at (0,0);
\coordinate (R) at (8,0);
\draw[thick] (L)--(T)--(R);
\Sonew
\Stwob
\Sthreeb
\Ptwow
\Itwob
\Pthreeb
\end{scope}
\YSonef
\YStwob
\YSthreeb
\YPtwof
\YItwob
\YPthreeb
\end{scope}
\begin{scope}[yshift=0cm,xshift=6cm]% ac
\draw[fill,color=white] (4,1.4) ellipse[x radius=3.5cm,y radius=3cm];
\begin{scope}[xshift=5mm,yshift=-5mm]
\coordinate (T) at (4,4);
\coordinate (L) at (0,0);
\coordinate (R) at (8,0);
\draw[thick] (L)--(T)--(R);
\Sonew
\Stwow
\Sthreeb
\Ptwow
\Itwob
\Pthreeb
\end{scope}
\YSonef
\YStwof
\YSthreeb
\YPtwof
\YItwob
\YPthreeb
\end{scope}
\begin{scope}[yshift=5cm,xshift=-6cm]% xz
\draw[fill,color=white] (4,1.4) ellipse[x radius=3.5cm,y radius=3cm];
\begin{scope}[xshift=5mm,yshift=-5mm]
\coordinate (T) at (4,4);
\coordinate (L) at (0,0);
\coordinate (R) at (8,0);
\draw[thick] (L)--(T)--(R);
\Sonew
\Stwob
\Sthreeb
\Ptwow
\Itwob
\Pthreew
\end{scope}
\YSonef
\YStwob
\YSthreeb
\YPtwof
\YItwob
\YPthreef
\end{scope}
\begin{scope}[yshift=10cm,xshift=6cm]% xz+
\draw[fill,color=white] (4,1.4) ellipse[x radius=3.5cm,y radius=3cm];
\begin{scope}[xshift=5mm,yshift=-5mm]
\coordinate (T) at (4,4);
\coordinate (L) at (0,0);
\coordinate (R) at (8,0);
\draw[thick] (L)--(T)--(R);
\Sonew
\Stwow
\Sthreeb
\Ptwow
\Itwob
\Pthreew
\end{scope}
\YSonef
\YStwof
\YSthreeb
\YPtwof
\YItwob
\YPthreef
\end{scope}
\begin{scope}[yshift=15cm,xshift=-12cm]% x
\draw[fill,color=white] (4,1.4) ellipse[x radius=3.5cm,y radius=3cm];
\begin{scope}[xshift=5mm,yshift=-5mm]
\coordinate (T) at (4,4);
\coordinate (L) at (0,0);
\coordinate (R) at (8,0);
\draw[thick] (L)--(T)--(R);
\Sonew
\Stwob
\Sthreew
\Ptwow
%\Itwob
\Pthreew
\end{scope}
\YSonef
\YStwob
\YSthreef
\YPtwof
\YItwow%
\YPthreef
\end{scope}
\begin{scope}[yshift=19cm,xshift=0cm]% y
\draw[fill,color=white] (4,1.4) ellipse[x radius=3.5cm,y radius=3cm];
\begin{scope}[xshift=5mm,yshift=-5mm]
\coordinate (T) at (4,4);
\coordinate (L) at (0,0);
\coordinate (R) at (8,0);
\draw[thick] (L)--(T)--(R);
\Soneb
\Stwow
\Sthreew
%\Ptwow
\Itwow
%\Pthreew
\end{scope}
\YSoneb
\YStwof
\YSthreef
\YPtwow%
\YItwof
\YPthreew%
\end{scope}
\begin{scope}[yshift=19cm,xshift=13cm]% z
\draw[fill,color=white] (4,1.4) ellipse[x radius=3.5cm,y radius=3cm];
\begin{scope}[xshift=5mm,yshift=-5mm]
\coordinate (T) at (4,4);
\coordinate (L) at (0,0);
\coordinate (R) at (8,0);
\draw[thick] (L)--(T)--(R);
\Sonew
\Stwow
\Sthreeb
\Ptwow
\Itwow
\Pthreew
\end{scope}
\YSonef
\YStwof
\YSthreeb
\YPtwof
\YItwof
\YPthreef
\end{scope}
\begin{scope}[yshift=4cm,xshift=18cm]% yz
\draw[fill,color=white] (4,1.4) ellipse[x radius=3.5cm,y radius=3cm];
\begin{scope}[xshift=5mm,yshift=-5mm]
\coordinate (T) at (4,4);
\coordinate (L) at (0,0);
\coordinate (R) at (8,0);
\draw[thick] (L)--(T)--(R);
\Soneb
\Stwow
\Sthreeb
%\Ptwow
\Itwow
%\Pthreew
\end{scope}
\YSoneb
\YStwof
\YSthreeb
\YPtwow
\YItwof
\YPthreew
\end{scope}
\begin{scope}[yshift=29cm,xshift=0cm]% top
\draw[fill,color=white] (4,1.4) ellipse[x radius=3.5cm,y radius=3cm];
\begin{scope}[xshift=5mm,yshift=-5mm]
\coordinate (T) at (4,4);
\coordinate (L) at (0,0);
\coordinate (R) at (8,0);
\draw[thick] (L)--(T)--(R);
\Sonew
\Stwow
\Sthreew
\Ptwow
\Itwow
\Pthreew
\end{scope}
\YSonef
\YStwof
\YSthreef
\YPtwof
\YItwof
\YPthreef
\end{scope}
%
%\coordinate (A) at (0,0);
%\draw[thick,color=blue!50!white] (A) ellipse [x radius=2.8cm,y radius=2.1cm];
\end{tikzpicture}
\caption{The figure above is a Tamari lattice in which the nodes are labeled with the compliment of the torsion classes (indicated in filled and empty red balls). The filled balls denote the torsion-free class. The compliment of the torsion-free class is the Young diagram with gaps (indicated in blue).}
%\label{Figure99}
\end{center}
\end{figure}
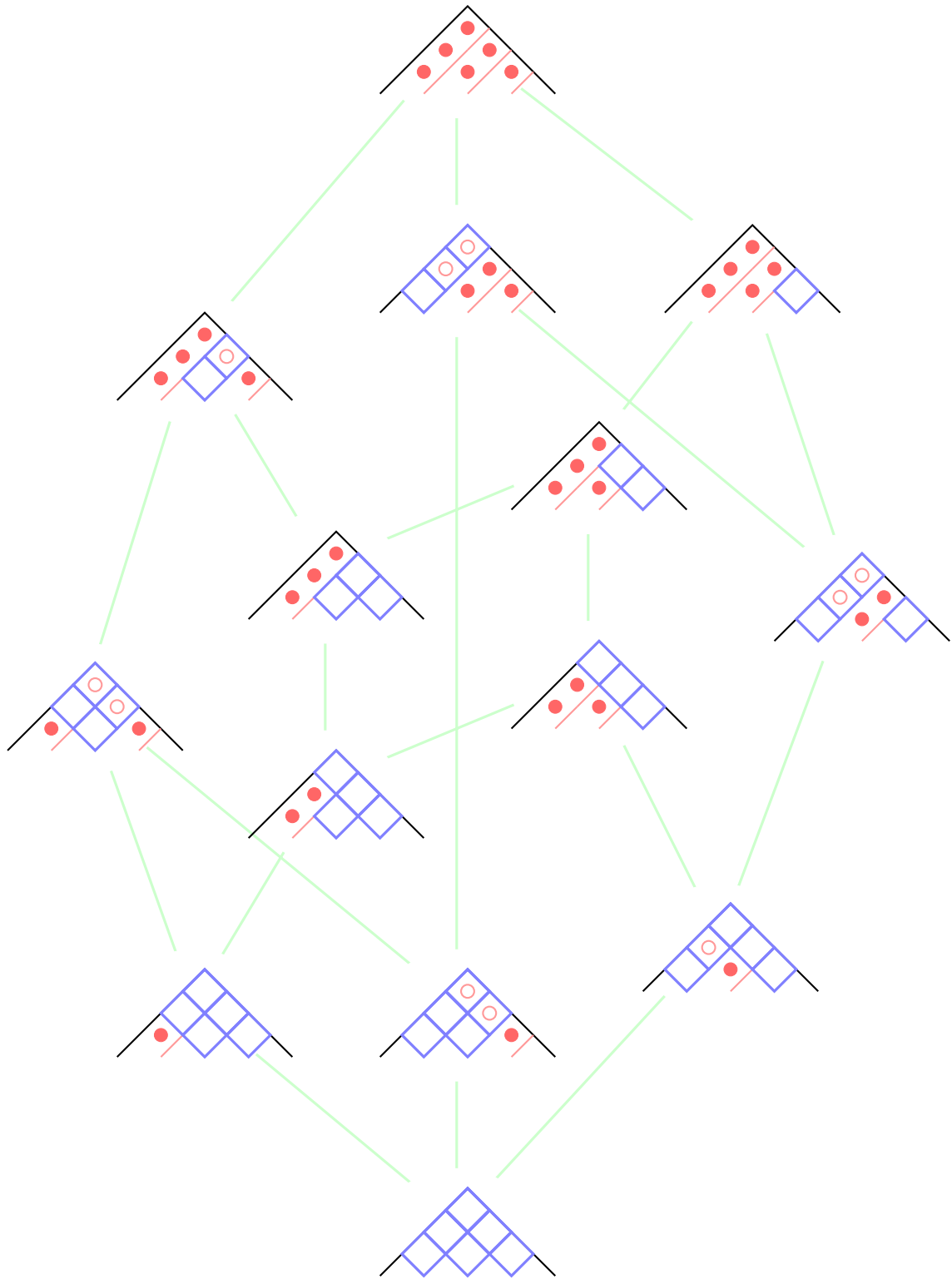
%
%%%%%%%%%%%%%%%%%%%%%%%%%%%

%

\end{document}